\documentclass[reqno,11pt]{amsart}
\usepackage{amsmath,amsthm,amsfonts,amscd,amssymb,comment,eucal,latexsym,mathrsfs}
\usepackage[all]{xy}

\newtheorem{theorem}{Theorem}[section]
\newtheorem{corollary}[theorem]{Corollary}
\newtheorem{lemma}[theorem]{Lemma}
\newtheorem{proposition}[theorem]{Proposition}

\theoremstyle{definition}
\newtheorem{definition}[theorem]{Definition}
\newtheorem{remark}[theorem]{Remark}

\newtheorem{question}[theorem]{Question}

\theoremstyle{plain}

\theoremstyle{definition}

\theoremstyle{remark}

\newcommand{\spn}{{\rm span}}

\newcommand{\Ad}{{\rm Ad}\,}
\newcommand{\diag}{{\rm diag}}
\newcommand{\id}{{\rm id}}

\newcommand{\cB}{{\mathcal B}}

\newcommand{\cC}{{\mathcal C}}
\newcommand{\cU}{{\mathcal U}}

\newcommand{\cQ}{{\mathcal Q}}
\newcommand{\cP}{{\mathcal P}}
\newcommand{\cR}{{\mathcal R}}
\newcommand{\cS}{{\mathcal S}}

\newcommand{\Cb}{{\mathbb C}}

\newcommand{\Zb}{{\mathbb Z}}

\newcommand{\Nb}{{\mathbb N}}

\newcommand{\tr}{{\rm tr}}

\newcommand{\os}{{\boldsymbol{s}}}
\newcommand{\ot}{{\boldsymbol{t}}}

\newcommand{\mn}{{\rm min}}

\newcommand{\sG}{{\mathscr G}}
\newcommand{\sH}{{\mathscr H}}

\newcommand{\sR}{{\mathscr R}}

\newcommand{\sN}{{\mathscr N}}

\newcommand{\sP}{{\mathscr P}}

\newcommand{\sX}{{\mathscr X}}
\newcommand{\sY}{{\mathscr Y}}

\newcommand{\Sym}{{\rm Sym}}

\newcommand{\SA}{{\rm SA}}

\newcommand{\VN}{{\rm VN}}
\newcommand{\sym}{{}}
\newcommand{\slower}{\underline{s}}
\newcommand{\HA}{{\rm HA}}
\newcommand{\GA}{{\rm GA}}

\newcommand{\so}{\mathfrak{s}}
\newcommand{\ra}{\mathfrak{r}}
\newcommand{\domain}{{\rm dom}}

\newcommand{\eps}{\varepsilon}

\allowdisplaybreaks

\begin{document}

\title{Sofic dimension for discrete measured groupoids}

\author{Ken Dykema}
\author{David Kerr}
\author{Mika{\"e}l Pichot}

\address{\hskip-\parindent
Ken Dykema, Department of Mathematics, Texas A{\&}M University,
College Station TX 77843-3368, U.S.A.}
\email{kdykema@math.tamu.edu}

\address{\hskip-\parindent
David Kerr, Department of Mathematics, Texas A{\&}M University,
College Station TX 77843-3368, U.S.A.}
\email{kerr@math.tamu.edu}

\address{\hskip-\parindent
Mika{\"e}l Pichot, Department of Mathematics and Statistics, McGill University,
Montreal, Quebec H3A 2K6, Canada}
\email{mikael.pichot@mcgill.ca}

\thanks{K.D. was partially supported by NSF grant DMS-0901220, D.K. was partially supported by NSF grant DMS-0900938,
and M.P. was partially supported by JSPS}

\begin{abstract}
For discrete measured groupoids preserving a probability measure
we introduce a notion of sofic dimension that measures the
asymptotic growth of the number of sofic approximations on larger and larger
finite sets. In the case of groups we give a formula 
for free products with amalgamation over an amenable subgroup. We also prove 
a free product formula for measure-preserving actions.
\end{abstract}
\date{November 11, 2012}

\maketitle

\section{Introduction}

For certain kinds of infinite-dimensional structures it is possible to define a notion of volume or
complexity by measuring the asymptotic growth of the number of models in finite or finite-dimensional 
spaces of increasing size. This idea occurs prototypically in the statistical mechanics 
of infinite lattice systems, where one defines the mean entropy as a limit 
of weighted averages over finite-volume configurations. 
Via the action of lattice translation, this mean entropy can be recast
as a particular instance of dynamical entropy. For continuous
actions of amenable groups on compact Hausdorff spaces, dynamical entropy can be expressed either 
in information-theoretic terms using open covers or as a measure of the exponential growth 
of the number of partial orbits up to an observational error. 
Kolmogorov-Sinai entropy for measure-preserving actions of amenable groups can also be viewed in a similar dual way.

In a recent breakthrough, Lewis Bowen showed how the statistical mechanical idea 
of counting finitary models can be used as a means for defining dynamical entropy
in the very broad context of measure-preserving actions of countable sofic groups \cite{Bow10}.
A generalization of both amenability and residual finiteness, soficity is defined by the existence of approximate actions
on finite spaces, and it is these approximate actions which provide the setting for dynamical models. 
Hanfeng Li and the second author 
subsequently applied an operator algebra perspective to develop a more general approach to sofic entropy 
that yields both topological and measure-theoretic entropy invariants \cite{KerLi10,Ker12}.

This ``microstates'' approach to dynamical entropy can be compared with the packing formulation of
Voiculescu's free entropy dimension for tracial von Neumann algebras, for which the finite modeling takes 
place in matrix algebras instead of
finite sets or commutative finite-dimensional $C^*$-algebras. While sofic entropy
measures the exponential growth of the number of dynamical models relative to a fixed background sequence
of sofic approximations for the group, free entropy dimension
counts the number of matrix models for a finite set of operators (which might for instance come from both the
group and the space in a crossed product) up to an observational error and measures the growth 
of this quantity within an appropriate superexponential regime as the dimension 
of the matrix algebra tends to infinity. A major open problem concerning free entropy dimension is whether
it takes a common value on all finite generating sets and hence yields an invariant for the von Neumann
algebra. This is true in the hyperfinite case \cite{Jun04} but is unknown for free group factors.
In \cite{Shl03} Shlyakhtenko defined a free-entropy-type quantity using a combination of
permutations and general unitaries that yields an invariant for discrete measured 
equivalence relations.

In the present paper we define a notion of sofic dimension for groups and measure-preserving group 
actions that is based on discrete models in the manner of sofic entropy but counts
all models for the structure in the spirit of free entropy dimension. In fact we set up the theory of sofic 
dimension in the more natural and general framework of discrete measured groupoids
(more precisely, what we call probability-measure-preserving (p.m.p.)\ groupoids), so that 
it simultaneously specializes to groups, measure-preserving group actions, and 
probability-measure-preserving equivalence relations.
This means in particular that, for free measure-preserving actions of countable groups, sofic dimension is
an orbit equivalence invariant. 

The dimension is first defined with respect to several local parameters. One of these parameters determines
the scale at which the sofic approximations are distinguished, while the others determine
how good the sofic approximation is. We take an infimum over the latter and then a supremum
over the former to produce an invariant. 
We show that the value of this invariant can be determined by restricting the parameters to a
generating set, which renders it accessible to computation.
Our main result in the group case gives, under certain regularity assumptions, a formula for the sofic dimension 
of free products with amalgamation over an amenable group, in analogy with those for free entropy dimension
\cite{BroDykJun10} and cost \cite{Gab00}. This gives in particular a 
free probability proof of the fact that soficity for groups is preserved under free products with
amalgamation over an amenable group, which was shown in \cite{ColDyk10} assuming the amenable
group to be monotileable and in \cite{EleSza10,Pau10} in general.
We also establish a free product formula 
for measure-preserving actions under similar regularity assumptions.
In a separate paper devoted to the equivalence relation viewpoint \cite{DykKerPic11} 
we give a formula for the sofic dimension of a free product of equivalence relations amalgamated 
over an amenable subrelation, which applies most notably to free actions of free products of groups 
amalgamated over an amenable subgroup.

We begin in Section~\ref{S-gd} by defining the sofic dimension $s(\sG )$ of a p.m.p.\ groupoid $\sG$, 
as well as a variant $\slower (\sG )$, the lower sofic dimension, obtained by replacing the limit supremum 
in the definition of $s(\sG )$ with a limit infimum.
We prove in Theorem~\ref{T-generating} that these invariants can be computed on any finite generating set. 
We also show that the lower sofic dimension of a sofic p.m.p.\ groupoid with
infinite classes is at least $1$ (Proposition~\ref{P-inf gd}). In Section~\ref{S-groups} we record a couple
of basic results for countable discrete groups, including the fact that $s(G) = 1 - |G|^{-1}$ for a finite group $G$ 
(Proposition~\ref{P-finite}). Section~\ref{S-amalg groups} contains the amalgamated free product formula for groups,
Theorem~\ref{T-amalgamated}, which asserts that, under suitable regularity assumptions, if $G_1$ and $G_2$ are 
countable discrete groups and $H$ is a common amenable subgroup then
\[
s(G_1 *_H G_2 ) = s(G_1 ) + s(G_2 ) - 1 + \frac{1}{|H|} .
\]
As corollaries we deduce that $s(F_r ) = \slower (F_r ) = r$ for every $r\in\Nb\cup\{ \infty \}$
where $F_r$ is the free group of rank $r$, and $s(G) = \slower (G) = 1 - |G|^{-1}$ for amenable groups $G$.
In Section~\ref{S-group actions} we show how the definition of sofic dimension for a measure-preserving action 
$G\curvearrowright X$ of a countable discrete group on a probability space, 
for which we use the notation $s(G,X)$ and $\slower (G,X)$, can be reformulated so as to conveniently separate
the group and space components. We use this reformulation in Section~\ref{S-free prod actions} to establish the 
free product formula, Theorem~\ref{T-free prod action}, which asserts that,
under suitable regularity assumptions, if $G_1$ and $G_2$ are countable
discrete groups and $G_1 * G_2 \curvearrowright X$ 
is a measure-preserving action on a probability space, then
\[
s(G_1 * G_2 ,X) = s(G_1 ,X) + s(G_2 ,X) .
\]
As a corollary, for every $r\in\Nb$ we obtain $s(F_r ,X) = \slower (F_r ,X) = r$ for every measure-preserving 
action of the free group $F_r$.

While working on this project we learned that Mikl\'{o}s Ab\'{e}rt, Lewis Bowen, and Nikolai Nikolov also defined
and studied the same notion of sofic dimension for groups. It is their terminology that we have adopted. 
Our paper answers a question of Mikl\'{o}s Ab\'{e}rt, who asked whether the theory can be extended to measure-preserving 
group actions \cite{AbeSze09}. 
\medskip

\noindent{\it Acknowledgments.} 
Some of this research was conducted while the first author was attending the
Erwin Schr\"odinger Institute in Vienna and he would like to thank
the institute and the organizers of the program on Bialgebras and
Free Probability. 
The second author would like to thank Yasuyuki Kawahigashi for hosting his
January 2010 visit to the University of Tokyo during which
the initial stages of this work were carried out. The third author 
thanks Narutaka Ozawa for helpful discussions on the subject.
We are especially grateful to Hanfeng Li for extensive comments and corrections.

\section{Probability-measure-preserving groupoids}\label{S-gd}

For a groupoid $\sG$ we denote the source and range maps by $\so$ and $\ra$, respectively,
and write $\sG^0$ for the set of units of $\sG$. For a set $A\subseteq\sG^0$ we write
$\sG_A$ for the subgroupoid of $\sG$ consisting of all $x\in\sG$ such that $\so (x) \in A$
and $\ra (x) \in A$, with unit space $A$.

A {\it discrete measurable groupoid} is a groupoid $\sG$ with the structure of a standard Borel space
such that $\sG^0$ is a Borel set, the source, range, multiplication, and inversion maps are all Borel,
and $\so^{-1} (x)$ is countable for every $x\in\sG^0$.

A {\it probability-measure-preserving (p.m.p.)\ groupoid} is a discrete measurable groupoid $\sG$ 
paired with a Borel probability measure $\mu$ on $\sG^0$ such that
\[
\int_{\sG^0} | \so^{-1} (x)\cap B | \, d\mu (x) 
= \int_{\sG^0} | \ra^{-1} (x)\cap B | \, d\mu (x) 
\]
for every Borel set $B\subseteq\sG$. The assignment of this common value to a Borel set $B$ defines 
a $\sigma$-finite Borel measure on $\sG$ which restricts to $\mu$ on $\sG^0$. 
It will also be denoted by $\mu$. When speaking about a p.m.p.\ groupoid $(\sG ,\mu )$ we
will often simply write $\sG$ with the measure $\mu$ being understood.

Let $(\sG ,\mu )$ and $(\sH ,\nu )$ be p.m.p.\ groupoids. 
We say that $\sG$ and $\sH$ are {\it isomorphic} if there exist Borel sets $A\subseteq \sG^0$ and 
$B\subseteq \sH^0$ such that $A$ and $B$ have full measure in 
$\sG^0$ and $\sH^0$, respectively, and a groupoid isomorphism
$\varphi : \sG_A \to\sH_B$ which is Borel and satisfies $\varphi_* \mu = \nu$.

In order to express the notion of a finite approximation to a p.m.p.\ groupoid $(\sG ,\mu )$ 
that will be the basis of our definition of sofic dimension, we will think of $\sG$ in terms of its inverse
semigroup $I_\sG$ of partial isometries, defined as follows.
Let $B$ be a Borel subset of $\sG$ such that the restrictions
of $\so$ and $\ra$ to $B$ are injective. We obtain a partial isometry $s_B$ on $L^2 (\sG ,\mu )$ 
by declaring $s_B \xi (x)$ for $\xi\in L^2 (\sG ,\mu )$ and $x\in\sG$ to be $\xi (y^{-1} )$  
where $y$ is the element of $\ra^{-1} (\so (x))$ satisfying $xy\in B$, 
or $0$ if there is no such $y$.
We then define $I_\sG$ as the collection of partial isometries
which arise in this way. When convenient we will think of elements in $I_\sG$ themselves 
as characteristic functions on $\sG$ which are identified if they agree $\mu$-almost everywhere. 
The collection $I_\sG$ forms an inverse semigroup, where the inverse
of an element $s$ is its adjoint $s^*$, and it is closed under taking sums of finitely many 
pairwise orthogonal elements. It is a subset of the von Neumann algebra $\VN (\sG )$ 
of $\sG$, which can be defined as the strong operator closure of the space $A$ of functions $\eta$ on $\sG$ 
for which the functions $x\mapsto \sum_{y\in\ra^{-1} (x)} |\eta(y)|$ and $x\mapsto \sum_{y\in\so^{-1} (x)} |\eta(y)|$
on $\sG_0$ are essentially bounded, with $A$ being represented on $L^2 (\sG ,\mu )$ via the convolution 
\[
\eta * \xi (x) = \sum_{y\in\ra^{-1} (\so (x))} \eta(xy)\xi(y^{-1} ) .
\]
One can show in fact that $I_\sG$ 
generates $\VN (\sG )$ as a von Neumann algebra. 

Write $\tau$ for the normal trace on $\VN (\sG )$ associated to $\mu$ and $\| \cdot \|_2$ 
for the $2$-norm $a\mapsto \tau (a^* a)^{1/2}$ on $\VN (\sG )$.
For elements $a$ in $L^1 (\sG ,\mu )$, and in particular for $a$ in the linear span of $I_\sG$, the trace is given by
\begin{align*}
\tau (a) = \langle a1_{\sG^0} , 1_{\sG^0} \rangle_{L^2 (\sG ,\mu )} = \int_{\sG^0} a*1_{\sG^0} \, d\mu
= \int_{\sG_0} a(x) d\mu(x).
\end{align*}
We will be using the $2$-norm to measure distances between elements of $I_\sG$.

The three basic examples of p.m.p.\ groupoids are the following:

\begin{enumerate}
\item a countable discrete group $G$, in which case $I_G$ can be identified with $G$ along with the
zero element, and the inverse of the inverse semigroup is the same as the group inverse,

\item a countable discrete group acting by measure-preserving transformations 
on a standard probability space, which reduces to the previous example
when the space consists of a single point, and

\item a measure-preserving equivalence relation $R$ on a standard probability space, 
in which case $I_R$ is the collection of partial transformations 
$\varphi$ with nonnull domain such that $(x,\varphi (x))\in R$ for all $x$ in the domain of $\varphi$,
with two such partial transformations being
identified if they agree on a subset which has full measure in the domain of each.
\end{enumerate}

We write $I_d$ for the inverse semigroup of all partial transformations of $\{ 1,\dots ,d \}$. This is the
inverse semigroup associated to the full equivalence relation $\{ 1,\dots ,d \} \times \{ 1,\dots ,d \}$
on $\{ 1,\dots ,d \}$, which we view as a p.m.p.\ groupoid with respect to the uniform probability
measure on $\{ 1,\dots ,d \}$. We thus view $I_d$ both as the set of all partial transformations 
of $\{ 1,\dots ,d \}$ and as the set of all partial permutation matrices in $M_d$, i.e., partial
isometries whose entries are all either $0$ or $1$. 
The context will dictate which particular meaning is intended. 
We write $S_d$ for the subset of $I_d$ consisting of all permutations of $\{ 1,\dots ,d \}$, 
which we also regard as permutation matrices in $M_d$ in accordance with our double interpretation of $I_d$. 
For a finite set $E$ we write $\Sym (E)$ for the set of all permutations of $E$. This
will occasionally be convenient as a substitute for $S_d$ when dealing with a $d$-element set 
that comes with a description other than $\{ 1,\dots ,d\}$.

We write the unique tracial state on $M_d$ as $\tr$, or sometimes $\tr_d$ if there are matrix algebras
of different dimensions at play.
Note that for $s\in I_d$ the square $\| s \|_2^2 = \tr (s^* s)$ of the $2$-norm
is equal to $1/d$ times the cardinality of the domain of $s$ as a partial transformation.
Also, for any $s,t\in I_d$ we have, writing $\xi_1 , \dots ,\xi_d$ for the standard basis vectors of $\Cb^d$
and $\domain$ for domain,
\begin{align*}
\| s-t \|_2^2 
&= \tau ((s-t)^* (s-t)) 
= \frac{1}{d} \sum_{j=1}^d \langle (s-t)\xi_j , (s-t)\xi_j \rangle \\
&\geq \frac{1}{d} \big| \{ c\in \{ 1,\dots ,d \} : sc \neq tc \} \big| \\
&= \frac{1}{d} \big| (\domain (s) \Delta \domain (t))\cup \{ c\in \domain (s) \cap \domain (t) : sc \neq tc \} \big| .
\end{align*}
This inequality will be useful for example in the proof of Lemma~\ref{L-nbhd gd}. In the case that
$s,t\in S_d$ we have
\begin{align*}
 \| s-t \|_2^2 = \frac{1}{d} \big| \{ c\in \{ 1,\dots ,d \} : sc \neq tc \} \big| .
\end{align*}

Given a p.m.p.\ groupoid $\sG$ and a $d\in\Nb$, we wish to count the number of models of $I_\sG$
in $I_d$. We do this by counting the number of approximately multiplicative maps $I_\sG \to I_d$.

For a subset $\Omega$ of $I_\sG$ we write $\Omega^*$ for $\{ s^* : s\in\Omega \}$.
For $n\in\Nb$ we write $\Omega^{\times n}$ for the $n$-fold Cartesian product 
$\Omega\times\cdots\times\Omega$.
This is to be distinguished from $\Omega^n$, which denotes the set of all products $s_1 \cdots s_n$
where $s_1 , \dots , s_n \in \Omega$. We write $\Omega^{\leq n}$ for the set $\bigcup_{k=1}^n \Omega^n$
and $[\Omega ]$ for the linear span of $\Omega$ in $\VN (\sG)$. 

Given an $\Omega\subseteq I_\sG$ we write $I(\Omega )$ for the set of all elements in $I_\sG$
which can be written as a finite sum of elements in $\Omega$.
Note that the elements in such a sum must have pairwise orthogonal source projections, as well as
pairwise orthogonal range projections.

Let $\sG$ and $\sH$ be p.m.p.\ groupoids. 
Let $F$ be a finite subset of $I_\sG$. For an $n\in\Nb$ and a $\delta > 0$,
a linear map $\varphi : [I_\sG ] \to [I_\sH ]$ is said to be 
{\it $(F,n,\delta )$-approximately multiplicative} if 
\[
\| \varphi (s_1 \cdots s_k ) - \varphi (s_1 ) \cdots \varphi (s_k ) \|_2 < \delta 
\]
for all $k=1,\dots,n$ and $(s_1 ,\dots ,s_k )\in F^{\times k}$.
For $d,n\in\Nb$ and a $\delta > 0$ we define 
$\SA (F ,n,\delta ,d)$ to be the set of all $(F\cup F^* ,n,\delta )$-approximately
multiplicative linear maps $\varphi : [I_\sG ]\to [I_d ] = M_d$ such that 
$\varphi ((F\cup F^* )^{\leq n})\subseteq I_d$ and $| \tr\circ\varphi (s) - \tau (s) | < \delta$ 
for all $s\in (F\cup F^* )^{\leq n}$. 

\begin{definition}
The p.m.p.\ groupoid $\sG$ is said to be {\it sofic} if for all finite sets 
$F\subseteq\sG$, $n\in\Nb$, and $\delta > 0$ the set  
$\SA (F ,n,\delta ,d)$ is nonempty for some $d\in\Nb$.
\end{definition}

Given sets $E$ and $A$ with $E\subseteq A$, a set $Z$, and a collection $\sY$ of maps $A\to Z$, we write
$|\sY |_E$ for the cardinality of the sets of restrictions $\varphi |_E$ where $\varphi\in\sY$.
Note that $\SA (F,n ,\delta ,d)\supseteq \SA (F',n' ,\delta' ,d)$ and hence
$|\SA (F,n ,\delta ,d)|_E \geq |\SA (F',n',\delta' ,d)|_{E'}$ whenever $F\subseteq F'$, $n\leq n'$,
$\delta\geq\delta'$, and $E$ and $E'$ are subsets of $I_\sG$ satisfying $E\supseteq E'$.

\begin{definition}\label{D-groupoid local}
Let $\Omega$ be a subset of $I_\sG$, $E$ and $F$ finite subsets of $I_\sG$, $n\in\Nb$, and $\delta > 0$. 
We set 
\begin{align*}
s_E (F,n,\delta ) &= \limsup_{d\to\infty} \frac{1}{d \log d} \log |\SA (F,n ,\delta ,d)|_E , \\
s_E (F,n) &= \inf_{\delta > 0} s_E (F,n,\delta ), \\
s_E (F) &= \inf_{n\in\Nb} s_E (F,n) , \\
s_E (\Omega ) &= \inf_F s_E (F) , \\
s(\Omega ) &= \sup_E s_E (\Omega ) 
\end{align*}
where $F$ in the second last line and $E$ in the last line both range over the finite subsets of $\Omega$.
We similarly define $\slower_E (F,n,\delta )$, $\slower_E (F,n)$, $\slower_E (F)$, $\slower_E (\Omega )$,
and $\slower (\Omega )$ by replacing the limit supremum in the first line with a limit infimum. If 
$\SA (F,n ,\delta ,d)$ is empty for all sufficiently large $d$ we set
$s_E (F,n,\delta ) = -\infty$, and if 
$\SA (F,n ,\delta ,d)$ is empty for arbitrarily large $d$ we set
$\slower_E (F,n,\delta ) = -\infty$.
\end{definition}

Note that if $\Omega$ is finite in the above definition then the notation $s_E (\Omega )$
is unambiguous since $s_E (F' ) \leq s_E (F)$ whenever $F$ and $F'$ are finite subsets of $I_\sG$
with $F' \supseteq F$.

\begin{definition}
The {\it sofic dimension} $s(\sG )$ of $\sG$ is defined as $s(I_\sG )$, and the 
{\it lower sofic dimension} $\slower (\sG )$ as $\slower (I_\sG )$.
\end{definition}

It is clear that sofic dimension and lower sofic dimension are invariants for isomorphism of p.m.p.\ groupoids.

For the remainder of the section $(\sG ,\mu )$ will be an arbitrary p.m.p.\ groupoid.

Given a finite set $E\subseteq I_\sG$, on the set of all unital linear maps from 
$[ I_\sG ]$ to $[I_d ] = M_d$ we define the pseudometric
\[
\rho_E (\varphi , \psi ) = \max_{s\in E} \| \varphi (s) - \psi (s) \|_2 . 
\]
For $\varepsilon\geq 0$ write $N_\varepsilon (\cdot ,\rho )$ for the maximal cardinality 
of an $\varepsilon$-separated subset 
with respect to the pseudometric $\rho$. Note that $N_0 (\SA (F ,n ,\delta ,d),\rho_E ) = |\SA (F ,n ,\delta ,d)|_E$.

\begin{definition}
Let $E$ and $F$ be finite subsets of $\sG$, $n\in\Nb$, and $\delta > 0$. We set
\begin{align*}
s_{E,\eps} (F,n,\delta ) &= \limsup_{d\to\infty} \frac{1}{d \log d} \log N_\eps (\SA (F,n ,\delta ,d),\rho_E ), \\
s_{E,\eps} (F,n) &= \inf_{\delta > 0} s_{E,\eps} (F,n,\delta ) .
\end{align*}
We similarly define $\slower_{E,\eps} (F,n,\delta )$ and $\slower_{E,\eps} (F,n)$
by replacing the limit supremum in the first line with a limit infimum. 
If $\SA (F,n ,\delta ,d)$ is empty for all sufficiently large $d$ we set
$s_{E,\eps} (F,n,\delta ) = -\infty$, and if 
$\SA (F,n ,\delta ,d)$ is empty for arbitrarily large $d$ we set
$\slower_{E,\eps} (F,n,\delta ) = -\infty$.
\end{definition}

\begin{lemma}\label{L-nbhd gd}
For every $\kappa > 0$ there is an $\varepsilon > 0$ such that
\[
| \{ t\in I_d : \| t-s \|_2 < \varepsilon \} | \leq d^{\kappa d} .
\]
for all $d\in\Nb$ and $s\in I_d$.
\end{lemma}

\begin{proof}
Let $\varepsilon > 0$. Let $d\in\Nb$ and $s\in I_d$. 
Given a $t\in I_d$ 
we have
\[
\| s-t \|_2^2 \geq \frac{1}{d} \big| \{ c\in \{ 1,\dots ,d \} : sc\neq tc \} \big|
\]
and so if $t$ satisfies $\| s-t \|_2 < \varepsilon$ then
the cardinality of the set of all $c\in \{ 1,\dots ,d \}$ such that $tc \neq sc$ is at most $\varepsilon^2 d$.
Consequently the set $A$ of all $t\in I_d$ such that $\| t-s \|_2 < \varepsilon$ has cardinality at most
$\binom{d}{\lfloor \varepsilon^2 d \rfloor} d^{\lfloor \varepsilon^2 d \rfloor}$,
which is less than $d^{\kappa d}$ for some $\kappa > 0$ depending on 
$\varepsilon$ but not on $d$ with $\kappa\to 0$ as $\varepsilon\to 0$.
\end{proof}

\begin{lemma}\label{L-ineq}
Let $E$ be a finite subset of $I_\sG$. Let $\kappa > 0$. Then there is an $\varepsilon > 0$ such that
$s_E (F,n) \leq s_{E,\eps} (F,n) + \kappa$ and 
$\slower_E (F,n) \leq \slower_{E,\eps} (F,n) + \kappa$
for all finite sets $F\subseteq I_\sG$ containing $E$ and all $n\in\Nb$.
\end{lemma}

\begin{proof}
This is a straightforward consequence of Lemma~\ref{L-nbhd gd}.
\end{proof}

\begin{definition}\label{D-gen}
A set $\Omega\subseteq I_\sG$ is said to be {\it generating} if 
$I(\bigcup_{n=1}^\infty (\Omega\cup\Omega^* )^n )$ is $2$-norm dense in $I_\sG$
and contains the orthogonal complement of each of its projections.
\end{definition}

In the case of a group, the above definition reduces to the usual notion of
generating set, modulo the possible inclusion of the zero element. 
Also, if $G\curvearrowright (X,\mu)$ is a probability-measure-preserving action,
$\cP$ is a set of projections in $L^\infty (X,\mu)$ which dynamically generates $L^\infty (X,\mu)$ 
(see the beginning of Section~\ref{S-group actions}), and $S$ is a 
generating set for $G$, then $\cP\cup S$ is a generating set 
in the sense of Definition~\ref{D-gen}. The second condition in Definition~\ref{D-gen} will 
be important in the proof of the following lemma.

\begin{lemma}\label{L-approx mult}
Let $\Omega$ be a subset of $I_\sG$ containing an independent generating set, and let $L$ be a finite subset of $I_\sG$.
Let $n\in\Nb$ and $\delta > 0$. Then there are a finite set $F\subseteq\Omega$, an $m\in\Nb$, and an
$(L,n,\delta )$-approximately multiplicative linear map $\theta : [I_\sG ]\to [I_\sG ]$
such that $\theta (L^{\leq n}) \subseteq I((F\cup F^*)^{\leq m} )$ and $\| \theta (s) - s \|_2 < \delta$
for all $s\in L^{\leq n}$.
\end{lemma}

\begin{proof}
Set $\delta' = \delta /(n+1)$. 
By a standard selection theorem 
\cite[Thm.\ 18.10]{Kec95}, there exists a countable Borel partition $\cQ$ of $\sG$ into sets on which the
range and source maps are injective. We can then find
a finite set $\cP$ of characteristic functions of pairwise disjoint measurable subsets of 
$\sG^0$ such that $\| 1-\sum_{p\in\cP} p \|_2 < \delta' /3$ and 
for all $s,s'\in L^{\leq n}$ and $p,q\in\cP$ the subsets of $\sG$ of which $qsp$ and $qs' p$ are 
characteristic functions are either the same subset of some member of $\cQ$ or 
subsets of different members of $\cQ$. 
Set $\tilde{L} = \{ qsp : s\in L^{\leq n}\text{ and } p,q\in\cP \}$,
and observe that $\tilde{L}$ is linearly independent.

Let $\delta'' > 0$ be such that $\delta'' \leq \delta' /2$, to be further specified.
Since $\Omega$ is generating, there are a finite set $F\subseteq\Omega$ and a $k\in\Nb$ such that
for every $p\in\cP$ there is a $t_p \in I((F\cup F^*)^{\leq k} )$ such that
$\| p - t_p \|_2 < \delta''$. By requiring $\| p - t_p \|_2$ to be even smaller, replacing $t_p$
with its source projection, and doubling $k$, we may assume that $t_p$ is the characteristic function 
of a subset of $\sG^0$. In view of the second part of the definition of a generating set above,
we may also assume, by a straightforward perturbation argument
that involves cutting down $t_p$ for each $p\in\cP$ by the products of the orthogonal 
complements of the projections $t_q$ for $q\in\cP\setminus\{ q \}$ (which requires us to increase $k$),
that the projections $t_p$ for $p\in\cP$ are pairwise orthogonal.

Since $\Omega$ is generating, by taking $F$ larger if necessary we can find an $\ell\in\Nb$
such that for every $s\in L^{\leq n}$ there is a $v_s \in I((F\cup F^*)^{\leq\ell} )$ with $\| v_s - s \|_2 < \delta''$.
For all $p,q\in\cP$ and $s\in L^{\leq n}$ take a $v_{s,p,q} \in \{ v_{s'} : s'\in L^{\leq n} \text{ and } qsp = qs' p \}$ so that
if $qsp = qs' p$ for $p,q\in\cP$ and $s,s' \in L^{\leq n}$ then $v_{s,p,q} = v_{s',p,q}$.
For $s\in L^{\leq n}$ and $p,q\in\cP$ write $r_{s,p,q} = t_q v_{s,p,q} t_p$ and
\begin{gather*}
t_{s,p,q} =  r_{s,p,q}^* r_{s,p,q} \bigg( \prod_{(p',q') \neq (p,q)} (1-r_{s,p',q'}^* r_{s,p',q'}) \bigg)\hspace*{40mm} \\
\hspace*{40mm} \times \hspace*{1mm} r_{s,p,q}^* \bigg( r_{s,p,q}r_{s,p,q}^*  \prod_{(p',q') \neq (p,q)} (1-r_{s,p',q'}r_{s,p',q'}^* ) \bigg) r_{s,p,q}
\end{gather*}
where $(p',q')$ ranges in $\cP\times\cP$. Then the elements $t_q v_{s,p,q} t_p t_{s,p,q}$ for $p,q\in\cP$ have
pairwise orthogonal source projections and pairwise orthogonal range projections, and thus setting
$\theta(s) = \sum_{p,q\in\cP} t_q v_{s,p,q} t_p t_{s,p,q}$ we obtain a
map $\theta : L^{\leq n} \to I((F\cup F^*)^{\leq m} )$ for a suitably large $m$.
Note that for all $s\in L^{\leq n}$ the elements $qsp$ for $p,q\in\cP$ have pairwise orthogonal source projections
and pairwise orthogonal range projections, and that for all $p,q\in\cP$ we have, taking $s'$ such that $v_{s,p,q} = v_{s'}$,
\begin{align*}
\| t_q v_{s,p,q} t_p - qsp \|_2 &= \| t_q v_{s'} t_p - qs'p \|_2 \\
&\leq \| (t_q - q) v_{s'} t_p \|_2 + \| q (v_{s'} - s') t_p \|_2 + \| qs' (t_p - p) \|_2  < 3\delta'' .
\end{align*}
We can thus take $\delta''$ to be small enough to ensure that for every $s\in L^{\leq n}$ the element 
$t_q v_{s,p,q} t_p t_{s,p,q}$ is close enough to $qsp$ for all $p,q\in\cP$ so that $\| \theta(s) - \sum_{p,q\in\cP} qsp \|_2 < \delta' /3$, 
in which case, writing $z=\sum_{p\in\cP} p$,
\begin{align*}
\| \theta(s) - s \|_2 \leq \| \theta(s) - zsz \|_2 + \| (z-1)sz \|_2 + \| s(z-1) \|_2 < 3\cdot \frac{\delta'}{3} = \delta' .
\end{align*}

Define a map $\psi : \tilde{L} \to [I_\sG ]$ by setting 
$\psi (qsp) = t_q v_{s,p,q} t_p t_{s,p,q}$ for all $s\in L^{\leq n}$ and $p,q\in\cP$, which is well defined since
$v_{s',p,q} = v_{s,p,q}$ and hence also $t_{s',p,q} = t_{s,p,q}$ whenever $qs'p=qsp$.
Since $\tilde{L}$ is linearly independent, 
this extends to a linear map $[\tilde{L} ]\to [I_\sG ]$, which again will be denoted by $\psi$. 
Now if $\sum_i c_i s_i$ is a linear combination of elements of
$L^{\leq n}$ which is equal to zero, then
\begin{align*}
\sum_i c_i \theta (s_i) =  \sum_i c_i \sum_{p,q\in\cP} \psi (qs_i p) 
&= \psi \Big(\sum_i c_i \sum_{p,q\in\cP} qs_i p \Big) \\
&= \psi \Big(\Big(\sum_{p\in\cP} p\Big)\Big(\sum_i c_i s_i \Big)\Big(\sum_{p\in\cP} p\Big)\Big) 
= \psi (0) = 0 .
\end{align*}
It follows that the map $\theta$ extends to a linear map $[L^{\leq n} ] \to [I_\sG ]$, 
which we then extend arbitrarily to a linear map $[I_\sG ]\to [I_\sG ]$, again denoted by $\theta$.

Finally, given $j\in\{ 1,\dots ,n\}$ and $(s_1 , \dots , s_j )\in L^{\times j}$ we have
\begin{align*}
\lefteqn{\| \theta (s_1 \cdots s_j ) - \theta (s_1 )\cdots\theta (s_j ) \|_2}\hspace*{15mm} \\
\hspace*{10mm} &\leq \| \theta (s_1 \cdots s_j ) - s_1 \cdots s_j \|_2 \\
&\hspace*{10mm} \ + \sum_{i=1}^j \| s_1 \cdots s_{i-1} \|_\infty
\| s_i - \theta (s_i ) \|_2 \| \theta (s_{i+1} )\cdots\theta (s_j ) \|_\infty \\
&< \delta' + j\delta' \leq \delta ,
\end{align*}
showing that $\theta$ is $(L,n,\delta )$-approximately multiplicative. 
\end{proof}

\begin{lemma}\label{L-partial isometry}
Let $\delta > 0$. Then whenever $v$ and $w$ are elements of $I_\sG$ satisfying $\| vwv - v \|_2 < \delta$
and $\| wvw - w \|_2 < \delta$ one has $\| w - v^* \|_2 < 4\delta$.
\end{lemma}

\begin{proof}
When acting on the Hilbert space $L^2(\sG,\mu)$, $v$ and $w$ are partial 
isometries and the four projections $v^*v$, $vv^*$, $w^*w$ and $ww^*$ commute with each other.
Consider the projection $r=(1-w^*w)vv^*$. Then $vv^*r=r$ and $wr=0$ so that
$(v-vwv)(v^*rv)=vv^*rv-vwrv=rv$ and hence
\[
(v^*rv)(v-vwv)^*(v-vwv)(v^*rv)=v^*rv.
\]
This yields
\begin{align*}
\tau(r)=\tau(vv^*r)&=\tau(v^*rv)\\
&=\tau\big((v^*rv)(v-vwv)^*(v-vwv)(v^*rv)\big) \\
&\le\tau\big((v^*rv)(v-vwv)^*(v-vwv)(v^*rv)\big)\\
&\hspace*{20mm} \ +\tau\big((1-(v^*rv))(v-vwv)^*(v-vwv)(1-(v^*rv))\big) \\
&=\tau\big((v-vwv)^*(v-vwv)\big)<\delta^2.
\end{align*}
In a similar manner, interchanging $v$ and $w$, we find that $\tau(s)<\delta^2$, where $s=(1-v^*v)ww^*$.
We therefore have $vv^*=x+r$ and $ww^*=y+s$
for projections $x=(vv^*)(w^*w)$ and $y=(ww^*)(v^*v)$.
Consequently,
\begin{gather*}
\tau(w^*w)\ge\tau(x)\ge\tau(vv^*)-\delta^2 ,\\
\tau(v^*v)\ge\tau(y)\ge\tau(ww^*)-\delta^2 ,
\end{gather*}
and $|\tau(v^*v)-\tau(w^*w)|<\delta^2$.
This implies that $w^*w=x+a$ and $v^*v=y+b$
for projections $a$ and $b$, both of trace less than $2\delta^2$.
So $\|vv^*-w^*w\|_2=\sqrt{\tau(r+a)}<\sqrt3\delta$ and, similarly, $\|ww^*-v^*v\|_2<\sqrt3\delta$,
and hence
\begin{align*}
\|v^*-w\|_2 &= \|v^*vv^*-ww^*ww^*w\|_2 \\
&\le\|v^*(v-vwv)v^*\|_2+\|(v^*v-ww^*)w(vv^*)\|_2+\|(ww^*)w(vv^*-w^*w)\|_2 \\
&\le\|v-vwv\|_2+\|v^*v-ww^*\|_2+\|vv^*-w^*w\|_2 <4\delta.
\end{align*}
\end{proof}

\begin{lemma}\label{L-adjoint}
Let $F$ be a finite subset of $\sG$, $n$ an integer greater than $2$, $\delta > 0$, and $d\in\Nb$. 
Let $\varphi\in\SA (F,n,\delta ,d)$. Then $\| \varphi (s^* ) - \varphi (s)^* \|_2 < 4\delta$ for every $s\in F$.
\end{lemma}

\begin{proof}
Let $s\in F$. Since $n\geq 3$ we have 
\[ 
\| \varphi (s)\varphi (s^* )\varphi (s) - \varphi (s) \|_2 = \| \varphi (s)\varphi (s^* )\varphi (s) - \varphi (ss^* s) \|_2 < \delta
\]
and similarly $\| \varphi (s^* )\varphi (s)\varphi (s^* ) - \varphi (s^* ) \|_2 < \delta$, so that 
$\| \varphi (s^* ) - \varphi (s)^* \|_2 < 4\delta$ by Lemma~\ref{L-partial isometry}.
\end{proof}

\begin{theorem}\label{T-generating}
Let $\Omega$ be a generating subset of $I_\sG$. Then $s(\sG ) = s(\Omega )$
and $\slower (\sG ) = \slower (\Omega )$. 
\end{theorem}

\begin{proof}
The theorem is equivalent to the assertion that if $\Upsilon$ is another generating subset of $I_\sG$
then $s(\Omega ) = s(\Upsilon )$ and $\slower (\Omega ) = \slower (\Upsilon )$, and to verify this 
it suffices by symmetry to show that 
$s(\Omega ) \leq s(\Upsilon )$ and $\slower (\Omega ) \leq \slower (\Upsilon )$. We will establish the first of 
these inequalities, with the second following by the same argument with the limit supremum replaced 
everywhere by a limit infimum. 
In view of the definitions we may assume that $\Omega^* = \Omega$ and $\Upsilon^* = \Upsilon$.

Let $E$ be a finite subset of $\Omega$.
Let $\kappa > 0$. By Lemma~\ref{L-ineq} 
there is an $\eps > 0$ such that $s_E (F,n) \leq s_{E,\eps} (F,n) + \kappa$ for all finite sets $F\subseteq G$
and $n\in\Nb$.
Since $\Upsilon$ is generating, we can find a finite set $K\subseteq\Upsilon$ and an integer $n>1$ such 
that for every $s\in E$ there are $\gamma_{s,\ot} \in \{ 0,1 \}$ for which the element 
$\tilde{s} = \sum_{\ot\in \bigcup_{k=1}^n K^{\times k}} \gamma_{s,\ot} \check{\ot} \in I(K^{\leq n} )$
satisfies $\| s - \tilde{s} \|_2 < \eps /16$, where $\check{\ot}$ means $t_1 \cdots t_k$ for $\ot = (t_1 , \dots ,t_k )$.
By increasing $n$ if necessary we can find a finite set $L\subseteq\Upsilon$ satisfying $L^* = L$
and $K\subseteq L$ and a $\delta > 0$ such that
\[
\limsup_{d\to\infty} \frac{1}{d \log d} \log |\SA (L,n ,\delta ,d)|_K \leq s_K (\Upsilon ) + \kappa .
\]
Choose a $\delta' > 0$ such that $|K^{\leq n}| \delta' < \varepsilon /8$.
Since $\Omega$ is generating, by Lemma~\ref{L-approx mult} we can find a finite set $F\subseteq\Omega$
with $E\subseteq F$ and $F^* = F$,
an $m\in\Nb$, and an $(L,n,\delta' /4)$-approximately multiplicative linear map $\theta : [I_\sG ]\to [I_\sG ]$
with $\theta (L^{\leq n} ) \subseteq I(F_\sym^{\leq m} )$ 
such that $\| t - \theta (t) \|_2 < \delta' /2$ for every $t\in L^{\leq n}$. Observe that
for every $s\in E$ we have, since $K^{\leq n} \subseteq L^{\leq n}$,
\begin{align*}
\| s - \theta (\tilde{s} ) \|_2
\leq \| s-\tilde{s} \|_2 + \| \tilde{s} - \theta (\tilde{s} ) \|_2 
< \frac{\eps}{16} + \sum_{\ot\in \bigcup_{k=1}^n K^{\times k}} |\gamma_{s,\ot}| \| \check{\ot} - \theta (\check{\ot} ) \|_2 
< \frac{\eps}{8} ,
\end{align*}
an estimate that will be used towards the end of the proof.

Take a $\delta'' > 0$ such that 
\begin{enumerate}
\item[(i)] $|F^{\leq mn}| (1+n)\delta'' \leq \delta' /2$, and

\item[(ii)] for every linear map $\varphi$ from $[F^{\leq mn}]$ to a Hilbert space, if 
$| \langle \varphi (s),\varphi (t) \rangle - \langle s,t \rangle | < (4+4mn)\delta''$ for all $s,t\in F^{\leq mn}$
then $\| \varphi (f) \|_2 \leq 2\| f \|_2$ for all $f\in [F^{\leq mn}]$.
\end{enumerate}
Let $\varphi\in\SA (F,2mn,\delta'' ,d)$. Given $k\in \{ 1,\dots,mn \}$ and $s,t\in F^k$ and
writing $s = s_1 \cdots s_k$ and $t = t_1 \cdots t_k$ where $s_1 , \dots , s_k , t_1 , \dots ,t_k \in F$,
we have, using Lemma~\ref{L-adjoint},
\begin{align*}
\lefteqn{\| \varphi (t_k )^* \cdots \varphi (t_1 )^* - \varphi (t_k^* )\cdots \varphi (t_1^* ) \|_2} \hspace*{10mm} \\
\hspace*{10mm} &\leq \sum_{i=1}^k \| \varphi (t_k )^* \cdots \varphi (t_{i+1} )^*  (\varphi (t_i )^* - \varphi (t_i^* )) 
\varphi (t_{i-1}^* )\cdots \varphi (t_1^* ) \|_2 \\
&< 4mn\delta''
\end{align*}
so that
\begin{align*}
\lefteqn{\| \varphi(t)^* \varphi (s) - \varphi (t^* s) \|_2}\hspace*{10mm} \\
\hspace*{10mm} &\leq \|\varphi(t_1\cdots t_k)^*-(\varphi(t_1)\cdots\varphi(t_k))^*\|_2\|\varphi(s)\|_\infty \\
&\hspace*{10mm}\ + \|\varphi(t_k)^*\cdots\varphi(t_1)^*-\varphi(t_k^*)\cdots\varphi(t_1^*)\|_2\|\varphi(s)\|_\infty \\
&\hspace*{10mm}\ + \|\varphi(t_k^*)\cdots\varphi(t_1^*)\|_\infty\|\varphi(s_1\cdots s_k)-\varphi(s_1)\cdots\varphi(s_k)\|_2 \\
&\hspace*{10mm}\ + \|\varphi(t_k^*)\cdots\varphi(t_1^*)\varphi(s_1)\cdots\varphi(s_k)-\varphi(t^*s)\|_2 \\
&< (3+4mn)\delta''
\end{align*}
and hence
\begin{align*}
| \langle \varphi (s),\varphi (t) \rangle - \langle s,t \rangle |
&\leq | \tr (\varphi(t)^* \varphi (s) - \varphi (t^* s)) | + | \tr\circ\varphi (t^* s) - \tau (t^* s) | \\
&< \| \varphi(t)^* \varphi (s) - \varphi (t^* s) \|_2 + \delta'' < (4+4mn)\delta'' .
\end{align*}
It follows by our choice of $\delta''$ that $\| \varphi (f) \|_2 \leq 2\| f \|_2$ for all $f\in [F^{\leq mn}]$.
Write $\varphi^\natural$ for $\varphi\circ\theta$. 
We will show that $\varphi^\natural \in \SA (L,n,\delta' ,d)$. Let $k\in \{ 1,\dots ,n\}$ and $t_1 , \dots , t_k \in L$.
For each $i=1,\dots ,k$ we can write 
$\theta (t_i ) = \sum_{\os\in \bigcup_{j=1}^m F^{\times j}} \lambda_{i,\os} \check{\os}$
where $\lambda_{i,\os} \in \{ 0,1 \}$ and $\check{\os}$ means $s_1 \cdots s_j$ for $\os = (s_1 , \dots ,s_j )$.
For every $k=1,\dots ,n$ and $(\os_1 ,\dots ,\os_k )\in F^{\times j_1} \times\cdots\times F^{\times j_k}$ 
where $1\leq j_1 , \dots , j_k \leq m$ we have, 
writing $\os_i = (s_{i,1} ,\dots , s_{i,j_i} )$, 
\begin{align*}
\bigg\| \varphi \bigg( \prod_{i=1}^k \check{\os}_i \bigg) - \prod_{i=1}^k \varphi ( \check{\os}_i ) \bigg\|_2 
&= \bigg\| \varphi \bigg( \prod_{i=1}^k \prod_{j=1}^{j_i} s_{i,j} \bigg) - \prod_{i=1}^k \prod_{j=1}^{j_i} \varphi ( s_{i,j} ) \bigg\|_2 \\
&\hspace*{20mm} \ + \bigg\| \prod_{i=1}^k \prod_{j=1}^{j_i} \varphi ( s_{i,j} ) - 
\prod_{i=1}^k \varphi \bigg( \prod_{j=1}^{j_i} s_{i,j} \bigg) \bigg\|_2 \\
&<\delta''+\sum_{p=1}^k\bigg\|\prod_{i=1}^{p-1}\prod_{j=1}^{j_i} \varphi ( s_{i,j} )\bigg\|_\infty
\bigg\|\prod_{j=1}^{j_i} \varphi ( s_{p,j} )-\varphi \bigg(\prod_{j=1}^{j_i} s_{p,j}\bigg)\bigg\|_2 \\
&\hspace*{35mm}
\ \times\bigg\|\prod_{i=p+1}^k\varphi\bigg(\prod_{j=1}^{j_i} s_{i,j}\bigg)\bigg\|_\infty \\
&\le \delta'' + \sum_{i=1}^k \bigg\| \prod_{j=1}^{j_i} \varphi ( s_{i,j} ) - \varphi \bigg( \prod_{j=1}^{j_i} s_{i,j} \bigg) \bigg\|_2 \\
&< (1+n)\delta'' 
\end{align*}
so that, with $\os$ ranging over $\bigcup_{j=1}^m F^{\times j}$ and $(\os_1 ,\dots ,\os_k )$ 
over $(\bigcup_{j=1}^m F^{\times j} )^{\times k}$ in the sums below,
\begin{align*}
\lefteqn{\| \varphi (\theta (t_1 )\cdots \theta (t_k )) - \varphi (\theta (t_1 ))\cdots \varphi (\theta (t_k )) \|_2}\hspace*{20mm} \\
\hspace*{20mm} &= \bigg\| \varphi \bigg( \prod_{i=1}^k \sum_{\os} \lambda_{i,\os} \check{\os} \bigg)
- \prod_{i=1}^k \varphi \bigg( \sum_{\os} \lambda_{i,\os} \check{\os} \bigg) \bigg\|_2 \\
&= \bigg\| \sum_{(\os_1 ,\dots ,\os_k )}\bigg( \prod_{i=1}^k \lambda_{i,\os_i} \bigg) 
\bigg[ \varphi \bigg( \prod_{i=1}^k \check{\os}_i \bigg)
- \prod_{i=1}^k \varphi ( \check{\os}_i ) \bigg] \bigg\|_2 \\
&\leq \sum_{(\os_1 ,\dots ,\os_k )} \bigg\| \varphi \bigg( \prod_{i=1}^k \check{\os}_i \bigg)
- \prod_{i=1}^k \varphi ( \check{\os}_i ) \bigg\|_2 \\
&< |F^{\leq mn}| (1+n)\delta'' \\
&\leq \frac{\delta'}{2} .
\end{align*}
Therefore
\begin{align*}
\lefteqn{\| \varphi^\natural (t_1 \cdots t_k ) - \varphi^\natural (t_1 )\cdots \varphi^\natural (t_k ) \|_2}\hspace*{20mm} \\
&\leq \| \varphi (\theta (t_1 \cdots t_k ) - \theta (t_1 )\cdots \theta (t_k )) \|_2 \\
&\hspace*{15mm} \ + \| \varphi (\theta (t_1 )\cdots \theta (t_k )) - \varphi (\theta (t_1 ))\cdots \varphi (\theta (t_k )) \|_2 \\
&\leq 2 \| \theta (t_1 \cdots t_k ) - \theta (t_1 )\cdots \theta (t_k ) \|_2 + \frac{\delta'}{2} \\
&< \frac{\delta'}{2} + \frac{\delta'}{2} = \delta' .
\end{align*}
Finally, for $t\in L^{\leq n}$ we can write
$\theta (t) = \sum_{\os\in \bigcup_{k=1}^m F^{\times k}} \lambda_{t,\os} \check{\os}$
where $\lambda_{t,\os} \in \{ 0,1 \}$ and $\check{\os}$ means $s_1 \cdots s_k$ for $\os = (s_1 , \dots ,s_k )$,
so that
\begin{align*}
| \tr\circ\varphi (\theta (t)) - \tau (\theta (t)) |
&= \bigg| \sum_{\os\in \bigcup_{k=1}^m F^{\times k}} \lambda_{t,\os} (\tr\circ\varphi (\check{\os} ) - \tau (\check{\os} )) \bigg| \\
&\leq \sum_{\os\in \bigcup_{k=1}^m F^{\times k}} | \tr\circ\varphi (\check{\os} ) - \tau (\check{\os} ) | 
\leq |F^{\leq m}| \delta'' 
< \frac{\delta'}{2}
\end{align*}
and hence
\begin{align*}
|\tr\circ\varphi^\natural (t) - \tau (t) | 
&\leq | \tr\circ\varphi (\theta (t)) - \tau (\theta (t)) | + | \tau (\theta (t) - t ) | \\
&< \frac{\delta'}{2} + \| \theta (t) - t \|_2 
< \delta' .
\end{align*}
Thus $\varphi^\natural \in \SA (L,n,\delta' ,d)$, as desired. 

Let $\Gamma : \SA (F,mn,\delta'' ,d) \to \SA (L,n,\delta' ,d)$ be the map $\varphi \mapsto \varphi^\natural$.
Pick an $\varepsilon' > 0$ such that $2|L|^n n\varepsilon' < \varepsilon /4$. 
Let $Z$ be an $\varepsilon'$-net in $\SA (L,n,\delta' ,d)$ with respect to
$\rho_K$ of minimal cardinality. Each element of $Z$ within distance $\varepsilon'$ to 
$\Gamma (\SA (F,mn,\delta'' ,d))$ we perturb to an element of $\Gamma (\SA (F,mn,\delta'' ,d))$
in order to construct a set $Y \subseteq \SA (F,mn,\delta'' ,d)$ such that $|Y|\leq |Z|$ and
$\Gamma (Y )$ is a $2\varepsilon'$-net for $\Gamma (\SA (F,mn,\delta'' ,d))$ with respect to $\rho_K$.
Let $\varphi$ and $\psi$ be elements of $\SA (F,mn,\delta'' ,d)$ with 
$\rho_K (\varphi^\natural , \psi^\natural ) < 2\varepsilon'$. Then for $k\in \{1,\dots,n\}$
and $\ot = (t_1 , \dots ,t_k )\in K^{\times k}$ we have, since $K\subseteq L$,
\begin{align*}
\| \varphi^\natural (\check{\ot} ) - \psi^\natural (\check{\ot} ) \|_2
&\leq \| \varphi^\natural (t_1 \cdots t_k ) - \varphi^\natural (t_1 )\cdots \varphi^\natural (t_1 ) \|_2 \\
&\hspace*{15mm} \ + \| \varphi^\natural (t_1 )\cdots \varphi^\natural (t_k ) - \psi^\natural (t_1 )\cdots \psi^\natural (t_k ) \|_2 \\
&\hspace*{15mm} \ + \|  \psi^\natural (t_1 )\cdots \psi^\natural (t_k ) - \psi^\natural (t_1 \cdots t_k ) \|_2 \\
&< 2\delta' + \sum_{i=1}^k \| \psi^\natural (s_1 )\cdots \psi^\natural (s_{i-1} ) \|_\infty \| \varphi^\natural (t_i ) - \psi^\natural (t_i ) \|_2 
\| \varphi^\natural (s_{i+1}) \cdots\varphi^\natural (s_k ) \|_\infty \\
&< 2(\delta' + n\varepsilon' )
\end{align*}
and thus, for $s\in E$, with $\ot$ ranging over $\bigcup_{k=1}^n K^{\times k}$ in the sums below,
\begin{align*}
\| \varphi^\natural (\tilde{s} ) - \psi^\natural (\tilde{s} ) \|_2
&= \bigg\| \varphi^\natural \bigg( \sum_{\ot} \gamma_{s,\ot} \check{\ot} \bigg)
- \psi^\natural \bigg( \sum_{\ot} \gamma_{s,\ot} \check{\ot} \bigg) \bigg\|_2 \\
&= \bigg\| \sum_{\ot} \gamma_{s,\ot} (\varphi^\natural (\check{\ot} ) 
- \psi^\natural (\check{\ot} )) \bigg\|_2 \\
&\leq \sum_{\ot} \| \varphi^\natural (\check{\ot} ) - \psi^\natural (\check{\ot} ) \|_2 \\
&< 2|K^{\leq n}| (\delta' + n\varepsilon' ) 
< \frac{\varepsilon}{2}
\end{align*}
whence, using the fact that $E\subseteq F$,
\begin{align*}
\rho_E (\varphi , \psi ) &= \max_{s\in E} \| \varphi (s) - \psi (s) \|_2 \\
&\leq \max_{s\in E} ( \| \varphi (s - \theta (\tilde{s} )) \|_2 + \| \varphi^\natural (\tilde{s} ) - \psi^\natural (\tilde{s} ) \|_2
+ \| \psi (\theta (\tilde{s} ) - s) \|_2 ) \\
&< 4\max_{s\in E} \| s - \theta (\tilde{s} ) \|_2 + \frac{\varepsilon}{2} \\
&< 4\cdot \frac{\eps}{8} + \frac{\eps}{2} = \eps .
\end{align*}
Therefore $Y$ is an $\varepsilon$-net for $\SA (F,mn,\delta'' ,d)$ with respect to $\rho_E$, and so
\begin{align*}
N_\varepsilon (\SA (F,mn,\delta'' ,d),\rho_E ) \leq |Y| 
\leq |Z| &\leq N_{\varepsilon'} (\SA (L,n,\delta' ,d),\rho_K ) 
\leq |\SA (L,n,\delta ,d)|_K
\end{align*}
using the fact that $\delta' \leq\delta$.
Consequently
\begin{align*}
s_E (\Omega ) &\leq s_E (F,mn) \\
&\leq s_{E,\eps} (F,mn) + \kappa \\
&\leq \limsup_{d\to\infty} \frac{1}{d \log d} \log N_\varepsilon (\SA (F,mn,\delta'' ,d),\rho_E ) + \kappa \\
&\leq \limsup_{d\to\infty} \frac{1}{d \log d} \log |\SA (L,n,\delta ,d)|_K + \kappa \\
&\leq s_K (\Upsilon ) + 2\kappa 
\leq s(\Upsilon ) + 2\kappa .
\end{align*}
Since $E$ was an arbitrary finite subset of $\Omega$ and $\kappa$ an arbitrary positive number, 
we conclude that $s(\Omega )\leq s(\Upsilon )$.
\end{proof}

\begin{definition}
A set $\Omega\subseteq\sG$ is said to be {\it approximation regular} if $s(\Omega ) = \slower (\Omega )$. 
We say that $\sG$ is {\it approximation regular} if $s(\sG ) = \slower (\sG )$.
\end{definition}

We round out this section by recording a few basic facts about sofic dimension.

\begin{lemma}\label{L-mult}
Let $E$ and $F$ be nonempty finite subsets of $I_\sG$ and let $n\in\Nb$.
Consider a sequence $1\le d_1 < d_2 <\dots$ of integers where $\lim_{k\to\infty} d_{k+1}/d_k = 1$. Then
\begin{align*}
s_E (F,n) &= \inf_{\delta > 0} \limsup_{k\to\infty} \frac{1}{d_k\log d_k}\log |\SA(F,n,\delta,d_k) |_E , \\
\slower_E (F,n) &= \inf_{\delta > 0} \liminf_{k\to\infty} \frac{1}{d_k\log d_k}\log |\SA(F,n,\delta,d_k) |_E .
\end{align*}
In particular, for every $\ell\in\Nb$,
\begin{align*}
s_E (F,n) &= \inf_{\delta > 0} \limsup_{d\to\infty} \frac{1}{\ell d\log\ell d}\log |\SA(F,n,\delta,\ell d)|_E , \\
\slower_E (F,n) &= \inf_{\delta > 0} \liminf_{d\to\infty} \frac{1}{\ell d\log\ell d}\log |\SA(F,n,\delta,\ell d)|_E .
\end{align*}
\end{lemma}

\begin{proof}
For integers $1\le d_1\le d_2$, we have the inclusion $I_{d_1}\subseteq I_{d_2}$,
as partial transformations of $\{1,\ldots,d_1\}$ may be viewed as partial transformations of $\{1,\ldots,d_2\}$
which fix the points from $d_1 + 1$ to $d_2$.
It is easily seen that this results in an inclusion $\SA(F,n,\delta,d_1)\subseteq \SA(F,n,\delta',d_2)$,
where 
$\delta'=\delta'(d_1,d_2)= \delta + \sqrt{(d_2-d_1 )/d_2}$.
Thus, $\delta'\to\delta$ if $d_1$ and $d_2$ are increasing without bound in such a way that $d_2/d_1\to1$.
Moreover, if $r=d_2/d_1$ then
\[
\frac{d_2\log d_2}{d_1\log d_1}=r\frac{\log d_1 +\log r}{\log d_1},
\] 
so also this ratio tends to $1$.
This implies that for every $\eta>0$ and $\lambda > 1$ there is a $k_0 \in\Nb$ such that for every integer $k\geq k_0$ 
and integer $d$ with $d_k \le d\le d_{k+1}$ we have
\begin{align*}
\frac{\lambda^{-1}}{d_k\log d_k}\log |\SA(F,n,\delta-\eta,d_k)|_E
&\le\frac{1}{d\log d}\log |\SA(F,n,\delta,d)|_E \\
&\le\frac{\lambda}{d_{k+1}\log d_{k+1}}\log |\SA(F,n,\delta+\eta,d_{k+1})|_E
\end{align*}
and the lemma follows from this.
\end{proof}

A p.m.p.\ groupoid $\sG$ is said to have infinite classes if $\so^{-1} (\{ x \} )$
is infinite (equivalently, $\ra^{-1} (\{ x \} )$ is infinite) for $\mu$-almost every $x\in\sG^0$.

\begin{proposition}\label{P-inf gd}
Suppose that the p.m.p.\ groupoid $\sG$ is sofic and has infinite classes. Then $\slower (\sG ) \geq 1$. 
\end{proposition}

\begin{proof}
Let $m$ and $n$ be integers greater than $1$ and let $0 < \varepsilon < 1/2$. 
Since $\sG$ has infinite classes, the sets 
$\so^{-1} (x) \cap (\sG\setminus\sG^0 )$ and $\ra^{-1} (x)\cap (\sG\setminus\sG^0 )$ 
are countably infinite for $\mu$-almost every $x\in\sG^0$. By a standard selection theorem 
\cite[Thm.\ 18.10]{Kec95}, as used in the proof of Theorem~1 in \cite{FelMoo77} 
in the equivalence relation setting,
there exist a countable Borel partition of $\sG$ into sets on which the
range and source maps are injective.
Thus we can find disjoint Borel sets $B_1 , \dots , B_k \subseteq \sG\setminus\sG^0$ and a Borel set
$Y\subseteq \sG^0$ with $\mu (Y) \geq 1-\varepsilon /2$ such that $\so |_{B_i}$ and $\ra |_{B_i}$ are injective
for every $i=1,\dots ,k$ and $|\so^{-1} (x) \cap \bigcup_{i=1}^k B_i | \geq m$ for every $x\in Y$.
For each $i=1,\dots ,k$ write $s_i$ for the element of $I_\sG$ defined by the characteristic function of $B_i$.
Note that $\tau (s_i ) = 0$ for every $i=1,\dots ,k$. Set $E = \{ s_1 , \dots ,s_k , 1_Y \}$. 
Take a finite set $F\subseteq I_\sG$ with $1_Y \in F^* = F$, an $n\in\Nb$, and a $\delta > 0$ such that 
\[
\slower_E (I_\sG ) + \eps \geq \liminf_{d\to\infty} \frac{1}{d\log d}\log |\SA (F,n,\delta ,d)|_E .
\]
By shrinking $\delta$ if necessary we may assume that it is sufficiently small as a function of 
$\varepsilon$, $m$, and $k$ for a purpose to be described in a moment.

Since $\sG$ is sofic we can find an $\ell\in\Nb$ and an
$(F,n,\delta )$-approximately multiplicative linear map
$\varphi : [I_\sG ]\to [I_\ell ]$ such that $\varphi (F^{\leq n})\subseteq I_\ell$ and
$|\tr_\ell \circ\varphi (s) - \tau (s) | < \delta$ for all $s\in F^{\leq n}$.
Since $\tau (s_i ) = 0$ for every $i=1,\dots ,k$,
$\tau (s_j^* s_i ) = 0$ for all distinct $i,j\in \{ 1,\dots k\}$, and $1_Y \in F$, 
by a straightforward approximation argument
we can find, assuming $\delta$ to be small enough as a function of $\varepsilon$, $m$, and $k$, 
a set $C\subseteq\{ 1,\dots ,\ell \}$ with $\tr_\ell (1_C ) > 1 - \varepsilon$ such that 
\begin{enumerate}
\item $\sum_{i=1}^k \varphi (s_i^* s_i )1_C \geq m\cdot 1_C$ in $M_\ell$,

\item $\tr_\ell (\varphi (s_i )1_C ) = 0$ for all $i=1,\dots ,k$,

\item $\tr_\ell (\varphi (s_j^* s_i )1_C ) = 0$ for all distinct $i,j\in \{ 1,\dots k\}$, and

\item $\varphi (s_j )^* \varphi (s_i )c = \varphi (s_j ^* s_i )c$ for all $i,j\in \{ 1,\dots ,k \}$ and $c\in C$.
\end{enumerate}

Decompose $\{ 1,\dots ,\ell \}$ into subsets which are invariant under 
$\varphi (s_i )$ for every $i=1,\dots ,k$ and are minimal 
with respect to this property. Write $A_1 , \dots ,A_q$ for the members of this
collection which have cardinality at least $m$.
We claim that $C\subseteq\bigcup_{i=1}^q A_i$. To verify this, 
let $c\in C$ and write $I$ for the set of all $i\in \{ 1,\dots ,k \}$
such that the domain of the partial transformation $\varphi (s_i )$ contains $c$. By conditions (1) and (4) above,
the set $I$ has cardinality at least $m$. Now suppose that $\varphi (s_i )c = \varphi (s_j )c$ 
for some $i,j\in I$. Then $\varphi (s_j^* s_i )c = \varphi (s_j )^* \varphi (s_i )c = \varphi (s_j )^* \varphi (s_j )c = c$.
It follows that $i=j$, for otherwise $\tr_d (\varphi (s_j^* s_i )1_C ) > 0$, contradicting (3). 
We thereby deduce that $C\subseteq\bigcup_{i=1}^q A_i$.

Now let $d\in\Nb$. For each $j=0,\dots ,d - 1$ define the bijection 
$\gamma_j : \{ 1,\dots ,\ell \} \to \{ j\ell +1 ,j\ell +2 ,\dots ,j\ell +\ell \}$ by $\gamma_j (c) = j\ell +c$.
Define a map $\psi : I_\sG \to I_{\ell d}$ by
\[
\psi (s) (j\ell +c) = \gamma_j \circ\varphi (s) \circ\gamma_j^{-1}  (j\ell +c) 
\]
for $s\in G$, $j=0,\dots ,d - 1$, and $c=1,\dots ,\ell$.  
Then $\psi$ is an $(F,n,\delta )$-approximately multiplicative map
such that $| \tr_{\ell d}\circ\psi (s) - \tau (s) | < \delta$ for all $s\in F^{\leq n}$.
For each $j=0,\dots ,d - 1$ and $i=1,\dots ,q$ write $A_{j,i}$ for the subset $\gamma_j (A_{i} )$ of
$\{ 1,\dots ,\ell d \}$. Set $A = \bigcup_{j=0}^{d - 1} \bigcup_{i=1}^q A_{j,i}$.

Write $n_i$ for the cardinality of $A_i$.
Note that the number of ways of partitioning $A$ into $\ell k$ many subsets with cardinalities $|A_{j,i} |$ for 
$i=1,\dots ,q$ and $j=0,\dots ,d - 1$ is bounded below by
\[
\frac{|A|!}{n_1 !^d \cdots n_q !^d (dq)!}
\]
(the factor $(dq)!$ in the denominator accounts for the possible repetition of cardinalities among the subsets, 
yielding the exact formula in the extreme case that all of the subsets have the same cardinality).
For each one of these partitions choose a permutation of $\{ 1,\dots ,\ell d \}$ which sends
each partition element to one of the $A_{j,i}$ with the same cardinality. Write $\cS$ for the collection
of these permutations. Then the conjugates of $\psi$ by the permutations in $\cS$, when restricted to $E$,
are pairwise distinct by construction. It follows using Lemma~\ref{L-mult} and Stirling's approximation that
\begin{align*}
\slower (\sG ) + \eps \geq \slower_E (I_\sG ) + \eps
&\geq \liminf_{d\to\infty} \frac{1}{\ell d\log (\ell d)}\log |\SA (F,n,\delta ,\ell d)|_E \\
&\geq \liminf_{d\to\infty} 
\frac{1}{\ell d\log (\ell d)} \log \bigg( \frac{|A|!}{n_1 !^d \cdots n_q !^d (dq)!} \bigg) \\
&\geq \liminf_{d\to\infty} \frac{1}{\ell d\log (\ell d)} \log \bigg( 
\frac{((1-\varepsilon )\ell d)^{(1-\varepsilon )\ell d}}{n_1^{d n_1} \cdots n_q^{d n_q}  (dq)^{dq}} \bigg) \\
&\geq \liminf_{d\to\infty} \bigg[ 1-\varepsilon 
- \frac{\sum_{i=1}^q n_i \log n_i}{\ell\log (\ell d)} - \frac{q\log (dq)}{\ell\log (\ell d)} \bigg] \\
&= 1 - \varepsilon - \frac{q}{\ell} \\
&\geq 1 - \varepsilon - \frac{1}{m} .
\end{align*}
Since $\varepsilon$ was an arbitrary positive number
and $m$ an arbitrary integer greater than $1$, we conclude that $\slower (\sG ) \geq 1$.
\end{proof}

\begin{proposition}\label{P-upper bound}
Let $F$ be a finite subset of $\sG$. Then $s(F) \leq |F|$.
\end{proposition}

\begin{proof}
For every $n\in\Nb$, $\delta > 0$, and $d\in\Nb$ the number of restrictions $\sigma |_F$ where 
$\sigma\in\SA (F,\delta ,n,d)$ is at most $\big(\sum_{k=0}^d \binom{d}{k}^2 k!\big)^{|F|}$, which is bounded
above by $\big(\binom{2d}{d} d!\big)^{|F|}$, which for a given $\eps > 0$ is less
than $d^{(1+\varepsilon )|F|d}$ for all sufficiently large $d$ by Stirling's approximation, 
giving the result.
\end{proof}

Proposition~\ref{P-upper bound} immediately implies the following.

\begin{proposition}\label{P-min gen}
The quantity $s(\sG )$ is bounded above by the smallest cardinality of a set of generators for $\sG$.
\end{proposition}

\section{Groups}\label{S-groups}

Throughout this section $G$ is a countable discrete group. In this case $I_G$ can be identified with $G$
along with the zero element.
We will simply record here some basic facts, and then discuss amalgamated free products
and amenability in the next section.

For the purpose of formulating sofic dimension in the case of groups it is equivalent and 
technically more convenient to work with maps into $S_d$ instead of $I_d$, so that the sofic models
for group elements are full permutation matrices. We will also write $\sigma_s$ instead of
$\sigma (s)$ for the image of an element $s\in G$ under a map $\sigma : G \to S_d$.
Given a finite set $F\subseteq G$,
$n,d\in\Nb$, and a $\delta > 0$, we write $\GA (F,n,\delta ,d)$ for the set of all
identity-preserving maps $\sigma : G \to S_d$ such that 
\begin{enumerate}
\item $\| \sigma_{s_1 ,\dots ,s_n} - \sigma_{s_1} \cdots\sigma_{s_n} \|_2 < \delta$ 
for all $(s_1 , \dots s_n )\in (F\cup F^* \cup \{ e \} )^{\times n}$, and

\item $\tr_d (\sigma_s ) < \delta$ for all $s\in (F\cup F^* \cup \{ e \} )^n \setminus \{ e \}$,
\end{enumerate}
For a finite set $E\subseteq G$ we write $|\GA (F,n,\delta ,d)|_E$ for the cardinality of $\GA (F,n,\delta ,d)$
modulo equality on $E$, i.e., the cardinality of the set of restrictions $\sigma |_E$ where 
$\sigma\in\GA (F,n,\delta ,d)$. By a straightforward argument that uses Lemma~\ref{L-nbhd gd} to handle
the problem that the images of a group element under maps in $\SA (F,n,\delta ,d)$ need not have full domain
and that also requires perturbing maps in $\SA (F,n,\delta ,d)$ so as to be identity-preserving,
one can readily verify in the case $E\subseteq F$ that 
\[
s_E (F) = \inf_{n\in\Nb} \inf_{\delta > 0} \limsup_{d\to\infty} \frac{1}{d\log d} \log |\GA (F,n,\delta ,d)|_E
\]
and
\[
\slower_E (F) = \inf_{n\in\Nb} \inf_{\delta > 0} \liminf_{d\to\infty} \frac{1}{d\log d} \log |\GA (F,n,\delta ,d)|_E .
\]

The following are special cases of Propositions~\ref{P-inf gd} and \ref{P-min gen}, respectively.

\begin{proposition}\label{P-sofic inf group}
If $G$ is sofic and infinite then $\slower (G) \geq 1$. 
\end{proposition}

\begin{proposition}
The quantity $s(G)$ is bounded above by the smallest cardinality of a set of generators for $G$.
\end{proposition}

\begin{proposition}\label{P-finite index}
Let $H$ be a finite index subgroup of $G$. Then
\[
(s(H)-1) \leq [G:H] (s(G)-1) .
\]
\end{proposition}

\begin{proof}
Set $m = [G:H]$. Take a set $R$ of representatives
for the left cosets of $H$ in $G$ with $e\in R$. 
Define a map $\beta : G\to R$ by declaring
$\beta (s)$ to be the unique element in $R\cap sH$ for every $s\in G$. Then,
given any $s\in G$, writing $\beta (s)(\beta (s)^{-1} s)$ gives a unique expression of $s$ 
as a product of an element in $R$ and an element of $H$. 

Suppose first that $s(H)$ is finite. Let $\kappa > 0$. 
Take a finite set $K\subseteq H$ such that $s_K (H) \geq s(H) - \kappa$.
Let $E$ be a finite subset of $G$ containing $R$ and $K$. 
Let $F$ be a finite symmetric subset of $G$ containing $R$, and let $\delta > 0$ and $n\in\Nb$. 
Set $L = H\cap R^{-1} FR$. 

Let $n\in\Nb$ and $\delta > 0$, and
let $d\in\Nb$. Let $\sigma\in\GA (L,n,\delta ,d)$. 
Define a map $\omega : G\to\Sym (\{ 1,\dots ,d \} \times R)$ by setting
\[
\omega_s (c,t) = (\sigma_{\beta (st)^{-1} st} (c),\beta (st)) 
\]
for all $s\in G$ and $(c,t) \in \{ 1,\dots ,d \} \times R$.
Now if $(s_1 , \dots , s_n )\in F^{\times n}$ and $(c,t) \in \{ 1,\dots ,d \} \times R$ then
\begin{align*}
\omega_{s_1 \cdots s_n} (c,t) = 
\big( \sigma_{\beta (s_1 \cdots s_n t)^{-1} s_1 \cdots s_n t}(c),
\beta (s_1 \cdots s_n t)\big) 
\end{align*}
and, using the fact that $\beta (r_1 \beta (r_2 )) = \beta (r_1 r_2 )$ for all $r_1 , r_2 \in G$, 
\begin{align*}
\omega_{s_1} \cdots\omega_{s_n} (c,t) = 
\bigg( \bigg( \prod_{i=1}^n \sigma_{\beta (s_i \cdots s_n t)^{-1} s_i \beta (s_{i+1} \cdots s_n t)} \bigg) (c),
\beta (s_1 \cdots s_n t) \bigg) .
\end{align*}
Now for every $t\in R$, the proportion of
$c\in \{ 1, \dots , d \}$ such that
\[
\sigma_{\beta (s_1 \cdots s_n t)^{-1} s_1 \cdots s_n t}(c) \neq
\bigg( \prod_{i=1}^n \sigma_{\beta (s_i \cdots s_n t)^{-1} s_i \beta (s_{i+1} \cdots s_n t)} \bigg) (c)
\]
is equal to $\| \sigma_{\beta (s_1 \cdots s_n t)^{-1} s_1 \cdots s_n t} -
\prod_{i=1}^n \sigma_{\beta (s_i \cdots s_n t)^{-1} s_i \beta (s_{i+1} \cdots s_n t)} \|_2^2$.
Since $\beta (s_i \cdots s_n t )^{-1} s_i \beta (s_{i+1} \cdots s_n t)$ is
an element of $L$ for each $i=1, \dots ,n$,
we infer that
$\| \omega_{s_1} \cdots\omega_{s_n} - \omega_{s_1 \cdots s_n} \|_2 < \delta$. 
If $s=e$, then $\omega_s$ is the identity permutation, as required.
If $s\in F\backslash\{e\}$, then for $t\in R$, either (i) $\beta(st)\ne t$, in which 
case $\omega_s(c,t)\ne(c,t)$ for every $c$, or (ii) $\beta(st)=t$, in which case $\beta(st)^{-1}st=t^{-1}st\ne e$,
and the proportion of $c$ for which $\omega_s(c,t)=(c,t)$ is less than $\delta$;
in either case, we have $|\tr(\omega_s)|<\delta$.
Therefore $\omega\in\GA (F,\delta ,n,md)$.

Note that $\omega_s(c,e)=(c,s)$ for every $s\in R$.
Write $\sP$ for the collection of all colorings of $\{ 1,\dots ,md\}$
into $d$ different colors $\{1,\ldots,d\}$, with exactly $m$ elements of each color.
Given $P\in\sP$ and $c\in\{1,\ldots,d\}$ write $P_c$ for the set of elements with color $c$
and choose
a bijection 
$\gamma_P : \{ 1,\dots ,d \} \times R \to \{ 1,\dots ,md \}$ such that 
$\gamma_P (\{(c,s):s\in R\}) = P_c$
for each $c=1,\dots ,d$.
Define $\sigma_P : G\to\Sym (md)$ by $s\mapsto \gamma_P\,\omega_s\gamma_P^{-1}$;
this is an element of $\GA (F,\delta ,n,md)$ since $\omega$ is. 

Having thus constructed a $\sigma_P \in\GA (F,\delta ,n,md)$ for every $\sigma\in\GA(L,\delta,n,d)$ and $P\in\sP$,
we observe that, given a $\rho\in\GA(F,\delta,n,md)$, 
if $W$ is a subset of $\GA(L,\delta,n,d)\times\sP$ such that the pairs 
$(\sigma |_K ,P)$ for $(\sigma ,P)\in W$ are all distinct and 
$\sigma_P |_E = \rho |_E$ for all $(\sigma ,P)\in W$, then $W$ has cardinality
at most $\frac{(md)!}{(md-d)!}$, since $R\subseteq E$. 
Indeed if $\sigma_P |_E$ and the $d$ values $x_c=\gamma_P(c,e)$ for $c=1,\ldots,d$ are specified, 
then the coloring $P$ is determined by $P_c=\{\sigma_{P ,s} (x_c ) :s\in R\}$.
Since $P$ is determined, we know $\gamma_P$ and recover $\sigma |_K$.
Thus, since $\frac{(md)!}{(md-d)!}\le(md)^d$,
\[
|\GA (F,\delta ,n,md)|_E \geq \frac{|\sP |}{(md)^d} |\GA (L,\delta ,n,d)|_K
=\frac{(md)!}{m!^d(md)^d} |\GA (L,\delta ,n,d)|_K .
\]
Therefore, employing Lemma~\ref{L-mult} and using Stirling's approximation,
\begin{align*}
s_E (F) &\geq \limsup_{d\to\infty} \frac{1}{md \log md} \log \frac{(md)!}{m!^d(md)^d} + \frac{s_K (L)}{m} \\
&\geq 1 - \frac{1}{m} + \frac{s(H)}{m} + \frac{\kappa}{m} .
\end{align*}
Taking an infimum over all finite sets $F\subseteq G$ and letting $\kappa\to 0$, we obtain
\[
s(G) \geq s_E (G) \geq 1 - \frac{1}{m} + \frac{s(H)}{m} ,
\]
yielding the desired conclusion. 

Observe finally that when $s(H) = \infty$ the above arguments show that $s(G) = \infty$.
\end{proof}

\begin{question}
When is the inequality in the above proposition an equality?
\end{question}

\begin{proposition}\label{P-finite}
Suppose that $G$ is finite. Then
\[
s(G) = 1 - \frac{1}{|G|} .
\]
\end{proposition}

\begin{proof}
Applying Proposition~\ref{P-finite index} with $H = \{ e \}$ we obtain $s(G) \geq 1 - |G|^{-1}$ since obviously
$s(\{ e \} ) = 0$. To complete the proof let us show that $s(G) \leq 1 - |G|^{-1}$. Set $m = |G|$. 
Let $0<\kappa <1$ be small and let $n\in\Nb$ and $\delta > 0$. Let $d\in\Nb$. 
It is readily seen that if $n\geq 2$ and $\delta$ is small enough as a function of $\kappa$ and $|G|$ then 
for every $\sigma\in\GA (G,\delta ,n,d)$ the set
\[
V_\sigma = \big\{ c\in\{1 ,\dots ,d\} :
\sigma_{st}(c)=\sigma_s(\sigma_t(c))\text{ for all } s,t\in G\text{ and }
\sigma_s (c) \neq c \text{ for all } s\in G\setminus\{ e \} \big\} .
\]
will have cardinality at least $(1-\kappa )d$. Observe that each of the sets $V_\sigma$ can be 
partitioned into $\sigma (G)$-invariant subsets of cardinality $m$, on each of which $\sigma$
yields a transitive action of $G$ (thus, a copy of $G$ acting on itself by left multiplication).
Let $q$ be the smallest multiple of $m$ which is no less than $(1-\kappa )d$.
The number of subsets of $\{ 1,\dots ,d \}$ of cardinality $q$ is at most $\binom{d}{\kappa d}$
and the number of ways of partitioning each such subset into subsets of cardinality $m$ is at most 
$q! / ((m!)^{q/m}(q/m)!)$
and the number of ways $G$ can act transitively on each of these sets is bounded above by $m!$.
Since $G$ can map to permutations on a set of cardinality at most $\kappa d$
in at most $((\kappa d)!)^m$ ways, we obtain
\begin{align*}
|\GA (G,n,\delta ,d)| \leq  \frac{q!}{(m!)^{q/m}(q/m)!}(m!)^{q/m} \binom{d}{\kappa d} ((\kappa d)!)^m .
\end{align*}
Using $(1-\kappa)d\le q\le(1-\kappa)d+m$ and applying Stirling's approximation,
\begin{align*}
s(G) = s_G (G) \leq s_G (G,n,\delta ) &= \limsup_{d\to\infty} \frac{1}{d\log d} \log |\GA (G,\delta ,n,d)|
\leq 1 - \frac{1}{m} + \kappa m .
\end{align*}
Since $\kappa$ was an arbitrary number in $(0,1)$ 
we conclude that $s(G)\leq 1 - 1/m$, as desired.
\end{proof}

\section{Free product groups with amalgamation over amenable subgroups}\label{S-amalg groups}

We begin by establishing an upper bound for the sofic dimension of amalgamated free products.
Recall that $S_d$ acts on the set of maps $\sigma : G\to S_d$ by 
$(\gamma\cdot\sigma )_s = \gamma\sigma_s \gamma^{-1}$.

\begin{lemma}\label{L-amalgamated upper}
Let $G_1$ and $G_2$ be countable discrete groups and
$H$ a common subgroup of $G_1$ and $G_2$. Then
\[
s(G_1 *_H G_2 ) \leq s(G_1 ) + s(G_2 ) - 1 + \frac{1}{|H|} .
\]
\end{lemma}

\begin{proof}
We may assume that both $s(G_1 )$ and $s(G_2 )$ are finite. We will also assume that
$s(G_1 *_H G_2 )$ is finite. The same argument with minor modifications can be used to handle 
the case that $s(G_1 *_H G_2 )$ is infinite.
Let $\kappa > 0$.
Since $G_1 \cup G_2$ generates $G_1 *_H G_2$, by Theorem~\ref{T-generating} there are 
nonempty finite sets $E_1 \subseteq G_1$ and $E_2 \subseteq G_2$ such that
$s(G_1 *_H G_2 ) \leq s_{E_1 \cup E_2} (G_1 \cup G_2 ) + \kappa$.
Take nonempty finite sets $F_1 \subseteq G_1$ and $F_2 \subseteq G_2$ such that 
$s_{E_1} (F_1 ) \leq s(G_1 ) + \kappa$ and $s_{E_1} (F_1 ) \leq s(G_2 ) + \kappa$.

Suppose first that $H$ is finite. We may assume that $H\subseteq E_1$ and $H\subseteq E_2$.
Let $d,n\in\Nb$ and $\delta > 0$.
Let $\sigma\in\GA (F_1 \cup F_2 ,\delta ,n,d)$.
Set
\[
V_\sigma = \big\{ c\in\{1 ,\dots ,d\} :
\sigma_{st}(c)=\sigma_s(\sigma_t(c))\text{ for all } s,t\in H\text{ and }
\sigma_s (c) \neq c \text{ for all } s\in H\setminus\{ e \} \big\} .
\]
and observe that $V_\sigma$ can be 
partitioned into $\sigma (H)$-invariant subsets of cardinality $|H|$, on each of which $\sigma |_H$
defines a transitive action of $H$. 
Since the number of partitions of $V_\sigma$ into sets of size $|H|$
is equal to $|V_\sigma|!/(|H|!^{|V_\sigma|/|H|} (|V_\sigma|/|H|)!)$ and $|V_\sigma|/d \to 1$ as $\delta\to 0$
independently of $d$ and $\sigma$, we see using Stirling's approximation that for all sufficiently large $d$ 
the cardinality of $S_d \cdot\sigma |_H$
is at least $d^{d(1-1/|H|-\kappa )}$ for some $\kappa > 0$ which does not depend on $d$ or $\sigma$ with 
$\kappa\to 0$ as $\delta\to 0$. Thus, writing $\Upsilon_\sigma$ for the subgroup
$\{ \gamma\in S_d : \gamma\cdot\sigma |_H = \sigma |_H \}$ of $S_d$, we have, 
for all $d$ larger than some $d_0$ not depending on $\sigma$,
\begin{align*}\tag{$\ast$}
|\Upsilon_\sigma| = \frac{|S_d|}{\big|S_d \cdot\sigma |_H \big|} \leq \frac{d!}{d^{d(1-1/|H|-\kappa )}}
\end{align*}

Now set $\Lambda_i = \{ \sigma |_{E_i} : \sigma\in\GA (F_i ,\delta ,n,d) \}$ for $i=1,2$
and $\Lambda = \{ \sigma |_{E_1 \cup E_2} : \sigma\in\GA (F_1 \cup F_2 ,\delta ,n,d) \}$.
Since for every $\sigma\in\GA (F_1 \cup F_2 ,\delta ,n,d)$
we have $\sigma |_{G_1} \in \GA (F_1 ,\delta ,n,d)$ and $\sigma |_{G_2} \in \GA (F_2 ,\delta ,n,d)$,
and $\GA (F_1 ,\delta ,n,d)$ is invariant under the action of $S_d$,
we can define a map $\Theta : S_d \times \Lambda \to \Lambda_1 \times \Lambda_2$
by $(\gamma,\omega)\mapsto (\gamma\cdot\omega |_{E_1} , \omega |_{E_2} )$.
If $(\gamma ,\omega)$ is a pair in $S_d \times \Lambda$, then every other pair in 
$S_d \times \Lambda$ with the same image as $(\gamma,\omega)$ under $\Theta$
has the form $(\tilde{\gamma} , \tilde{\omega} )$ where $\tilde{\gamma} \in S_d$,
$\tilde{\gamma} \cdot\omega |_H = \omega |_H$, $\tilde{\omega} |_{E_1} = \tilde{\gamma}^{-1} \cdot\omega |_{E_1}$, 
and $\tilde{\omega} |_{E_2} = \omega |_{E_2}$. Note in particular that $\tilde{\omega}$ is determined 
by $\tilde{\gamma}$. It follows by ($\ast$) that for all sufficiently large $d$ 
the inverse image under $\Theta$ of each pair in $\Lambda_1 \times\Lambda_2$ has cardinality at most 
$d! / d^{d(1-1/|H|-\kappa )}$, in which case
\begin{align*}
d! |\Lambda| = |S_d \times \Lambda | 
= \sum_{(\omega_1,\omega_2)\in\Lambda_1 \times\Lambda_2} |\Theta^{-1} (\omega_1,\omega_2)| 
\leq |\Lambda_1| |\Lambda_2| \frac{d!}{d^{d(1-1/|H|-\kappa )}} .
\end{align*}
We consequently obtain
\begin{align*}
s(G_1 *_H G_2 ) &\leq s_{E_1 \cup E_2} (F_1 \cup F_2 ) + \kappa \\
&\leq s_{E_1} (F_1 ) + s_{E_2} (F_2 ) - 1 + \frac{1}{|H|} + 2\kappa \\
&\leq s(G_1 ) + s(G_2 ) - 1 + \frac{1}{|H|} + 4\kappa
\end{align*}
Since $\kappa$ was an arbitrary positive number this yields the desired inequality.

Suppose now that $H$ is infinite. 
By an argument as in the proof of Proposition~\ref{P-inf gd} that produces a collection
of sofic approximations on arbitrarily large finite sets by concatenating together 
sofic approximations on a fixed finite set and conjugating, we can find
a finite set $H_0 \subseteq H$, an $n\in\Nb$, and
a $\delta$ such that, for all sufficiently large $d$, given a $\sigma\in\GA (F_1 ,n,\delta ,d)$
the number of restrictions of elements in $S_d \cdot\sigma$ to $H_0$ is at least
$d^{d(1-\kappa )}$. Assuming that $H_0 \subseteq E_1$ and $H_0 \subseteq E_2$, this yields,
by the same type of argument used above in the case of finite $H$,
\begin{align*}
|\GA (F_1 \cup F_2 ,n,\delta ,d)|_{E_1 \cup E_2}
\leq |\GA (F_1 ,n,\delta ,d)|_{E_1} |\GA (F_2 ,n,\delta ,d)|_{E_2} d^{-d(1-\kappa )} ,
\end{align*}
which again leads to the desired inequality.
\end{proof}
 
Our goal now is to establish the reverse inequality 
for lower sofic dimension under the assumption that the common subgroup is amenable
(Lemma~\ref{L-amalgamated lower}).
 
The following is a perturbative version of the universal property for amalgamated free products.

\begin{lemma}\label{L-universal}
Let $G_1$ and $G_2$ be countable discrete groups and $H$ a common subgroup.
Let $F_1 \subseteq G_1$ and $F_2 \subseteq G_2$ be finite symmetric sets both containing $e$. 
Let $n\in\Nb$ and $\delta > 0$.
Then there are an $m\in\Nb$ and an $\eps > 0$ 
such that if $d\in\Nb$ and $\sigma : G_1 \to S_d$ and $\omega : G_2 \to S_d$ are identity-preserving maps satisfying 
\begin{enumerate}
\item $\| \sigma_s - \omega_s \|_2 < \varepsilon$ for all $s\in H$ which are contained in both $F_1^m$ and $F_2^m$,

\item $\| \sigma_{st} - \sigma_s \sigma_t \|_2 < \varepsilon$ for all $s,t\in F_1^m$, and

\item $\| \omega_{st} - \omega_s \omega_s \|_2 < \varepsilon$ for all $s,t\in F_2^m$,
\end{enumerate}
then there is an identity-preserving map $\rho : G_1 *_H G_2 \to S_d$ satisfying
\begin{enumerate}
\item[(4)] $\| \rho_s - \sigma_s \|_2 < \delta$ for all $s\in F_1$,

\item[(5)] $\| \rho_s - \omega_s \|_2 < \delta$ for all $s\in F_2$, and

\item[(6)] $\| \rho_{s_1 \cdots s_r} - \rho_{s_1}\cdots \rho_{s_r} \|_2 < \delta$
for all $r=2,\dots ,n$ and $s_1 , \dots , s_r \in F_1 \cup F_2$. 
\end{enumerate}
\end{lemma}

\begin{proof}
Suppose to the contrary that no such $m$ and $\varepsilon$ exist. We may assume that
$G_1$ is generated by $F_1$ and $G_2$ is generated by $F_2$. Then for every $k\in\Nb$ we can find
a $d_k \in\Nb$ and identity-preserving maps $\sigma_k : G_1 \to S_{d_k}$ and $\omega_k : G_2 \to S_{d_k}$
such that $\| \sigma_{k,s} - \omega_{k,s} \|_2 < 1/k$ for all $s\in F_1^k \cap F_2^k$, 
$\| \sigma_{k,st} - \sigma_{k,s}\sigma_{k,t} \|_2 < 1/k$ for all $s,t\in F_1^k$,
and $\| \omega_{k,st} - \omega_{k,s}\omega_{k,t} \|_2 < 1/k$ for all $s,t\in F_2^k$ but there is no
identity-preserving map $\rho : G_1 *_H G_2 \to S_{d_k}$ such that $\rho |_{G_1} = \sigma_k$,
$\rho |_{G_2} = \omega_k$, and
$\| \rho_{s_1 \cdots s_r} - \rho_{s_1}\cdots \rho_{s_r} \|_2 < \delta$
for all $r=2,\dots ,n$ and $s_1 , \dots , s_r \in F_1 \cup F_2$. 
Take a nonprincipal ultrafilter $\cU$ on $\Nb$. Write $\sN$ for the normal
subgroup of $\sG = \prod_{k=1}^\infty S_{d_k}$ consisting of all sequences $(g_k )_k$ such that
$\lim_{k\to\cU} \| g_k - \id \|_2 = 0$. Let $\pi : \sG \to \sG / \sN$ be the quotient map.
Define $\sigma' : G_1 \to \sN$ by $\sigma'_s = \pi ((\sigma_{k,s} )_k )$ and
$\omega' : G_2 \to \sN$ by $\omega'_s = \pi ((\omega_{k,s} )_k )$. Then 
$\sigma'$ and $\omega'$ are homomorphisms since $F_1$ and $F_2$ are symmetric and both contain $e$, 
and they agree on $H$. It follows by the universal property of the amalgamated free
product there is a homomorphism $\gamma : G_1 *_H G_2 \to \sG / \sN$ such that
$\gamma |_{G_1} = \sigma'$ and $\gamma |_{G_2} = \omega'$. Choose a lift 
$\tilde{\gamma} : G_1 *_H G_2 \to \sG$ of $\gamma$, which we may take to be identity-preserving. 
Then for some $m\in\Nb$ the composition
$\rho = \pi_m \circ\gamma$, where $\pi_m : \sG = \prod_{k=1}^\infty S_{d_k} \to S_{d_m}$ is the projection,
satisfies $\| \rho_s - \sigma_{m,s} \|_2 < \delta$ for all $s\in F_1$,
$\| \rho_s - \omega_{m,s} \|_2 < \delta$ for all $s\in F_2$, and
$\| \rho_{s_1 \cdots s_r} - \rho_{s_1}\cdots \rho_{s_r} \|_2 < \delta$
for all $r=2,\dots ,n$ and $s_1 , \dots , s_r \in F_1 \cup F_2$, a contradiction.
\end{proof}

Next we record a special case of Lemma~4.5 of \cite{KerLi10a}, which is based 
on the quasitiling theorem of Orntein and Weiss \cite{OrnWei87}.
The numbers $\lambda_1 , \dots , \lambda_k$ and
condition (3) below do not appear in the statement of Lemma~4.5 of \cite{KerLi10a},
but the proof of the latter is easily seen to yield this stronger version.

For a finite set $D$ and an $\varepsilon \geq 0$, we say that
a collection $\{ A_i \}_{i\in I}$ of subsets of $D$ is 
{\it $\varepsilon$-disjoint} if there exist pairwise disjoint sets $\widehat{A}_i \subseteq A_i$ such that
$|\widehat{A}_i | \geq (1-\varepsilon ) | A_i |$ for all $i\in I$. A set $A\subseteq D$ is said
to {\it $\eps$-cover} $D$ if $|A| \geq\eps |D|$.

\begin{lemma} \label{L-qt}
Let $G$ be a countable discrete group. 
Let $0<\eps<1$. Then there are a $k\in \Nb$, numbers $0< \lambda_1 , \dots , \lambda_k \leq 1$
with $1-\eps < \lambda_1 + \cdots + \lambda_k \leq 1$, 
and an $\eta >0$ such that whenever $e\in T_1\subseteq T_2\subseteq \cdots \subseteq T_k$ 
are finite subsets of $G$ with $|(T_{j-1}^{-1}T_j) \setminus T_j|\le \eta |T_j|$ for $j=2, \dots, k$
there exists a finite set $K\subseteq G$ containing $e$ and a $\delta > 0$ such that
for every $d\in \Nb$ and every map $\sigma: G\rightarrow S_d$ satisfying
\begin{enumerate}
\item $\| \sigma_{st} - \sigma_s \sigma_t \|_2 < \delta$ for all $s,t\in K$, and
 
\item $\tr (\sigma_s ) < \delta$ for all $s\in K^{-1} K\setminus \{ e \}$
\end{enumerate}
there exist $C_1, \dots, C_k \subseteq \{ 1,\dots,d\}$ such that
\begin{enumerate}
\item[(3)] $\big| |T_j||C_j |/d - \lambda_j \big| < \eps$ for every $j=1,\dots,d$,

\item[(4)] for every $j=1, \dots, k$ and $c\in C_j$, the map $s\mapsto \sigma_s(c)$ 
from $T_j$ to $\sigma(T_j)c$ is bijective,

\item[(5)] the sets $\sigma(T_1)C_1, \dots, \sigma(T_k)C_k$ are pairwise disjoint
and the family $\bigcup_{j=1}^k\{\sigma(T_j)c: c\in C_j\}$ is $\eps$-disjoint 
and $(1-\eps )$-covers $\{1, \dots, d\}$.
\end{enumerate}
\end{lemma}

\begin{lemma}\label{L-single}
For every $\eps > 0$ we have
\begin{align*}
\lim_{d\to\infty} \,\min_{A\in I_d} \frac{1}{d!^2} | \{ (U,V)\in S_d \times S_d : \tr (UAV^* ) < \varepsilon \} | = 1 .
\end{align*}
\end{lemma}

\begin{proof}
Let $\eps > 0$.
Since the map $(U,V)\mapsto V^* U$ from $S_d \times S_d$ to $S_d$ is $d!$-to-$1$ 
and $\tr (UAV^* ) = \tr (V^* UA)$ for all $A\in I_d$, it is enough to prove that 
\begin{align*}\tag{$\ast$}
\lim_{d\to\infty} \,\min_{A\in I_d} \frac{1}{d!} | \{ U\in S_d : \tr (UA) < \varepsilon \} | = 1 .
\end{align*}
Let $A\in I_d$ for some $d\in\Nb$. Take a $W\in S_d$ such that $AW = AA^*$.
If $U\in S_d$ satisfies $\tr (U) < \varepsilon$ then, 
writing $\delta_1 , \dots , \delta_d$ for the standard basis vectors of $\Cb^d$,
\begin{align*}
\tr (WUA) &= \tr (UAW) = \tr (UAA^* ) \\
&\hspace*{5mm} \ = \frac{1}{d} \sum_{k=1}^d \langle UA^* A\delta_k , \delta_k \rangle \leq 
\frac{1}{d} \sum_{k=1}^d \langle U\delta_k , \delta_k \rangle = \tr (U) < \varepsilon .
\end{align*}
Since the map $U\mapsto WU$ is a bijection from $S_d$ to itself, we obtain
\[
| \{ U\in S_d : \tr (UA) < \varepsilon \} |
\geq | \{ U\in S_d : \tr (U) < \varepsilon \} | .
\]
Now it is well known that, for a fixed $k\in\Nb$, the proportion of permutations of $\{ 1,\dots ,d \}$
which have exactly $k$ fixed points tends to $e^{-1} /k!$ as $d\to\infty$ (see \cite{Rio58}, Chap.\ 3, Sect.\ 5). 
It follows that $\lim_{d\to\infty} | \{ U\in S_d : \tr (U) < \varepsilon \} |/d! = 1$, yielding ($\ast$).
\end{proof}

The following is a multiparameter version of Theorem~2.1 in \cite{ColDyk10}.

\begin{lemma}\label{L-ColDyk10}
Let $n\in\Nb$. Let $\ell\in \{1,\dots ,n \}$ and let
$\rho : \{ 1,\dots ,2n\} \to \{ 1,2,\dots ,\ell \}$ be a surjective map.
For $k=1,\ldots,2n$ and $d\in\Nb$ let $A_k^{(d)} \in I_d$. Then there are $C_n ,D_n > 0$
depending only on $n$ such that
\begin{align*}
\frac{1}{d!^\ell}
\sum_{U_1,\dots,U_\ell \in S_d} \tr \big(A_1^{(d)}(U_{\rho(1)}A_2^{(d)}U_{\rho(2)}^*)
A_3^{(d)}(U_{\rho(3)}A_4^{(d)}U_{\rho(4)}^*) \cdots A_{2n-1}^{(d)}(U_{\rho(2n-1)}
A_{2n}^{(d)}U_{\rho(2n)}^*)\big) \hspace*{5mm} \\
\ < C_n\max_{k=1,\dots,2n}\tr_d \big(A_k^{(d)}\big) + \frac{D_n}{d} .
\end{align*}
\end{lemma}

\begin{proof}
Using independence with respect to the variables $U_1 , \dots , U_\ell$ 
and an observation in the first part of the proof of Theorem~2.1 in \cite{ColDyk10}, 
for all $1\leq i_1,i_2,\dots,i_{4n} \leq d$ we have, writing $U_k = (u^{(k)}_{i,j} )_{i,j}$,
\begin{align*}
\lefteqn{\frac{1}{d!^\ell} \sum_{U_1 ,\dots, U_\ell \in S_d} u^{(\rho (1))}_{i_1,i_2} u^{(\rho (2))}_{i_4,i_3}
u^{(\rho(3))}_{i_5,i_6} u^{(\rho(4))}_{i_8,i_7} 
\cdots u^{(\rho(2n-1))}_{i_{4n-3,4n-2}} u^{(\rho(2n))}_{i_{4n,4n-1}}} \hspace*{30mm} \\
\hspace*{30mm} &= \prod_{m=1}^\ell \frac{1}{d!}\sum_{U\in S_d} \hspace*{1mm} 
\prod_{k\in \rho^{-1} (m)} 
\begin{cases}
u^{(m)}_{i_{2k-1},i_{2k}} &  k \text{ odd} \\
u^{(m)}_{i_{2k},i_{2k-1}} &  k \text{ even} 
\end{cases} \\
&= 
\begin{cases}
\prod_{m=1}^\ell \frac{(d-|r_m|)!}{d!} & \text{if } r_m = s_m \text{ for all } m=1,\dots ,\ell \\
0 & \text{otherwise}
\end{cases}
\end{align*}
where $r_m$ and $s_m$ are the partitions of $\rho^{-1} (m)$ such that $k$ and $k'$ belong to 
the same element of $r_m$ if and only if the first lower indices of the corresponding factors 
in the middle line agree, and belong to the same element of $s_m$ if and only if the second lower indices 
of these corresponding factors agree.
For a finite set $F$ write $\sP (F)$ for the set of all partitions of $F$, and for $r\in \sP (\{ 1,\dots ,2n \})$ write 
$I(r)$ for the subset of $\{ 1,\dots ,d \}^{2n}$ depending on $r$ that appears in the proof of Theorem~2.1 in \cite{ColDyk10}.
Writing $A_k^{(d)} = (a^{(k)}_{i,j} )_{i,j}$, we then have
\begin{align*}
\lefteqn{\frac{1}{d!^\ell}
\sum_{U_1,\dots,U_\ell \in S_d} \tr \big(A_1^{(d)}(U_{\rho(1)}A_2^{(d)}U_{\rho(2)}^*)
A_3^{(d)}(U_{\rho(3)}A_4^{(d)}U_{\rho(4)}^*) \cdots A_{2n-1}^{(d)}(U_{\rho(2n-1)}A_{2n}^{(d)}
U_{\rho(2n)}^*)\big)} \hspace*{15mm} \\
\hspace*{10mm} &\leq \frac1d \sum_{r_1\in \sP (\rho^{-1} (1))} \cdots \sum_{r_\ell \in \sP (\rho^{-1} (\ell))}
\prod_{m=1}^\ell \frac{(d-|r_m|)!}{d!}
\sum_{i\in I(r_1 \cup\cdots\cup r_\ell )} a^{(1)}_{i_1,i_2} a^{(2)}_{i_3,i_4} \cdots a^{(2n)}_{i_{4n-1},i_{4n}} \\
&\leq 2^{2n} \sum_{r\in \sP (\{ 1,\dots,2n\})} d^{-|r|-1} \sum_{i\in I(r)} a^{(1)}_{i_1,i_2} a^{(2)}_{i_3,i_4} \cdots a^{(2n)}_{i_{4n-1},i_{4n}} .
\end{align*}
One can now estimate the last expression in the above display as in the proof of Theorem~2.1 in \cite{ColDyk10} 
to obtain the result.
\end{proof}

The following result is a standard sort of strengthening of Lemma~\ref{L-ColDyk10}
based on concentration results of Gromov and Milman \cite{GroMil83}. 

\begin{lemma}\label{L-conc}
Let $n\in\Nb$ and $\varepsilon > 0$. Let $\ell\in \{1,\dots ,n \}$ and let
$\rho : \{ 1,\dots ,2n\} \to \{ 1,2,\dots ,\ell \}$ be a surjective map.
For $k=1,\ldots,2n$ and $d\in\Nb$ let $A_k^{(d)} \in I_d$. 
Let $C_n > 0$ be as in Lemma~\ref{L-ColDyk10} and set
\begin{align*}
\Omega_{d,\varepsilon} =\bigg\{
(U_1,\dots,U_\ell )\in S_d^\ell : \tr_d \big(A_1^{(d)}(U_{\rho(1)}A_2^{(d)}U_{\rho(2)}^*)A_3^{(d)}
(U_{\rho(3)}A_4^{(d)}U_{\rho(4)}^*)\hspace*{30mm} \\
\hspace*{20mm} \ \cdots A_{2n-1}^{(d)}(U_{\rho(2n-1)}A_{2n}^{(d)}U_{\rho(2n)}^*)\big) 
< C_n\max_{k=1,\dots,2n}\tr_d \big(A_k^{(d)}\big) + \eps \bigg\} .
\end{align*}
Then $\lim_{d\to\infty} |\Omega_{d,\varepsilon} |/d!^\ell = 1$.
\end{lemma}

\begin{proof}
From Lemma~\ref{L-ColDyk10} we have
\begin{align*}
\frac{1}{d!^\ell} \sum_{U_1,\dots,U_\ell \in S_d}
\tr \big(A_1^{(d)}(U_{\rho(1)}A_2^{(d)}U_{\rho(2)}^*)A_3^{(d)}(U_{\rho(3)}A_4^{(d)}U_{\rho(4)}^*) 
\cdots A_{2n-1}^{(d)}(U_{\rho(2n-1)}A_{2n}^{(d)}U_{\rho(2n)}^*)\big) \hspace*{5mm} \\
< C_n\max_{k=1,\dots,2n}\tr_d \big(A_k^{(d)}\big) + \frac{D_n}{d}  
\end{align*}
for certain constants $C_n$ and $D_n$ depending only on $n$.
Note that 
\[
0\le \tr_d\big(A_1(U_1A_2U_2^*)A_3(U_3A_4U_4^*)\cdots A_{2n-1}(U_{2n-1}A_{2n}U_{2n}^*)\big)\le1
\]
for all $U_1, \dots , U_{2n} \in S_d$ and all $A_j\in I_d$.
Set $f(d)=\max_{k=1,\dots,2n}\tr_d(A_k^{(d)})$.
From the first display above we get 
$|\Omega_{d,\varepsilon} |/d!^\ell \ge\frac{\varepsilon/2}{C_n f(d)+\varepsilon}$ 
for all sufficiently large $d$.
Since $f(d)\le 1$ for all $d$, 
we infer that $\liminf_{d\to\infty} |\Omega_{d,\varepsilon} |/d!^\ell > 0$ for all $\varepsilon>0$.

If on $S_d$ we express the normalized Hamming distance
\[
\rho_d(U,V)=\frac12\tr_d(| U-V |^2)=\frac12\| U-V \|_2^2
\]
in terms of the $2$-norm and use the Cauchy-Schwarz and triangle inequalities, then we find that for every $\delta>0$
there is an $\eta>0$ such that $N_\eta(\Omega_{d,\varepsilon})\subseteq\Omega_{d,\varepsilon+\delta}$,
where $N_\eta(\cdot)$ denotes the $\eta$-neighbourhood with respect to the maximum of the coordinatewise
normalized Hamming distances.
Gromov and Milman observe in Remark~3.6 of~\cite{GroMil83} that results of Maurey~\cite{Mau79} imply that
the symmetric groups $S_d$ equipped with normalized Hamming metrics and uniform probability measures
form a L\'{e}vy family as $d\to\infty$ (see Chapter~7 of \cite{MilSch86}). Since a finite product of L\'{e}vy families
is again a L\'{e}vy family (see Section~2 of \cite{GroMil83}), 
we conclude that $\lim_{d\to\infty} |N_\eta(\Omega_{d,\varepsilon})|/d!^\ell = 1$ for all $\eta>0$, 
yielding the lemma.
\end{proof}

\begin{lemma}\label{L-free prod}
Let $n,m\in\Nb$ and $\eps > 0$. Let $\{ Y_1 , \dots , Y_\ell \}$ be a partition of $\{ 1,\dots ,m \}$. 
Then there is a $\delta > 0$ such that the following holds. For $d\in\Nb$
fix an identification of $M_{md}$ with $M_m \otimes M_d$ which pairs off 
matrix units with tensors products of matrix units, and 
for each $k=1,\ldots,2n$ let $A_k^{(d)}$ be a partial permutation matrix in $M_{md}$ such that,
writing $A_k^{(d)} = \sum_{i,j=1}^m E_{i,j} \otimes A_{k,i,j}^{(d)} \in M_m \otimes M_d$
where the $E_{i,j}$ are matrix units, one has
\[
\max_{q=1,\dots ,\ell} \,\max_{(i,j)\in Y_q \times Y_q} \tr_d (A_{k,i,j}^{(d)}) < \delta .
\]
Write $\sX_d$ for the set of all permutation matrices in $M_{md}$ of the form 
$\sum_{r=1}^\ell P_r \otimes U_r \in M_m \otimes M_d$ where $P_r$ is the characteristic function
of $Y_r$ viewed as a diagonal matrix in $M_m$. Set
\begin{align*}
\Upsilon_{d,\eps}
&=\bigg\{
U\in \sX_d : \tr_{md} \bigg( \prod_{k=1}^n A_{2k-1} ^{(d)}\big(UA_{2k}^{(d)} U^* \big)\bigg) < \eps \bigg\}.
\end{align*}
Then $\lim_{d\to\infty} |\Upsilon_{d,\eps} |/|\sX_d | =1$.
\end{lemma}

\begin{proof}
Since the map $(U,V)\mapsto V^* UV$ from $\sX_d \times \sX_d$ to $\sX_d$ is $|\sX_d |$-to-$1$,
it suffices to prove that the set
\begin{align*}
\Lambda_{d,\eps}
&=\bigg\{
(U,V)\in \sX_d \times\sX_d : \tr_{md} \bigg( \prod_{k=1}^n A_{2k-1} ^{(d)} \big(V^* UVA_{2k}^{(d)} (V^* UV)^* \big)\bigg) < \eps \bigg\}.
\end{align*}
satisfies $\lim_{d\to\infty} |\Lambda_{d,\eps} |/|\sX_d |^2 =1$.
Let $\delta > 0$, to be specified.
Define $h : \{ 1,\dots ,m \} \to \{ 1,\dots ,\ell \}$ so that $i\in Y_{h(i)}$ for every $i=1,\dots ,m$.  
Write $\sY_d$ for the set of all $\sum_{r=1}^\ell P_r \otimes V_r \in\sX_d$ such that
$\tr (V_{h(i)} A_{k,i,j}^{(d)} V_{h(j)}^* ) < \delta$ for all $k=1,\dots ,2n$ 
and $i,j=1,\dots ,m$ with $h(i)\neq h(j)$.
By multiple applications of Lemma~\ref{L-single} we 
infer that $\lim_{d\to\infty} |\sY_d |/|\sX_d | = 1$. 

Let $V=\sum_{r=1}^\ell P_r \otimes V_r \in \sY_d$. Set
\begin{align*}
\sX_{d,V,\varepsilon} = \bigg\{ U\in\sX_d : 
\tr_{md} \bigg( \prod_{k=1}^n VA_{2k-1} ^{(d)}V^* \big(UVA_{2k}^{(d)}V^* U^* \big)\bigg) < \eps \bigg\}
\end{align*}
Let $U=\sum_{r=1}^\ell P_r \otimes U_r \in\sX_d$. 
The $(i,i)$ entry of $\prod_{k=1}^n VA_{2k-1}^{(d)}V^* (UVA_{2k}^{(d)}V^* U^*)$ in 
$M_m (M_d ) \cong M_m \otimes M_d$ is equal to the sum over all $(i_1 ,\dots ,i_{2n-1} )\in \{ 1,\dots ,m\}^{2n-1}$
of the products
\begin{gather*}
\prod_{k=1}^n V_{h(i_{2k-2})} A_{2k-1,i_{2k-2} ,i_{2k-1}}^{(d)} 
V_{h(i_{2k-1} )}^* \big( U_{h(i_{2k-1} )} V_{h(i_{2k-1} )} 
A_{2k,i_{2k-1} ,i_{2k}}^{(d)} V_{h(i_{2k} )}^* U_{h(i_{2k} )}^* \big) 
\end{gather*}
where $i_0 = i_{2n} = i$. Now given $k\in\{ 1,\dots,2n \}$ and $i,j\in\{ 1,\dots,m\}$, if $h(i)\neq h(j)$ then 
$\tr (V_{h(i)} A_{k,i,j}^{(d)} V_{h(j)}^* ) < \delta$ since $V\in\sY_d$, while if $h(i)=h(j)$ then by hypothesis
\begin{align*}
\tr (V_{h(i)} A_{k,i,j}^{(d)} V_{h(j)}^* ) = \tr (A_{k,i,j}^{(d)} ) < \delta .
\end{align*}
It follows by multiple applications of Lemma~\ref{L-conc} that if $\delta$ 
is small enough as a function of $\eps$, $m$, and $n$
then we have $\lim_{d\to\infty} \mn_{V\in\sY_d} |\sX_{d,V,\varepsilon} | / |\sX_d | = 1$.
Since $\bigcup_{V\in\sY_d} \{ (U,V) : U\in\sX_{d,V,\varepsilon} \}$ is contained in $\Lambda_{d,\eps}$ and
$\lim_{d\to\infty} |\sY_d |/|\sX_d | = 1$ from above, we conclude that $\lim_{d\to\infty} |\Lambda_{d,\eps} |/|\sX_d |^2 = 1$,
as desired.
\end{proof}

Recall that for sets $W\subseteq X$ and $Z$ and a collection $\sY$ of
maps from $X$ to $Z$ we write $|\sY |_W$ for the cardinality of the the set of restrictions
of elements of $\sY$ to $W$.

\begin{lemma}\label{L-amalgamated lower}
Let $G_1$ and $G_2$ be countable discrete groups with common amenable subgroup $H$.
Then 
\[
\slower (G_1 *_H G_2 ) \geq \slower (G_1 ) + \slower (G_2 ) - 1 + \frac{1}{|H|} .
\]
\end{lemma}

\begin{proof}
To avoid ambiguities we will view $G_1$ and $G_2$ as subgroups of $G_1 *_H G_2$,
in which case $G_1 \cap G_2 = H$ (see \cite[Thm.\ 11.67]{Rot95}).
Let $\theta > 0$.
Then there exist finite sets $E_1 \subseteq G_1$ and $E_2 \subseteq G_2$ such that
$\slower_{E_1} (G_1 ) \geq\slower (G_1 ) - \theta$ and $\slower_{E_1} (G_2 ) \geq\slower (G_2 ) - \theta$.
Take  finite symmetric sets $F_1 \subseteq G_1$ and $F_2 \subseteq G_2$ each containing $e$
such that $\slower_{E_1 \cup E_2} (G_1 \cup G_2 ) \geq\slower_{E_1 \cup E_2} (F_1 \cup F_2 ) - \theta$.
We may assume that $E_2 \subseteq F_2$.

Let $\kappa > 0$.
Let $\delta > 0$, to be determined as a function of $\kappa$, and let $n$ be an integer greater than $2$. 
Set $\tilde{F}_1= F_1 \cup ((F_1 \cup F_2 )^n \cap H)$ and $\tilde{F}_2= F_2 \cup ((F_1 \cup F_2 )^n \cap H)$.
By Lemma~\ref{L-universal} there are an integer $M\geq n$ and a $\delta' > 0$ 
such that if $d\in\Nb$ and $\sigma : G_1 \to S_d$ and $\omega : G_2 \to S_d$ are 
identity-preserving maps satisfying 
$\| \sigma_s - \omega_s \|_2 < \delta'$ for all $s\in \tilde{F}_1^M \cap F_2^M$,
$\| \sigma_{st} - \sigma_s \sigma_t \|_2 < \delta'$ for all $s,t\in \tilde{F}_1^M$, and
$\| \omega_{st} - \omega_s \omega_t \|_2 < \delta'$ for all $s,t\in F_2^M$,
then there is an identity-preserving map $\Omega : G_1 *_H G_2 \to S_d$ for which
$\| \Omega_s - \sigma_s \|_2 < \delta /(16n)$ for all $s\in \tilde{F}_1^n$,
$\| \Omega_s - \omega_s \|_2 < \delta /(16n)$ for all $s\in F_2^n$, and
$\| \Omega_{s_1 \cdots s_r} - \Omega_{s_1} \cdots \Omega_{s_r} \|_2 < \delta /(16n)$
for all $r=2,\dots ,n$ and $s_1 , \dots , s_r \in \tilde{F}_1^n \cup F_2^n$. 
We may assume that $\delta' \leq \delta /(8n)$.

Set $F=(F_1^{-1} F_1 )^M \cap (F_2^{-1} F_2 )^M$. 

We will first assume that $H$ is infinite in the following part of the argument, 
and then explain afterward how to handle the case when $H$ is finite.   
Let $\eps>0$, to be determined as a function of $\delta$, $\delta'$, $n$, and $|F|$.
Let $\eta>0$ be such that $\eta\leq\eps$, to be further specified.
Since $H$ is amenable, we can apply Lemma~\ref{L-qt} and pass from
an $\eps$-disjoint family in its conclusion to a genuinely disjoint family consisting
of translates of tiles of proportionally slightly smaller size than the original tiles 
to obtain the following: there exist finite subsets 
$e\in T_1\subseteq T_2\subseteq \cdots \subseteq T_k$ of $H$ with
$|FT_j \Delta T_j |/|T_j| < \eta$ for $j=1,\dots, k$,
rational numbers $0<\lambda_1 , \dots , \lambda_k \leq1$ with $1-\eps < \lambda_1 + \cdots + \lambda_k \leq 1$,
a finite set $K\subseteq H$ containing $e$, and a $\beta > 0$ such that, writing $m$ for the smallest
positive integer such that for each $j=1,\dots ,k$ the number $\lambda_j m$ is an integer which $|T_j|!$ divides,
for every $d\in \Nb$ and every map $\sigma: G\rightarrow S_{md}$ satisfying
\begin{enumerate}
\item $\| \sigma_{st} - \sigma_s \sigma_t \|_2 < \beta$ for all $s, t\in K$, and 

\item $\tr_{md} (\sigma_s ) < \beta$ for all $s\in K^{-1} K\setminus \{ e \}$
\end{enumerate}
there exist $r_1 ,\dots ,r_k \in\Nb$ and, for each $j=1,\dots,k$, 
sets $C_{j,1} , \dots, C_{j,r_j} \subseteq \{ 1,\dots,md\}$ and $T_{j,1} ,\dots, T_{j,r_j} \subseteq T_j$ 
with $|T_{j,r} \Delta T_j |/|T_j| < \eps$ for $r=1,\dots,r_j$ such that
\begin{enumerate}
\item[(3)] $|\sigma (T_{j,1} )C_{j,1} \cup\cdots\cup \sigma (T_{j,r_j} )C_{j,{r_j}}| = \lambda_j md$ for every $j=1,\dots ,k$,

\item[(4)] for every $j=1, \dots, k$, $r=1,\dots,r_j$, and $c\in C_{j,r}$, the map $s\mapsto \sigma_s(c)$ 
from $T_{j,r}$ to $\sigma(T_{j,r})c$ is bijective,

\item[(5)] the sets $\sigma(T_{j,r})c$ for $j=1, \dots, k$, $r=1,\dots,r_j$, and $c\in C_{j,r}$ are pairwise disjoint.
\end{enumerate}

Let $N$ be a positive integer that $m\cdot\prod_{j=1}^k |T_j|!$ divides. 
Then the cardinality of every subset of each $T_j$ divides $N$.

Choose $\ell\in\Nb$ such that $(\ell - b)/\ell > \max (1 - \eps ,1-\kappa )$ where $b$ is the sum of the cardinalities
of the power sets of the tiles $T_1 , \dots , T_k$. Then $b$ is a bound of the total number of
tiles $T_{j,r}$ which can appear for any $\sigma$ as above. 
Set $\ell' = \ell - b$. We will also assume that $\ell$ is sufficiently large as a function of
$n$ and $\delta$ for a purpose to be described below.

Let $\delta'' > 0$, to be determined as a function of $\delta'$. 
Let $\delta''' > 0$ be such that $\ell N \delta''' < \delta'' /2$.
Let $\delta'''' > 0$ be smaller than $\beta$ and $\delta /(4n)$, to be further specified as a function of $\ell N\delta'''$.

Let $d\in\Nb$.
For brevity write $L_1$ for $(\tilde{F}_1^{-1} \tilde{F}_1)^M \cup (K^{-1} K) \cup T_k$ 
and $L_2$ for $(\tilde{F}_2^{-1} \tilde{F}_2)^M \cup (K^{-1} K) \cup T_k$,
and $\sY_1$ for $\SA (L_1 ,n,\delta'''' ,\ell Nd)$ and $\sY_2$ for $\SA (L_2 ,n,\delta'''' ,\ell Nd)$.

Let $\sigma\in\sY_1$ and $\omega\in\sY_2$. We apply our invocation of Lemma~\ref{L-qt}
first to $\sigma$ to get tiles $T_{j,r}$ and sets $C_{j,r}$ from which, in view of our definition of $\ell'$, 
we can produce sets $\tilde{S}_1, \dots , \tilde{S}_{\ell'} \subseteq H$ and $D_1 ,\dots , D_{\ell'} \subseteq \{ 1,\dots ,\ell Nd \}$
and a function $h:\{ 1,\dots ,\ell' \} \to \{ 1,\dots ,k\}$
such that the sets $\sigma (\tilde{S_i})D_i$ for $i=1,\dots,\ell'$ are pairwise disjoint
and, for each $i=1,\dots,\ell'$,
\begin{enumerate}
\item[(6)] $\tilde{S}_i$ is equal to $T_{h(i),r}$ for some $r$ and $D_i$ is a subset of the 
corresponding $C_{h(i),r}$, and 

\item[(7)] $|\tilde{S}_i ||D_i | = Nd$.
\end{enumerate}
Note in particular that $|\tilde{S}_i \Delta T_{h(i)} |/|T_{h(i)} | < \eps$ for $i=1,\dots,\ell'$.

Now we apply our invocation of Lemma~\ref{L-qt} to $\omega$ to get tiles like
the $T_{j,r}$ and $\tilde{S}_i$ above. In this case the tiles will be different,
but since the numbers $\lambda_i$ are independent of $\sigma$ and $\omega$ 
we can pair off the tiles $\tilde{S}_i$ with their counterparts for $\omega$ in a way
that enables us to construct a permutation $W\in S_{\ell Nd}$ such that,
setting $\omega' = W\cdot\omega \in \sY_2$,
there exist an $N' \in\Nb$ with $N'/N > 1-3\eps$ and tiles
$S_1 \subseteq \tilde{S}_1 , \dots , S_{\ell'} \subseteq \tilde{S}_{\ell'}$ 
with $|S_i |/|\tilde{S}_i | = N'/N$ for each $i=1,\dots, \ell'$ (which we obtain by intersecting 
each $\tilde{S}_i$ with its counterpart for $\omega$ and proportionally slightly shrinking these intersections
to achieve the relative cardinality condition, as is possible if we assume $N$ to be large enough) 
such that
\begin{enumerate}
\item[(8)] $|S_i ||D_i | = N'd$ for every $i=1,\dots, \ell'$,

\item[(9)] for every $i=1,\dots,\ell'$ and $c\in D_i$ the maps $s\mapsto \sigma_s(c)$ and
$s\mapsto\omega'_s (c)$ from $S_i$ to $\sigma(S_i)c$ agree and are bijective,

\item[(10)] $|S_i \Delta T_{h(i)} |/|T_{h(i)}| < 4\eps$ for every $i=1,\dots ,\ell'$,

\item[(11)] the sets $\sigma(S_i)c$ for $i=1,\dots,\ell'$ and $c\in D_i$ are pairwise disjoint.
\end{enumerate}
It follows from (3) that if $\eps$ is sufficiently small so that $|T_{h(i)} |/|S_i| \leq 2$ for $i=1,\dots,\ell'$
then we will have, for each $i=1,\dots ,\ell'$,
\begin{align*}\tag{$\ast$}
\frac{|FS_i \Delta S_i |}{|S_i|} &\leq \frac{|T_{h(i)} |}{|S_i|} 
\bigg( \frac{|FS_i \Delta FT_{h(i)} |}{|T_{h(i)}|}  + \frac{|FT_{h(i)} \Delta T_{h(i)} |}{|T_{h(i)}|}  
+ \frac{|T_{h(i)} \Delta S_i |}{|T_{h(i)}|} \bigg) \\
&< 2(4|F|\eps + \eta +4\eps ) < (8|F| + 10)\eps .
\end{align*}

For $i=1,\dots ,\ell'$ set $N_i = N'/|S_i|$, which is equal to $N/|\tilde{S_i}|$ and hence is integral
by our choice of $N$,
and write $Z_i$ for the set $\sigma (S_i)D_i = \omega' (S_i)D_i$, which has cardinality $N'd$.
We identify the subalgebra $\cB (\ell^2 (Z_i ))\cong M_{N'd}$ of 
$\cB (\ell^2 (\{ 1,\dots ,\ell Nd \} )) \cong M_{\ell Nd}$ 
with $\cB (\ell^2 (S_i)) \otimes \cB (\ell^2 (D_i )) \cong M_{|S_i|} \otimes M_{N_i d}$
in such a way that the elements $\omega'_s$ for $s\in F$ act in a manner consistent with (9).
Fix an identification $M_{N_i d} \cong M_{N_i} \otimes M_d$
under which matrix units pair with tensor products of matrix units. 
This gives us an identification 
$\cB (\ell^2 (Z_i ))\cong M_{|S_i|} \otimes M_{N_i} \otimes M_d$.
Writing $P_i$ for the characteristic function of $Z_i$ viewed as a diagonal matrix in $M_{\ell Nd}$,
for $t\in (F_1^n \cup F_2^n ) \setminus H$ we express $P_i \sigma_t P_i$ (in the case $t\in G_1$) or
$P_i \omega_t' P_i$ (in the case $t\in G_2$) as 
\[
\sum_{a,b\in S_i} \sum_{p,q=1}^{N_i} E_{a,b} \otimes E_{p,q} \otimes V^{(t)}_{i,a,b,p,q} 
\in M_{|S_i|} \otimes M_{N_i} \otimes M_d \cong \cB (\ell^2 (Z_i ))
\]
where the $E_{a,b}$ and $E_{p,q}$ are matrix units.
Write $R$ for the orthogonal projection of $\ell^2 (\{ 1,\dots ,\ell Nd \} )$ onto
$\bigoplus_{i=1}^{\ell'} \ell^2 (Z_i )$ and $\sX_{d,\sigma,\omega,W}$ for the set of all 
$U\in S_{\ell Nd}$ of the form $(1-R) + \sum_{i=1}^{\ell'} \sum_{p=1}^{N_i} P_i \otimes E_{p,p} \otimes U_{i,p}$
where $P_i \otimes E_{p,p} \otimes U_{i,p} \in M_{|S_i|} \otimes M_{N_i} \otimes M_d \cong \cB (\ell^2 (Z_i ))$.

Let $t\in (F_1^n \cup F_2^n )\setminus H$,
$i\in\{ 1,\dots ,\ell' \}$, and $p\in \{ 1,\dots N_i \}$, and let $a,b\in S_i$.
We will verify that $\tr_d (V^{(t)}_{i,a,b,p,p} ) < \delta''$. 
Suppose first that $t\in F_1^n \setminus H$.
Since $\sigma_b \sigma_a^{-1}$ sends $\sigma_a (D_i)$ to $\sigma_b (D_i)$, writing
$Q$ for the projection $E_{b,b} \otimes E_{p,p} \otimes 1_{M_d}$ we see that
\begin{align*}
E_{b,b} \otimes E_{p,p} \otimes V^{(t)}_{i,a,b,p,p}
= \sigma_b \sigma_a^{-1} \big( E_{a,b} \otimes E_{p,p} \otimes V^{(t)}_{i,a,b,p,p} \big) 
= Q\sigma_b \sigma_a^{-1} \sigma_t Q
\end{align*}
Also, since $\sigma\in\SA (L_1 ,n,\delta''''' ,\ell Nd)$, $S_i \subseteq T_{h(i)} \subseteq L_1$, $t\in L_1$, and $n\geq 3$,
we have 
\begin{align*}
\| \sigma_{ba^{-1} t} - \sigma_b \sigma_a^{-1} \sigma_t \|_2
&\leq \| \sigma_{ba^{-1} t} - \sigma_b \sigma_{a^{-1}} \sigma_t \|_2 + \| \sigma_b (\sigma_{a^{-1}} - \sigma_a^{-1} )\sigma_t \|_2 \\
&< \delta'''' + \| \sigma_a \sigma_{a^{-1}} - 1 \|_2 < 2\delta'''' .
\end{align*}
Since $Q\sigma_{ba^{-1} t}Q = E_{b,b} \otimes E_{p,p} \otimes V^{(ba^{-1} t)}_{i,b,b,p,p}$, 
it follows that if $\delta''''$ small enough, independently of $d$, then we will have 
$\| V^{(ba^{-1} t)}_{i,b,b, p,p} - V^{(t)}_{i,a,b,p,p} \|_{M_d ,2} <\ell N\delta'''$.
Next note that $\sigma$, being an element of $\SA (L_1 ,n,\delta'''' ,\ell Nd)$,
satisfies $\tr_{\ell Nd} (\sigma_{ba^{-1} t} ) < \delta''''$, since $b,a^{-1} ,t\in L_1$, $n\geq 3$, and $ba^{-1} t \neq e$ 
(as $ba^{-1} \in H$ and $t\notin H$).
Consequently
\begin{align*}
\tr_d (V^{(t)}_{i,a,b,p,p} ) &\leq \tr_d (V^{(ba^{-1} t)}_{i,b,b,p,p} )
+ \| V^{(ba^{-1} t)}_{i,b,b, p,p} - V^{(t)}_{i,a,b,p,p} \|_{M_d ,2} \\
&< \ell N \tr_{\ell Nd} (\sigma_{ba^{-1}t} ) + \ell N\delta''' 
< \ell N \delta'''' + \frac{\delta''}{2} < \delta''.
\end{align*}
By a similar argument using $\omega'$, we may also arrange, by taking $\delta''''$
smaller if necessary, that 
$\tr_d (V^{(t)}_{i,a,b,p,p} ) < \delta''$ in the case that $t\in F_2^n \setminus H$.

Write $\Upsilon_{d,\sigma ,\omega ,W}$ for the set of all 
$U\in\sX_{d,\sigma,\omega,W}$ such that
\begin{align*}
\tr_{\ell' N' d} \big(\tilde{A}_1 \tilde{A}_2\cdots \tilde{A}_k \big) < \delta'
\end{align*}
for all even numbers $k \in \{ 1,\dots ,n\}$ and $A_1 , \dots , A_k$ alternating membership in
$\{ \sigma_s: s\in F_1^n \setminus H \}$ and $\{ \omega'_s : s\in F_2^n \setminus H \}$
where $\tilde{A}_j = A_j$ if $A_j$ lies in the first of these two sets and 
$\tilde{A}_j = (RUR)A_j (RUR)^*$ if $A_j$ lies in the second.
It then follows by the previous paragraph and
Lemma~\ref{L-free prod} (taking the projections labeled $P_1 ,\dots , P_\ell$ there
to be the projections $P_i \otimes E_{p,p}$ for $i=1,\dots,\ell'$ and $p=1,\dots,N_i$
in the present context) that if $k$ is even and $\delta''$ is small enough
then $\lim_{d\to\infty} |\Upsilon_{d,\sigma ,\omega ,W} |/|\sX_{d,\sigma,\omega,W} | = 1$ 
where the convergence is uniform with respect to $\sigma$, $\omega$, and $W$.
Then for all $U\in\Upsilon_{d,\sigma ,\omega ,W}$, $k=1,\dots ,\lfloor n/2\rfloor$, 
$A_1 , A_3 ,\dots , A_{2k-1}\in\{ \sigma_s : s\in F_1^n \setminus H \}$, and 
$A_2 , A_4 ,\dots ,A_{2k} \in\{ \omega'_s : s\in F_2^n \setminus H \}$ we have,
writing $R_0$ for $R$, $R_1$ for $1-R$, and $\sigma_0$ for the element of $\{ 0,1 \}^{4k}$ whose entries
are all $0$,
\begin{align*}
\lefteqn{\tr_{\ell Nd} \bigg( \prod_{j=1}^k A_{2j-1} \big(UA_{2j} U^* \big)\bigg)}\hspace*{20mm} \\
\hspace*{20mm} &= \frac{\ell' N'}{\ell N} \tr_{\ell' N' d} \bigg( \prod_{j=1}^k A_{2j-1}
\big( (RUR)A_{2j} (RUR)^* \big)\bigg) \\
&\hspace*{10mm} \ + \sum_{\sigma\in \{ 0,1 \}^{4k} \setminus \{\sigma_0 \}} 
\tr_{\ell Nd} \bigg( \prod_{j=1}^k A_{2j-1} R_{\sigma (4j-3)} UR_{\sigma (4j-2)} 
A_{2j} R_{\sigma (4j-1)}UR_{\sigma (4j)} \big)\bigg) \\
&< \delta' + 2^{4n} \| 1-R \|_2 < \frac{\delta}{4n}
\end{align*}
using the fact that $\delta' \leq\delta /(8n)$ and assuming that $\ell$ is large enough as a function of $n$ and $\delta$
and that $\eps$ is small enough as a function of $\delta$ so that $\ell' /\ell$ and $N' /N$ are sufficiently close to $1$ 
to ensure that $2^{4n} \| 1-R \|_2 < \delta /(8n)$.
By the tracial property this shows more generally that
\begin{gather*}\tag{$\dagger$}
\tr_{\ell Nd} \big(\tilde{A}_1 \tilde{A}_2\cdots \tilde{A}_k \big) < \frac{\delta}{4n}
\end{gather*}
for all $U\in\Upsilon_{d,\sigma ,\omega ,W}$, 
even numbers $k\in \{1,\dots ,n \}$, and $A_1 , \dots , A_k$ alternating membership in
$\{ \sigma_s: s\in F_1^n \setminus H \}$ and $\{ \omega'_s : s\in F_2^n \setminus H \}$
where $\tilde{A}_j = A_j$ if $A_j$ lies in the first of these two sets and 
$\tilde{A}_j = UA_j U^*$ if $A_j$ lies in the second.

Take a $U\in\Upsilon_{d,\sigma ,\omega ,W}$ and set $\omega'' = U\cdot\omega'$. 
By ($\ast$) above as it applies to $\omega'$, if we assume $\eps$ to be small enough as a function
of $\delta'$ and $|F|$ and assume $\delta''''$ to be small enough so that the proportion of $v\in \{ 1,\dots ,d \}$
for which $\omega'_s \omega'_t (v) = \omega'_{st} (v)$ for all $s\in F$ and $t\in T_k$ is sufficiently close to one,
then for every $s\in F_1^M \cap F_2^M$ 
the matrix $\omega'_s$ almost commutes with $U$ in trace norm to within a small enough tolerance
to ensure that $\| \omega''_s - \omega'_s \|_2 < \delta' /2$.
We may similarly guarantee, by assuming $\eps$ to be small enough as a function of $\delta'$
and assuming $\delta''''$ to be small enough so that the proportion of $v\in \{ 1,\dots ,d \}$
for which $\sigma_s \sigma_t (v) = \sigma_{st} (v)$ for all $s\in F$ and $t\in T_k$ is sufficiently close to one,
that $\| \omega'_s - \sigma_s \|_2 < \delta' /2$ for every $s\in F_1^M \cap F_2^M$.
Thus $\| \omega''_s - \sigma_s \|_2 < \delta'$ for all $s\in F_1^M \cap F_2^M$.

In the case that $H$ is finite, we substitute for Lemma~\ref{L-qt} the fact that
a good sofic approximation for $H$ will decompose into transitive orbits off of a set of small proportion.
Thus, if we are given positive integers $\ell ,\ell'$ with $\ell' /\ell > \max (1-\eps ,1-\kappa )$ as before
then, setting $N=|H|$, for every sufficiently good sofic approximation
$\sigma : H\to \{ 1,\dots ,\ell Nd\}$ there are sets $D_1 , \dots , D_{\ell'} \subseteq \{ 1,\dots , \ell Nd \}$
each of cardinality $d$ such that 
\begin{enumerate}
\item[(i)] the map $s\mapsto \sigma_s (c)$ from $H$ to $\{ 1,\dots , \ell Nd \}$ is injective for every 
$c\in\bigcup_{i=1}^{\ell'} D_i$, and

\item[(ii)] the sets $\sigma (H)c$ for $c\in\bigcup_{i=1}^{\ell'} D_i$ are pairwise disjoint.
\end{enumerate}
We can then carry out a similar kind of analysis as above with the tiles $S_1 , \dots , S_{\ell'}$
all being equal to $H$, $N'$ being equal to $N$, and each $N_i$ being equal to $1$
in order to obtain $\| \omega''_s - \sigma_s \|_2 < \delta'$ for all $s\in F_1^M \cap F_2^M$
as in the previous paragraph. We also define $\sX_{d,\sigma,\omega,W}$ as before, i.e.,
as the set of all $U\in S_{\ell Nd}$ of the form 
$(1-R) + \sum_{i=1}^{\ell'} P_i \otimes U_i$ under the 
appropriate identifications, where $P_i \otimes U_i \in M_{|H|} \otimes M_d \cong \cB (\ell^2 (\sigma (H)D_i ))$
and $R$ is the orthogonal projection of $\ell^2 (\{ 1,\dots ,\ell Nd\} )$ onto
$\ell^2 \big(\bigcup_{i=1}^{\ell'} \sigma (H)D_i \big)$.

Now by our application of Lemma~\ref{L-universal} at the beginning of the proof there exists 
an identity-preserving map $\Omega : G_1 *_H G_2 \to S_d$ such that
$\| \Omega_s - \sigma_s \|_2 < \delta /(16n)$ for all $s\in \tilde{F}_1^n$,
$\| \Omega_s - \omega''_s \|_2 < \delta /(16n)$ for all $s\in \tilde{F}_2^n$, and
$\| \Omega_{s_1 \cdots s_r} - \Omega_{s_1} \cdots\Omega_{s_r} \|_2 < \delta /(16n)$
for all $r=1,\dots ,n$ and $s_1 ,\dots ,s_r \in \tilde{F}_1^n \cup \tilde{F}_2^n$.
Let us verify that $\Omega$ belongs to $\SA (F_1  \cup F_2 ,n,\delta ,\ell Nd)$.
We need only check that $\Omega_s$
has trace less than $\delta$ when $s$ is a word in $F_1 \cup F_2$ of length at most $n$ 
which does not equal $e$.
Suppose then that we are given $r\in \{ 1,\dots ,n \}$ and 
$s_1 ,\dots ,s_r \in F_1 \cup F_2$ such that $s_1 \cdots s_r \neq e$.
By a reduction procedure that starts by taking a maximal collection $\cC$
of disjoint subwords of $s_1 \cdots s_r$ each of which belongs to $H$, partitions the complement of the union of $\cC$
into subwords which alternate membership in $F_1^n$ and $F_2^n$, and then concatenates these subwords as necessary
using the fact that $\tilde{F}_1$ and $\tilde{F}_2$ both contain $(F_1 \cup F_2 )^n \cap H$,
we can write $s_1 \cdots s_r = t_1 \cdots t_v$ where either 
\begin{enumerate}
\item[(a)] $v=1$ and $t_1 \in H$, or

\item[(b)] $v\leq r$ and $t_1 ,\dots , t_v$ alternate membership in $\tilde{F}_1^n \setminus H$ and $\tilde{F}_2^n \setminus H$.
\end{enumerate}
In case (a) we have $t_1 \in (F_1 \cup F_2 )^n \cap H \subseteq \tilde{F}_1 \cap L_1$ and so
\[
\tr (\Omega_{s_1 \cdots s_r} ) 
\leq \tr (\sigma_{t_1} ) + \| \sigma_{t_1} - \Omega_{t_1} \|_2 
< \delta'''' + \frac{\delta}{16n} \leq \delta .
\]
To handle case (b), first observe that,
writing $\rho_{t_i} = \sigma_{t_i}$ if $t_i \in \tilde{F}_1^n$ and $\rho_{t_i} = \omega''_{t_i}$ if $t_i \in \tilde{F}_2^n$, 
\begin{align*}
\| \Omega_{s_1 \cdots s_r} - \rho_{t_1} \cdots\rho_{t_v} \|_2
&\leq \| \Omega_{t_1 \cdots t_v} - \Omega_{t_1} \cdots\Omega_{t_v} \|_2 \\
&\hspace*{10mm} \ + \sum_{i=1}^v 
\| \rho_{t_1} \cdots\rho_{t_{i-1}} (\Omega_{t_i} - \rho_{t_i} )\Omega_{t_{i+1}} \cdots\Omega_{t_v} \|_2 \\
&< \frac{\delta}{16n} + \sum_{i=1}^v \| \Omega_{t_i} - \rho_{t_i} \|_2 
< \frac{\delta}{16n} + n\cdot \frac{\delta}{16n} \leq \frac{\delta}{2} 
\end{align*}
so that
\begin{align*}
\tr (\Omega_{s_1 \cdots s_r} ) 
\leq \tr (\rho_{t_1} \cdots\rho_{t_v} ) +
\| \Omega_{s_1 \cdots s_r} - \rho_{t_1} \cdots\rho_{t_v} \|_2 
< \tr (\rho_{t_1} \cdots\rho_{t_v} ) + \frac{\delta}{2} .
\end{align*}
It thus suffices to show that $\tr (\rho_{t_1} \cdots\rho_{t_v} ) < \delta /2$.
If $v$ is even then we have the estimate $\tr (\rho_{t_1} \cdots\rho_{t_v} ) < \delta /(4n)$ 
provided by ($\dagger$) (the division by $4n$ will be useful below).
In the case $v=1$ we have $t_1 \in \tilde{F}_1^n \cup \tilde{F}_2^n \subseteq L_1 \cup L_2$ so that
\[
\tr (\rho_{t_1} ) < \delta'''' < \frac{\delta}{4n} .
\]
Finally, if $v$ is odd and greater than $1$, we first note that
$t_1$ and $t_v$ either both belong to $\tilde{F}_1^n$ or both belong to $\tilde{F}_2^n$ so that 
$\| \rho_{t_v} \rho_{t_1} - \rho _{t_v t_1} \|_2 < \delta'''' \leq \delta /(4n)$, and then
subdivide into two cases:
\begin{enumerate}
\item[(c)] $t_v t_1 \notin H$. As $t_v t_1$ belongs to $\tilde{F}_1^n$ if $t_1 , t_v \in \tilde{F}_1^n$ or to $\tilde{F}_2^n$ if $t_1 , t_v \in \tilde{F}_2^n$,
we can reduce to the case of even $v$:
\begin{align*}
\tr (\rho_{t_1} \cdots\rho_{t_v} ) 
&\leq \tr (\rho_{t_v t_1} \rho_{t_2} \cdots\rho_{t_{v-1}} ) 
+ \| (\rho_{t_v} \rho_{t_1} - \rho _{t_v t_1} )\rho_{t_2} \cdots\rho_{t_{v-1}} \|_2 \\
&< \frac{\delta}{4n} + \frac{\delta}{4n} = \frac{\delta}{2n} .
\end{align*}

\item[(d)] $t_v t_1 \in H$. In this case $t_v t_1 t_2 \notin H$, $t_v t_1$ belongs to one of $\tilde{F}_1^n$ and $\tilde{F}_2^n$
and $t_2$ to the other, and $t_v t_1 t_2 \in \tilde{F}_1^n \cup \tilde{F}_2^n$ since $t_v t_1 \in (F_1 \cup F_2 )^n \cap H$. Thus
\begin{align*}
\| \rho_{t_v t_1 t_2} - \rho_{t_v t_1} \rho_{t_2} \|_2
&\leq \| \rho_{t_v t_1 t_2} - \Omega_{t_v t_1 t_2} \|_2 + \| \Omega_{t_v t_1 t_2} - \Omega_{t_v t_1} \Omega_{t_2} \|_2 \\
&\hspace*{10mm} \ + \| (\Omega_{t_v t_1} - \rho_{t_v t_1} )\Omega_{t_2} \|_2 
+  \| \rho_{t_v t_1} (\Omega_{t_2} - \rho_{t_2} ) \|_2 \\
&< 4\cdot \frac{\delta}{16n} = \frac{\delta}{4n}
\end{align*}
and so
\begin{align*}
\tr (\rho_{t_1} \cdots\rho_{t_v} ) 
&\leq \tr (\rho_{t_v t_1 t_2} \rho_{t_3} \cdots\rho_{t_{v-1}} ) + \| \rho_{t_v t_1 t_2} - \rho_{t_v t_1} \rho_{t_2} \|_2 \\
&\hspace*{15mm} \ + \| (\rho_{t_v t_1} - \rho_{t_v} \rho_{t_1} )\rho_{t_2} \|_2 \\
&< \tr (\rho_{t_v t_1 t_2} \rho_{t_3} \cdots\rho_{t_{v-1}} ) + \frac{\delta}{2n} .
\end{align*}
To estimate $\tr (\rho_{t_v t_1 t_2} \rho_{t_3} \cdots\rho_{t_{v-1}} )$ we again run through the procedure 
for handling the case of odd $v$ in (b) with $\rho_{t_1} \cdots\rho_{t_v}$ now replaced by the shorter product
$\rho_{t_v t_1 t_2} \rho_{t_3} \cdots\rho_{t_{v-1}}$. This may lead us back repeatedly into
case (d), but after less than $n$ steps the process will stop, with the estimates 
accumulating to yield $\tr (\rho_{t_1} \cdots\rho_{t_v} ) < \delta /2$.
\end{enumerate}
We conclude that $\Omega\in\SA (F_1  \cup F_2 ,n,\delta ,\ell Nd)$, as desired.

We finish the proof with the following counting argument. 
Note that $\Omega$ was obtained by amalgamating perturbations of $\sigma$ and $\omega''$, 
where the latter was obtained from $\omega$ by conjugating by $W$ and then by $U$. 
The set $\Lambda_{d,\sigma ,\omega}$ of all products $UW$ such that $U$ and $W$
together do the required job has cardinality at least $|\sX_{d,\sigma,\omega,W_0} |/2$
for all $d$ larger than some $d_0$ not depending on $\sigma$ or $\omega$, where $W_0$ is any fixed
$W$ doing the required job. Thus,
by Stirling's approximation and the fact that $\ell' > (1-\kappa )\ell$, for all sufficiently large $d$ we have
\begin{align*}
|\Lambda_{d,\sigma ,\omega} | 
\geq \frac12 |\sX_{d,\sigma,\omega,W_0} | \geq \frac12 d!^{\ell'} \geq d^{\ell d(1-\kappa )} 
\end{align*}
and hence, writing $\cS$ for the set of all $U\in S_{\ell Nd}$
such that $(U\cdot\omega )|_{E_2} = \omega |_{E_2}$,
\begin{align*}
|S_{\ell Nd} \cdot\omega |_{E_2} = \frac{|S_{\ell Nd} |}{|\cS |} 
&\leq \frac{(\ell Nd)^{\ell Nd}}{|\cS |} \cdot \frac{|\Lambda_{d,\sigma ,\omega} |}{d^{\ell d(1-\kappa )}} \\
&= (\ell Nd)^{\ell Nd}d^{-\ell d(1-\kappa )}\frac{|\Lambda_{d,\sigma ,\omega} |}{|\cS |} \\
&\leq (\ell Nd)^{\ell Nd}d^{-\ell d(1-\kappa )} |\Lambda_{d,\sigma ,\omega} \cdot\omega |_{E_2} .
\end{align*}
Let $\sR$ be a set of representatives for the orbits of the action of $S_{\ell Nd}$ on $\sY_2$
modulo the relation of equality on $E_2$.
Then $|\sY_2 |_{E_2} = \sum_{\omega\in\sR} |S_{\ell Nd} \cdot\omega |_{E_2}$.
Since $E_2 \subseteq L_2$, 
by Lemma~\ref{L-nbhd gd} we see that, independently of $d$, 
if $\delta$ is small enough as a function of $\kappa$ then,
modulo the relation of equality on $E_2$, given a $\sigma\in\sY_1$ at most $(\ell Nd)^{\kappa \ell Nd}$ many $\omega\in\sY_2$ 
which all differ on $E_2$ can lead via our procedure to maps $\Omega$ which all agree
on $E_1 \cup E_2$. Take a set $\sY_1'$
of representatives for the relation on $\sY_1$ given by equality on $E_1$. By Lemma~\ref{L-nbhd gd},
if $\delta$ is small enough independently of $d$ then we can find a $\sY_1'' \subseteq \sY_1'$ with 
$|\sY_1'' |\geq (\ell Nd)^{-\kappa\ell Nd} |\sY_1' |$ such that distinct elements of $\sY_1''$ give rise to
maps $\Omega$ which disagree on $E_1$ no matter which elements of $\sY_2$ are used. We therefore obtain,
for all sufficiently large $d$,
\begin{align*}
|\SA (F_1 \cup F_2 ,n ,\delta ,\ell Nd)|_{E_1 \cup E_2}
&\geq (\ell Nd)^{-\kappa \ell Nd} \sum_{\sigma\in\sY_1''}
\sum_{\omega\in\sR} |\Lambda_{d,\sigma ,\omega}\cdot\omega |_{E_2} \\
&\geq (\ell Nd)^{-\ell Nd(1 + \kappa )} d^{\ell d(1-\kappa )} |\sY_1'' | 
\sum_{\omega\in\sR} |S_{\ell Nd} \cdot\omega |_{E_2} \\
&\geq (\ell Nd)^{-\ell Nd(1 + 2\kappa )} d^{\ell d(1-\kappa )} |\sY_1 |_{E_1} 
\sum_{\omega\in\sR} |S_{\ell Nd} \cdot\omega |_{E_2} \\
&= (\ell Nd)^{-\ell Nd(1 + 2\kappa )} d^{\ell d(1-\kappa )}
|\sY_1 |_{E_1} |\sY_2 |_{E_2} 
\end{align*}
and hence, in view of Lemma~\ref{L-mult},
\begin{align*}
\lefteqn{\liminf_{d\to\infty} \frac{1}{\ell Nd\log (\ell Nd)} 
|\SA (F_1 \cup F_2 ,n,\delta ,\ell Nd)|_{E_1 \cup E_2}}\hspace*{15mm} \\
\hspace*{10mm} &\geq \liminf_{d\to\infty} \frac{1}{\ell Nd\log (\ell Nd)} |\SA (L_1 ,n,\delta'''',\ell Nd)|_{E_1} \\
&\hspace*{7mm} \ + \liminf_{d\to\infty} \frac{1}{\ell Nd\log (\ell Nd)} |\SA (L_2 ,n,\delta'''',\ell Nd)|_{E_2} - 1 
+ \frac1N (1-\kappa ) - 2\kappa \\
&\geq \slower_{E_1} (L_1 ) + \slower_{E_2} (L_2 ) - 1 + \frac1N (1-\kappa ) - 2\kappa .
\end{align*}
Since $n$ was an arbitary postive integer, $\kappa$ and $\delta$ can be taken arbitrarily small, 
and $N = |H|$ in the case that $H$ is finite,
it follows that $\slower_{E_1 \cup E_2} (F_1 \cup F_2 ) 
\geq \slower_{E_1} (G_1 ) + \slower_{E_2} (G_2 ) - 1 + |H|^{-1}$. Thus, using Theorem~\ref{T-generating} together with
the fact that $G_1 \cup G_2$ generates $G_1 *_H G_2$,
\begin{align*}
\slower (G_1 *_H G_2 ) 
&= \slower (G_1 \cup G_2 ) \\
&\geq\slower_{E_1 \cup E_2} (G_1 \cup G_2 ) \\
&\geq\slower_{E_1 \cup E_2} (F_1 \cup F_2 ) - \theta \\
&\geq \slower_{E_1} (G_1 ) + \slower_{E_2} (G_2 ) - 1 + \frac{1}{|H|} - \theta \\
&\geq \slower (G_1 ) + \slower (G_2 ) - 1 + \frac{1}{|H|} - 3\theta .
\end{align*}
Since $\theta$ was an arbitrary positive number we thereby obtain the result.
\end{proof}

\begin{remark}
Lemma~\ref{L-amalgamated lower}, in conjunction with Propositions~\ref{P-sofic inf group} and
\ref{P-finite}, gives a free probability proof that $G_1 *_H G_2$ is sofic
whenever $G_1$ and $G_2$ are sofic countable discrete groups with common amenable subgroup $H$.
This fact was established for monotileable $H$ in \cite{ColDyk10} using similar free probability 
arguments, and in general in \cite{EleSza10} by means of graph techniques and in \cite{Pau10}
using Bernoulli shifts and equivalence relations.
\end{remark}

Combining Lemmas~\ref{L-amalgamated upper} and \ref{L-amalgamated lower} we obtain the following.

\begin{theorem}\label{T-amalgamated}
Let $G_1$ and $G_2$ be countable discrete groups with common amenable subgroup $H$.
Suppose that $G_1$ and $G_2$ are approximation regular.
Then $G_1 *_H G_2$ is approximation regular and
\[
s(G_1 *_H G_2 ) = s(G_1 ) + s(G_2 ) - 1 + \frac{1}{|H|} .
\]
\end{theorem}

\begin{corollary}\label{C-free}
Let $r\in\Nb\cup \{ \infty \}$. Then $s(F_r ) = \slower (F_r ) = r$.
\end{corollary}

\begin{proof}
If $r < \infty$ then we can repeatedly apply Theorem~\ref{T-amalgamated} using the fact 
that $s(\Zb ) = \slower (\Zb ) = 1$, which one can either compute directly or
obtain from Theorem~\ref{T-amenable group} below. Consider then the case $r=\infty$.
Let $s_1 , s_2 ,\dots$ be the standard generators for $F_\infty$. Then,
in the spirit of the proof of Lemma~\ref{L-free prod}, for positive integers
$n\leq m$ one can show by repeated application of Lemma~\ref{L-conc} 
that a random choice of $m$ permutations of a finite set $\{ 1,\dots ,d \}$ will, with high
probability, be a good sofic model for $G$ up to within some prescribed precision.
This will demonstrate that $s_{\{ s_1 , \dots , s_n \}} (\{ s_1 , \dots , s_m \} ) = n$
and hence $\slower_{\{ s_1 , \dots , s_n \}} (F_\infty ) = n$,
so that $\slower (F_\infty ) = \infty$ by Theorem~\ref{T-generating}. 
\end{proof}

\begin{theorem}\label{T-amenable group}
Suppose that $G$ is amenable. Then $s(G) = \slower (G) = 1 - |G|^{-1}$. 
\end{theorem}

\begin{proof}
By Proposition~\ref{P-finite} we may assume that $G$ is infinite.
Then by Proposition~\ref{P-inf gd} we have $s(G)\geq\slower (G) \geq 1$.
On the other hand, it follows from Lemma~\ref{L-qt} that for any two good enough sofic approximations
$G\to\Sym (d)$ there is an element of $\Sym (d)$ which 
approximately conjugates one to the other in trace norm on a prescribed finite set,
with this approximation not depending on $d$. Consequently $s(G)\leq 1$.
\end{proof}

The above theorem shows that all subsets of an amenable group are approximation regular, 
since amenability passes to subgroups.

\section{Group actions}\label{S-group actions}

Throughout this section and the next $G$ denotes a countable discrete group
and $(X,\mu )$ a standard probability space,
which are arbitrary unless otherwise specified.
Let $\alpha$ a measure-preserving action of $G$
on $(X,\mu )$. The notation $\alpha$ will actually be reserved
for the induced action of $G$ on $L^\infty (X,\mu)$, so that $\alpha_s (f)(x) = f(s^{-1} x)$ 
for $s\in G$, $f\in L^\infty (X,\mu )$, and $x\in X$, with concatenation being used for
the action on $X$. For a set of projections $\cP\subseteq L^\infty (X,\mu )$ 
and a nonempty finite set $F\subseteq G$, we write $\cP_F$ for the
set of the projections of the form $\prod_{s\in F} \alpha_s (p_s )$ where $p_s \in \cP$. 
We say that a subset $\Omega$ of $L^\infty (X,\mu )$ is 
{\it dynamically generating} if the set $\bigcup_{s\in G} \{ \alpha_s (a) : a\in\Omega \}$ 
generates $L^\infty (X,\mu )$ as a von Neumann algebra.
In the case that $\Omega$ is a partition of unity consisting of projections this is the same
as the underlying partition of $X$ being generating for the action.
 
We write $s(G,X)$ for $s(\sG )$, $\slower (G,X)$ for $\slower (\sG )$, and $I_{G,X}$ for $I_\sG$, 
where $\sG$ is the p.m.p.\ groupoid associated to the action. For a group element $s$ we write $u_s$ 
for the corresponding element in $I_{G,X}$. We say that the action is {\it approximation regular} if 
$s(G,X) = \slower (G,X)$.

For the purpose of working with $s(G,X)$ and $\slower (G,X)$ it is often more convenient 
to handle the group and space components separately as follows.
Let $\sigma$ be a map from $G$ to $S_d$ for some $d\in\Nb$.
The image $\sigma_s$ of a group element $s$ under $\sigma$ will usually be interpreted 
as a permutation matrix in $M_d$. Viewed as such, $\sigma_s$ gives rise to an automorphism 
$\Ad \sigma_s$ of $\Cb^d$ as identified with the algebra $\diag (M_d )$ of diagonal matrices in $M_d$.
Let $F$ be a nonempty finite subset of $G$ and $\delta > 0$.
Recall from the previous section that $\GA (F,n,\delta ,d)$ denotes the set of all identity-preserving maps 
$\sigma : G \to S_d$ such that $\| \sigma_{s_1 ,\dots ,s_n} - \sigma_{s_1} \cdots\sigma_{s_n} \|_2 < \delta$ 
for all $(s_1 , \dots s_n )\in (F\cup F^* \cup \{ e \} )^{\times n}$ and $\tr (\sigma_s ) < \delta$ for all 
$s\in (F\cup F^* \cup \{ e \} )^n \setminus \{ e \}$.
Let $\cP$ be a finite set of projections in $L^\infty (X,\mu )$.
Write $\HA (F,\cP ,n,\delta ,d)$ for the set of all pairs $(\sigma , \varphi )$
where $\sigma\in\GA (F,n,\delta ,d)$ and $\varphi$ is a unital homomorphism from 
$\spn (\cP_{(F\cup F^* \cup \{ e \} )^n} )$ to
$\Cb^d = \diag (M_d )$ satisfying
\begin{enumerate}
\item[(i)] $| \tr\circ\varphi (p) - \mu (p) | < \delta$ for all $p\in\cP_{(F\cup F^* \cup \{ e \})^n}$,

\item[(ii)] $\| \varphi\circ\alpha_s (p) - \Ad\sigma_s \circ\varphi (p) \|_2 < \delta$ for all $p\in\cP$ 
and $s\in (F\cup F^* \cup \{ e \})^n$.
\end{enumerate}
Given sets $A_1 , A_2 , B_1 , B_2 , Z_1 , Z_2$ and a collection $\sY$
of ordered pairs consisting of maps $A_1 \to Z_1$ and $A_2 \to Z_2$,
we write $|\sY |_{B_1 , B_2}$ for the cardinality of the set of pairs 
$(\sigma |_{B_1 \cap A_1} , \varphi |_{B_2 \cap A_2} )$
where $(\sigma , \varphi )\in\sY$.
For a finite set $E\subseteq G$ and a finite set $\cQ$ of projections in $L^\infty (X,\mu )$ we set
\begin{align*}
s_{E,\cQ} (F,\cP ,n,\delta ) &= \limsup_{d\to\infty} \frac{1}{d \log d} \log |\HA (F,\cP ,n ,\delta ,d)|_{E,\cQ} , \\
s_{E,\cQ} (F,\cP ,n) &= \inf_{\delta > 0} s_{E,\cQ} (F,\cP ,n,\delta ) , \\
s_{E,\cQ} (F,\cP ) &= \inf_{n\in\Nb} s_{E,\cQ} (F,\cP ,n) .
\end{align*}
and
\begin{align*}
\slower_{E,\cQ}  (F,\cP ,n,\delta ) &= \liminf_{d\to\infty} \frac{1}{d \log d} \log |\HA (F,\cP ,n ,\delta ,d)|_{E,\cQ} , \\
\slower_{E,\cQ}  (F,\cP ,n) &= \inf_{\delta > 0} \slower_{E,\cQ}  (F,\cP ,n,\delta ) , \\
\slower_{E,\cQ}  (F,\cP ) &= \inf_{n\in\Nb} \slower_{E,\cQ}  (F,\cP ,n) .
\end{align*}
A simple approximation argument shows the following.

\begin{proposition}\label{P-pou}
Let $E,F \subseteq G$ be finite sets and let $\cQ$ be a finite set of projections in $L^\infty (X,\mu )$.
Let $\cP$ be a finite partition of unity in $L^\infty (X,\mu )$ consisting of projections. 
Then the set $\tilde{\cP}$ of all projections in the $^*$-subalgebra spanned by $\cP$ satisfies
\[
s_{E,\cQ} (F,\tilde{\cP} ) = s_{E,\cQ} (F,\cP ) .
\]
\end{proposition}

\begin{proposition}\label{P-actions local}
Let $F$ be a finite symmetric subset of $G$ containing $e$
and let $\cP$ be a set consisting of the projections in some finite-dimensional unital $^*$-subalgebra $A$
of $L^\infty (X,\mu )$.
Let $E$ be a finite subset of $G$ and $\cQ$ a subset of $\cP$.
Then $s_{E\cup\cQ} (F\cup\cP ) = s_{E,\cQ} (F,\cP )$ and $\slower_{E\cup\cQ} (F\cup\cP ) = \slower_{E,\cQ} (F,\cP )$.
\end{proposition}

\begin{proof}
First we show that $s_{E\cup\cQ} (F\cup\cP ) \geq s_{E,\cQ} (F,\cP )$ 
and $\slower_{E\cup\cQ} (F\cup\cP ) \geq \slower_{E,\cQ} (F,\cP )$.
Let $n\in\Nb$ and $\delta > 0$. 
Let $\delta' > 0$ be such that $3n\delta' < \delta$.
Let $d\in\Nb$.
Let $(\sigma , \varphi )\in\HA (F,\cP ,n,\delta' ,d)$.
Write $\tilde{\cP}$ for the subset of $\cP$ consisting of the minimal projections of $A$.
Define a map $\Phi_{\sigma ,\varphi} : I_\sG \to I_d$ 
by setting $\Phi_{\sigma ,\varphi} (pu_s ) = \varphi (p) \sigma_s$ for all $p\in\tilde{\cP}_{F^n}$ and $s\in G$,
extending linearly, and then extending arbitrarily to all of $I_\sG$. 
Note in particular that $\Phi_{\sigma ,\varphi} (1) = 1$
since $\sigma$ is identity-preserving and $\varphi$ is unital. We will show 
that $\Phi_{\sigma ,\varphi}\in\SA (F\cup\cP ,n,\delta ,d)$.

Let $p_1 , \dots p_n \in\cP$ and $s_1 , \dots s_n \in F$. 
Then
\[
\bigg( \prod_{i=1}^n \varphi (p_i ) \sigma_{s_i} \bigg) \sigma_{s_n}^{-1} \cdots \sigma_{s_1}^{-1}
= \prod_{i=1}^n (\Ad \sigma_{s_1} \cdots\sigma_{s_{i-1}} )(\varphi (p_i ))
\]
and so by untelescoping to estimate the difference of products we obtain
\begin{align*}
\lefteqn{\bigg\| \varphi \bigg( \prod_{i=1}^n \alpha_{s_1 \cdots s_{i-1}} (p_i )\bigg) 
- \bigg( \prod_{i=1}^n \varphi (p_i ) \sigma_{s_i} \bigg) 
\sigma_{s_n}^{-1} \cdots \sigma_{s_1}^{-1} \bigg\|_2}\hspace*{15mm} \\
\hspace*{15mm} &\leq \sum_{i=2}^n \big( \| \varphi (\alpha_{s_1 \cdots s_{i-1}} (p_i )) 
- \Ad\sigma_{s_1 \cdots s_{i-1}} (\varphi (p_i )) \|_2 \\
&\hspace*{20mm} \ + \| (\Ad\sigma_{s_1 \cdots s_{i-1}} - \Ad\sigma_{s_1} \cdots\sigma_{s_{i-1}} )(\varphi (p_i )) \|_2 \big) \\
&< (n-1)\delta' + 2\sum_{i=2}^n \| \sigma_{s_1 \cdots s_{i-1}} - \sigma_{s_1} \cdots\sigma_{s_{i-1}} \|_2 \\
&< 3(n-1)\delta' .
\end{align*}
Since 
$\prod_{i=1}^n p_i u_{s_i} = (\prod_{i=1}^n \alpha_{s_1 \cdots s_{i-1}} (p_i ))u_{s_1 \cdots s_n}$, 
it follows that
\begin{align*}
\lefteqn{\bigg\| \Phi_{\sigma ,\varphi} \bigg( \prod_{i=1}^n p_i u_{s_i} \bigg) 
- \prod_{i=1}^n \Phi_{\sigma ,\varphi} (p_i ) \Phi_{\sigma ,\varphi} (u_{s_i} ) \bigg\|_2}\hspace*{20mm} \\ 
\hspace*{20mm} &= \bigg\| \varphi \bigg( \prod_{i=1}^n \alpha_{s_1 \cdots s_{i-1}} (p_i )\bigg) 
\sigma_{s_1 \cdots s_n} - \prod_{i=1}^n \varphi (p_i ) \sigma_{s_i} \bigg\|_2 \\
\hspace*{20mm} &= \bigg\| \varphi \bigg( \prod_{i=1}^n \alpha_{s_1 \cdots s_{i-1}} (p_i )\bigg) 
(\sigma_{s_1 \cdots s_n} - \sigma_{s_1} \cdots \sigma_{s_n} ) \bigg\|_2 \\
&\hspace*{15mm} \ + \bigg\| \varphi \bigg( \prod_{i=1}^n \alpha_{s_1 \cdots s_{i-1}} (p_i )\bigg) 
- \bigg( \prod_{i=1}^n \varphi (p_i ) \sigma_{s_i} \bigg) \sigma_{s_n}^{-1} \cdots \sigma_{s_1}^{-1} \bigg\|_2 \\
&< \| \sigma_{s_1 \cdots s_n} - \sigma_{s_1} \cdots \sigma_{s_n} \|_2 + 3(n-1)\delta' < (3n-2)\delta' < \delta .
\end{align*}
Since $1\in\cP$ and $e\in F$ this shows that 
$\| \Phi_{\sigma ,\varphi} (a_1 \cdots a_k ) 
- \Phi_{\sigma ,\varphi} (a_1 )\cdots\Phi_{\sigma ,\varphi} (a_k ) \| < \delta$
for all $k=1,\dots ,n$ and $(a_1 , \dots , a_k )\in (F\cup\cP )^{\times k}$.
Note also that if $s_1 \cdots s_n = e$ then, since $\prod_{i=1}^n \alpha_{s_1 \cdots s_{i-1}} (p_i ) \in\cP_{F^n}$,
\begin{align*}
\bigg| \tr\circ\Phi_{\sigma ,\varphi} \bigg( \prod_{i=1}^n p_i u_{s_i} \bigg) 
- \tau \bigg( \prod_{i=1}^n p_i u_{s_i} \bigg) \bigg|
= \bigg| (\tr\circ\varphi - \mu )\bigg( \prod_{i=1}^n \alpha_{s_1 \cdots s_{i-1}} (p_i )\bigg) \bigg| 
< \delta' < \delta ,
\end{align*}
while if $s_1 \cdots s_n \neq e$ then
\begin{align*}
\bigg| \tr\circ\Phi_{\sigma ,\varphi} \bigg( \prod_{i=1}^n p_i u_{s_i} \bigg) - \tau \bigg( \prod_{i=1}^n p_i u_{s_i} \bigg) \bigg|
&= \bigg| \tr \bigg( \varphi \bigg( \prod_{i=1}^n \alpha_{s_1 \cdots s_{i-1}} (p_i )\bigg) \sigma_{s_1 \cdots s_n} \bigg) \bigg| \\
&\leq \tr (\sigma_{s_1 \cdots s_n} ) 
< \delta' < \delta .
\end{align*}
Since $1\in\cP$ and $e\in F$, this shows that $|\tr\circ\Phi_{\sigma ,\varphi} (a) - \tau (a) | < \delta$ 
for all $a\in (F\cup\cP )^{\leq n}$. 
We have thus verified that $\Phi_{\sigma ,\varphi}\in\SA (F\cup\cP ,n,\delta ,d)$.
Since for any $(\sigma ,\varphi ),(\omega ,\psi )\in\HA (F ,\cP ,n,\delta' ,d)$
such that $(\sigma |_E ,\varphi |_{\cQ} )$ and $(\omega |_E ,\psi |_{\cQ} )$ are distinct
the restrictions of $\Phi_{\sigma ,\varphi}$ and $\Phi_{\omega ,\psi}$ to $E\cup\cQ$ are distinct, 
it follows that
\[
|\SA (F\cup\cP ,n,\delta ,d)|_{E\cup\cQ} \geq |\HA (F,\cP ,n,\delta' ,d)|_{E,\cQ} ,
\]
from which we infer that $s_{E\cup\cQ} (F\cup\cP ) \geq s_{E,\cQ} (F,\cP )$ 
and $\slower_{E\cup\cQ} (F\cup\cP ) \geq \slower_{E,\cQ} (F,\cP )$. 

To prove the reverse inequalities, let $n\in\Nb$, and 
let $m$ be an integer larger than $6|F^n |$. 
Let $\delta'$ be a positive number smaller than $\delta / (5+4n)$, $\delta /(2|\cP |^{|F|^m})$, and $\delta /36$,
to be further specified. Let $d\in\Nb$.
Let $\Phi\in\SA (F\cup\cP ,m,\delta' ,d)$. Let $s_1 , \dots , s_k$ be a enumeration of the elements
of $F^n$, for the purpose of indexing noncommutative products below.
Given a $p\in\cP_{F^n}$, writing $p=\prod_{i=1}^k u_{s_i} p_{s_i} u_{s_i}^*$ 
where the $p_{s_i}$ are projections in $\cP$, we have by Lemma~\ref{L-adjoint}
\begin{align*}
\lefteqn{\bigg\| \prod_{i=1}^k \Phi (u_{s_i}^* )^* \Phi (p_{s_i} )^* \Phi (u_{s_i} )^* 
- \prod_{i=1}^k \Phi (u_{s_i} )\Phi (p_{s_i} )\Phi (u_{s_i}^* ) \bigg\|_2}\hspace*{15mm} \\
\hspace*{15mm} &\leq \sum_{i=1}^k \big( \| \Phi (u_{s_i}^* )^* - \Phi (u_{s_i} ) \|_2 
+ \| \Phi (p_{s_i} )^* - \Phi (p_{s_i} ) \|_2 + \| \Phi (u_{s_i} )^* - \Phi (u_{s_i}^* ) \|_2 \big) \\
&\leq 12|F^n |\delta'
\end{align*}
and thus, using the fact that $m > 6|F^n |$, 
\begin{align*}\tag{1}
\lefteqn{\| \Phi (p)^* \Phi (p) - \Phi (p) \|_2}\hspace*{15mm} \\
\hspace*{15mm} &\leq \bigg\| \bigg(\Phi (p) 
- \prod_{i=1}^k \Phi (u_{s_{k-i+1}} )\Phi (p_{s_{k-i+1}} )\Phi (u_{s_{k-i+1}}^* )\bigg)^* \Phi (p) \bigg\|_2 \\
&\hspace*{10mm} \ + \bigg\| \bigg( \prod_{i=1}^k \Phi (u_{s_i}^* )^* \Phi (p_{s_i} )^* \Phi (u_{s_i} )^* 
- \prod_{i=1}^k \Phi (u_{s_i} )\Phi (p_{s_i} )\Phi (u_{s_i}^* ) \bigg) \Phi (p) \bigg\|_2 \\
&\hspace*{10mm} \ + \bigg\| \bigg( \prod_{i=1}^k \Phi (u_{s_i} )\Phi (p_{s_i} )\Phi (u_{s_i}^* ) \bigg)
\bigg(\Phi (p) - \prod_{i=1}^k \Phi (u_{s_i} )\Phi (p_{s_i} )\Phi (u_{s_i}^* )\bigg) \bigg\|_2 \\
&\hspace*{10mm} \ + \bigg\| \prod_{i=1}^k \Phi (u_{s_i} )\Phi (p_{s_i} )\Phi (u_{s_i}^* )
\prod_{i=1}^k \Phi (u_{s_i} )\Phi (p_{s_i} )\Phi (u_{s_i}^* ) - \Phi (p^2 ) \bigg\|_2 \\
&< (3+12|F^n |)\delta' .
\end{align*}
For $p,q\in\cP_{F^n}$ we have, by a similar estimate again using the fact that $m > 6|F^n |$,
\begin{align*}\tag{2}
\| \Phi (pq) - \Phi (p) \Phi (q) \|_2 < 3\delta' .
\end{align*}
As before, write $\tilde{\cP}$ for the subset of $\cP$ consisting of the minimal projections in $A$.
Pick a $p_0 \in \tilde{\cP}_{F^n}$. Since $\Phi (p)^* \Phi (p)$ is a projection in $\diag (M_d )$
for every $p\in\tilde{\cP}_{F^n}$, it follows from (1) and (2) and a straightforward perturbation argument
that we can find pairwise orthogonal projections
$\varphi_\Phi (p) \in\Cb^d \cong\diag (M_d )$ for $p\in\tilde{\cP}_{F^n} \setminus \{ p_0 \}$ such that 
$\| \varphi_\Phi (p) - \Phi (p) \|_2$ is as small as we wish for every $p\in\tilde{\cP}_{F^n} \setminus \{ p_0 \}$
granted that $\delta'$ is taken small enough. 
Setting $\varphi_\Phi (p_0 ) = 1 - \sum_{p\in\tilde{\cP}_{F^n} \setminus \{ p_0 \}} \varphi_\Phi (p)$
and extending linearly we obtain a unital homomorphism $\varphi_\Phi : \spn (\cP_{F^n} )\to \Cb^d$,
and by taking $\delta'$ small enough we can ensure that
$\| \varphi_\Phi (p) - \Phi (p) \|_2 < \delta /(3n)$ 
for every projection $p$ in the linear span of $\cP_{F^n}$.
For $s\in F^n \setminus\{ e \}$ the partial isometry $\Phi (u_s )$ satisfies 
$\| \Phi (u_s^* ) - \Phi (u_s )^* \|_2 < 3\delta'$ by Lemma~\ref{L-adjoint} and hence
\begin{align*}
\| \Phi (u_s )^* \Phi (u_s ) - 1 \|_2 &\leq \| (\Phi (u_s )^* - \Phi (u_s^* ))\Phi (u_s ) \|_2 
+ \| \Phi (u_s^* )\Phi (u_s ) - \Phi (u_s^* u_s ) \|_2 \\
&< 4\delta' ,
\end{align*}
which means that we can construct a permutation matrix $\sigma_{\Phi ,s} \in S_d$ such that 
$\| \sigma_{\Phi ,s} - \Phi (u_s ) \|_2 < 4\delta'$. 
For all other $s\in G$ we set $\sigma_{\Phi ,s} = 1$, giving us a map $\sigma_\Phi : G\to S_d$.
For all $(s_1 , \dots ,s_n )\in F^{\times n}$ we have 
\begin{align*}
\lefteqn{\| \sigma_{\Phi ,s_1 \cdots s_n} - \sigma_{\Phi ,s_1} \cdots \sigma_{\Phi ,s_n} \|_2}\hspace*{15mm} \\
\hspace*{15mm} &\leq \| \sigma_{\Phi ,s_1 \cdots s_n} - \Phi (u_{s_1 \cdots s_n} ) \|_2 \\
&\hspace*{10mm} \ + \| \Phi (u_{s_1} \cdots u_{s_n} ) - \Phi (u_{s_1} ) \cdots \Phi (u_{s_n} ) \|_2 \\
&\hspace*{10mm} \ + 
\sum_{i=1}^n \| \sigma_{\Phi ,s_1} \cdots \sigma_{\Phi ,s_{i-1}} (\Phi (u_{s_i} ) -\sigma_{\Phi ,s_i} )
\Phi (u_{s_{i+1}}) \cdots \Phi (u_{s_n} )\|_2 \\
&< (5 + 4n)\delta' < \delta
\end{align*}
while for
$s\in F^n \setminus \{ e \}$ we have 
\[
\tr (\sigma_{\Phi ,s} ) =  \tr (\sigma_{\Phi ,s} - \Phi (u_s )) + \tr (\Phi (u_s )) < 5\delta' < \delta ,
\]
so that $\sigma_\Phi \in\GA (F,n,\delta ,d)$. 

For $p\in\cP_{F^n}$ we have,
since $\cP_{F^n}\subseteq (F\cup\cP )^m$,
\begin{align*}
| \tr\circ\varphi_\Phi (p) - \mu (p) | 
&\leq | \tr (\varphi_\Phi (p) - \Phi (p)) | + | \tr\circ\Phi (p) - \tau (p) | \\
&\leq \| \varphi_\Phi (p) - \Phi (p) \|_2 + \delta' < \frac{\delta}{2} + \frac{\delta}{2} = \delta .
\end{align*}
Note also that for $p\in\cP$ and $s\in F^n$ we have, using Lemma~\ref{L-adjoint},
\begin{align*}
\lefteqn{\| \varphi_\Phi \circ\alpha_s (p) - \Ad\sigma_{\Phi ,s} \circ\varphi_\Phi (p) \|_2}\hspace*{15mm} \\
\hspace*{10mm} &\leq \| \varphi_\Phi (u_s p u_s^* ) - \Phi (u_s p u_s^* ) \|_2 
+ \| \Phi (u_s p u_s^* ) - \Phi (u_s )\Phi (p)\Phi (u_s^* ) \|_2 \\
&\hspace*{10mm} \ + \| (\Phi (u_s ) - \sigma_{\Phi ,s} )\Phi (p)\Phi (u_s^* ) \|_2 
+ \| \sigma_{\Phi ,s} (\Phi (p) - \varphi_\Phi (p))\Phi (u_s^* ) \|_2 \\
&\hspace*{10mm} \ + \| \sigma_{\Phi ,s} \varphi_\Phi (p) (\Phi (u_s^* ) - \Phi (u_s )^* ) \|_2 
+ \| \sigma_{\Phi ,s} \varphi_\Phi (p) (\Phi (u_s ) - \sigma_{\Phi ,s} )^* \|_2 \\
&< \frac{\delta}{3} + \delta' + 4\delta' + \frac{\delta}{3} + 3\delta' + 4\delta' < \delta .
\end{align*}
Thus $(\sigma_\Phi ,\varphi_\Phi )\in\HA (F,\cP ,n,\delta ,d)$.

It is clear from the above construction of $\sigma_\Phi$ and $\varphi_\Phi$ for each 
$\Phi\in\SA (F\cup\cP ,m,\delta' ,d)$
that we can find a small enough $\eps > 0$ not depending on $d$ with $\eps\to 0$ 
as $\delta\to 0$ such that 
for any $\Phi , \Psi\in\SA (F\cup\cP ,m,\delta' ,d)$ satisfying $\rho_{E\cup\cQ} (\Phi , \Psi ) \geq \varepsilon$ 
the pairs $(\sigma_\Phi |_E ,\varphi_\Phi |_{\cQ} )$ and $(\sigma_\Psi |_E ,\varphi_\Psi |_{\cQ} )$ are distinct.
Therefore
\[
d^{-\kappa d} |\SA (F\cup\cP ,m,\delta' ,d)|_{E\cup\cQ} \leq |\HA (F,\cP ,n,\delta ,d)|_{E,\cQ}
\]
for some $\kappa > 0$ with $\kappa\to 0$ as $\varepsilon\to 0$, by Lemma~\ref{L-nbhd gd}.
Letting $\delta\to 0$ we obtain
$s_{E\cup\cQ} (F\cup\cP) \leq s_{E,\cQ} (F,\cP )$ and $\slower_{E\cup\cQ} (F\cup\cP) \leq \slower_{E,\cQ} (F,\cP )$,
yielding the proposition.
\end{proof}

The next result is a consequence of Theorem~\ref{T-generating}, Proposition~\ref{P-pou}, and 
Proposition~\ref{P-actions local}. 
Note that $L^\infty (X,\mu )$ can be written as the $L^2$ closure of a increasing sequence
of finite-dimensional unital $^*$-subalgebras, 
and the set of nonzero projections in the union of such a sequence is dynamically generating.

\begin{proposition}\label{P-actions}
Let $G\curvearrowright (X,\mu )$ be a measure-preserving action. Let $\Omega$ be a generating subset of $G$
and $M$ a dynamically generating $^*$-subalgebra of $L^\infty (X,\mu )$. Then
\begin{align*}
s(G,X) &=  \sup_E \sup_\cQ \inf_F \inf_\cP s_{E,\cQ} (F,\cP ) ,\\
\slower (G,X) &= \sup_E \sup_\cQ \inf_F \inf_\cP \slower_{E,\cQ} (F,\cP )
\end{align*}
where in both lines $E$ and $F$ run over the finite subsets of $\Omega$ and $\cP$ and $\cQ$
run over the finite partitions of unity in $M$ consisting of projections. 
In particular, if $F$ is a finite generating subset of $G$
and $\cP$ a dynamically generating finite partition of unity in $L^\infty (X,\mu )$ consisting of projections then
\begin{align*}
s(G,X) &= s_{F,\cP} (F, \cP ) ,\\
\slower (G,X) &= \slower_{F,\cP} (F, \cP ) .
\end{align*}
\end{proposition}

\begin{proposition}\label{P-action upper}
Let $G\curvearrowright (X,\mu )$ be a measure-preserving action.
Then $s(G,X) \leq s(G)$ and $\slower (G,X) \leq \slower (G)$.
\end{proposition}

\begin{proof}
Let $E$ and $F$ be finite subsets of $G$ and $\cP$ and $\cQ$ finite partitions of unity in 
$L^\infty (X,\mu )$ consisitng of projections.
Let $n\in\Nb$, and $\delta > 0$. Let $d\in\Nb$. 
The number of restrictions $\varphi |_\cQ$ where $\varphi$ is a unital homomorphism from $\spn (\cP_{F_n} )$
to $\Cb^d$ is at most $|\cQ |^d$. Therefore
\begin{align*}
|\HA (F,\cP ,n,\delta ,d)|_{E,\cQ} \leq |\cQ |^d |\SA (F,n,\delta ,d)|_E ,
\end{align*}
from which we deduce that $s_{E,\cQ} (F,\cP ) \leq s_E (F)$ and $\slower_{E,\cQ} (F,\cP ) \leq \slower_E (F)$.
Now apply Proposition~\ref{P-actions} to obtain the result.
\end{proof}

\begin{theorem}
Let $(Y,\nu )$ be a probability space with $Y$ finite and let 
$G\curvearrowright (X,\mu ) = (Y,\nu )^G$ be the Bernoulli action. Then $s(G,X) = s(G)$
and $\slower (G,X) = \slower (G)$.
\end{theorem}

\begin{proof}
By Proposition~\ref{P-action upper} it suffices to show that $s(G,X) \geq s(G)$ and $\slower (G,X) \geq \slower (G)$.
This is a consequence of Section~8 of \cite{Bow10}, which shows that every sufficiently good
sofic approximation for $G$ is compatible with a suitable sofic approximation for the action.
\end{proof}

\begin{proposition}\label{P-action lower}
Let $G\curvearrowright (X,\mu )$ be a measure-preserving action.
Then either $\slower (G,X) \geq 1 - |G|^{-1}$ or $s(G,X) = -\infty$. 
\end{proposition}

\begin{proof}
If $s(G,X) \neq -\infty$ then the groupoid associated to the action is sofic, and an argument
as in the proof of Proposition~\ref{P-inf gd} shows that $\slower (G,X) \geq 1 - |G|^{-1}$.
\end{proof}

\begin{theorem}\label{T-amen action}
Suppose that $G$ is amenable. Let $G\curvearrowright (X,\mu )$ be a 
measure-preserving action.
Then $s(G,X) = \slower (G,X) = 1 - |G|^{-1}$. 
\end{theorem}

\begin{proof}
In view of Propositions~\ref{P-action upper} and \ref{P-action lower} and Theorem~\ref{T-amenable group},
it suffices to show that $s(G,X) \neq -\infty$. But this follows for example from Theorem~6.8 of \cite{KerLi10a}.
\end{proof}

\section{Actions of free products}\label{S-free prod actions}

In this final section we derive a free product formula for actions. In the case of
free actions a more general formula is established in \cite{DykKerPic11} using an
equivalence relation approach.
We follow the notational conventions of the previous section. 
Also, we will write $s_{E,\cP} (G,X)$ and $\slower_{E,\cP} (G,X)$ to mean $s_{E\cup\cP} (I_{G,X} )$ 
and $\slower_{E\cup\cP} (I_{G,X} )$, respectively, where as before
$I_{G,X}$ is the p.m.p.\ groupoid associated to the action $G\curvearrowright (X,\mu )$.

\begin{lemma}\label{L-amalgamated upper actions}
Let $G_1$ and $G_2$ be countable discrete groups. Let $\alpha$ be a
measure-preserving action of $G_1 * G_2$ on $(X,\mu )$. 
Then
\[
s(G_1 * G_2 ,X) \leq s(G_1 ,X) + s(G_2 ,X) .
\]
\end{lemma}

\begin{proof}
Let $\kappa > 0$. Since $G_1 \cup G_2$ generates $G_1 * G_2$, by Theorem~\ref{T-generating} there are nonempty
finite sets $E_1 \subseteq G_1$ and $E_2 \subseteq G_2$ and a finite set $\cQ$
of projections in $L^\infty (X,\mu )$ such that 
$s(G_1 * G_2 ,X) \leq s_{E_1 \cup E_2 \cup\cQ} (G_1 * G_2 ,X) + \kappa$.
Take nonempty finite sets $F_1 \subseteq I_{G_1 ,X}$ and $F_2 \subseteq I_{G_2 ,X}$ such that
$s_{E_1 \cup\cQ} (F_1 ) \leq s(G_1 ,X) + \kappa$
and $s_{E_2 \cup\cQ} (F_2 ) \leq s(G_2 ,X) + \kappa$.
Given $d,n\in\Nb$ and $\delta > 0$, for every element $\varphi\in\SA (F_1 \cup F_2 , n,\delta ,d)$
we have $\varphi |_{[I_{G_1 ,X} ]} \in\SA (F_1 , n,\delta ,d)$
and $\varphi |_{[I_{G_2 ,X} ]} \in\SA (F_2 , n,\delta ,d)$. Hence
\[
|\SA (F_1 \cup F_2 , n,\delta ,d)|_{E_1 \cup E_2 \cup\cQ} 
\leq |\SA (F_1 , n,\delta ,d)|_{E_1 \cup\cQ} |\SA (F_2 , n,\delta ,d)|_{E_2 \cup\cQ} 
\]
and so
\begin{align*}
s(G_1 * G_2 ,X) &\leq s_{E_1 \cup E_2 \cup\cQ} (F_1 \cup F_2 ) + \kappa \\
&\leq s_{E_1 \cup\cQ} (F_1 ) + s_{E_2 \cup\cQ} (F_2 ) + \kappa \\
&\leq s(G_1 ,X) + s(G_2 ,X) + 3\kappa .
\end{align*}
Since $\kappa$ was an arbitrary positive number we obtain the lemma.
\end{proof}

The proof of the following lemma is similar to that of Lemma~\ref{L-mult}.

\begin{lemma}\label{L-multiple}
Let $\cP$ be a finite partition of unity in $L^\infty (X,\mu )$ consisting of projections. 
Let $\cQ$ be a finite set of projections in $L^\infty (X,\mu )$.
Let $E$ and $F$ be finite subsets of $G$, $n\in\Nb$, and $\delta > 0$. Let $\ell\in\Nb$. Then 
\[
\slower_{E,\cQ} (F,\cP ,n) = \inf_{\delta > 0} 
\liminf_{d\to\infty} \frac{1}{\ell d \log (\ell d)} \log |\HA (F,\cP ,n ,\delta ,\ell d)|_{E,\cQ} .
\]
\end{lemma}

Recall that $S_d$ acts on the set of maps 
$\sigma : G\to S_d$ by $(U\cdot\sigma )_s = U\sigma_s U^{-1}$. Also, given a unital $^*$-subalgebra 
$M\subseteq L^\infty (X,\mu )$, $S_d$ acts on the set of unital homomophisms 
$\varphi : M\to \diag (M_d ) \cong \Cb^d$ by $(U\cdot\varphi )(f) = U\varphi (f)$. 
Thus we have an action $S_d$ on the set of pairs $(\sigma ,\varphi )$ consisting of such $\sigma$ and
$\varphi$.

Recall also that for sets $A_1 \subseteq A_2$, $B_1 \subseteq B_2$, $Z_1$, and $Z_2$ and a collection $\sY$
of ordered pairs consisting of maps $A_1 \to Z_1$ and $A_2 \to Z_2$
we write $|\sY |_{B_1 , B_2}$ for the cardinality of the set of pairs 
$(\sigma |_{B_1} , \varphi |_{B_2} )$ where $(\sigma , \varphi )\in\sY$.

\begin{lemma}\label{L-free}
Let $G_1$ and $G_2$ be countable discrete groups and let $\alpha$ be a
measure-preserving action of $G_1 * G_2$ on $(X,\mu )$. Then
\[
\slower (G_1 * G_2 ,X) \geq \slower (G_1 ,X) + \slower (G_2 ,X) .
\]
\end{lemma}

\begin{proof}
Let $\eta > 0$. Then by Theorem~\ref{T-generating}, Proposition~\ref{P-actions local}, and
Proposition~\ref{P-pou} there exist finite sets 
$E_1 \subseteq G_1$ and $E_2 \subseteq G_2$ and finite sets of projections
$\cQ_1 , \cQ_2 \subseteq L^\infty (X,\mu )$ such that
$\slower_{E_1 ,\cQ_1} (G_1 ,X) \geq \slower (G_1 ,X) - \eta$ and
$\slower_{E_2 ,\cQ_2} (G_2 ,X) \geq \slower (G_2 ,X) - \eta$. 
Write $\cR$ for the set of all projections in $L^\infty (X,\mu )$. Take 
finite symmetric sets $F_1 \subseteq G_1$ and $F_2 \subseteq G_2$ containing $e$ and a 
set $\cP$ consisting of the nonzero projections of some finite-dimensional unital $^*$-subalgebra of
$L^\infty (X,\mu )$ containing $\cQ_1 \cup \cQ_2$ such that
$\slower_{E_1 \cup E_2 ,\cQ_1 \cup \cQ_2} (G_1 \cup G_2 ,\cR ) 
\geq \slower_{E_1 \cup E_2 ,\cQ_1 \cup \cQ_2} (F_1 \cup F_2 ,\cP ) - \eta$.

Let $\delta > 0$ and $n\in\Nb$. Set $K=(F_1 \cup F_2 )^n$.
Let $0 < \delta' < \delta /(7n)$, 
to be further specified. Let $\kappa > 0$, to be further specified.

Fix an $\ell\in\Nb$ such that for every $p\in\cP_K$ we can find a $b_p \in\Nb$
such that $| \mu (p) - b_p /\ell | < \kappa$.
Let $0<\delta'' <\delta'$, to be further specified as a function of $\ell$. 
Let $d\in\Nb$. 
For $i=1,2$ set $\sY_i = \HA (F_i^{2n} ,\cP_K ,n,\delta'' ,\ell d)$ for brevity.

Fix a $(\sigma ,\varphi )\in \HA (F_1^{2n} ,\cP_K ,n,\delta'' ,\ell d)$. For every map $\omega : G_2 \to S_d$
we construct, using freeness, a map $\Omega = \Omega_\omega : G_1 * G_2 \to S_d$ such that
for a reduced word $t_1 \cdots t_k$ where the $t_i$ alternate membership in $F_1^n$ and $F_2^n$ we have
$\Omega_{t_1 \cdots t_n} = \rho_{1,t_1}\rho_{2,t_2}\cdots\rho_{n,t_n}$ 
where $\rho_i = \sigma$ if $s_i \in F_1^n$ and $\rho_i = \omega$ otherwise. 

A simple perturbation argument shows that if $\kappa$ is small enough as a function of $\delta'$ 
then for sufficiently large $d$ we can fix an identification of $M_{\ell d}$ with $M_\ell \otimes M_d$ such that matrix units 
pair with tensor products of matrix units and for every $p\in\cP_K$ there is a diagonal projection $D \in M_\ell$
such that $\| \varphi (p) - D \otimes 1 \|_2 < \delta'$. Writing $E_{i,j}$ for the standard matrix units in $M_d$,
we denote by $\sX_d$ the set of all permutation matrices in $M_{\ell d}$ of the form 
$\sum_{i=1}^\ell E_{i,i} \otimes U_i \in M_\ell \otimes M_d$. Note that for every $p\in\cP_K$ we have,
taking a projection $D\in M_\ell$ such that $\| \varphi (p) - D\otimes 1 \|_2 < \delta'$,
\begin{align*}\tag{$\ast$}
\| U\varphi (p) - \varphi (p)U \|_2 \leq \| U(\varphi (p) - D\otimes 1) \|_2 + \| (D\otimes 1 - \varphi (p))U \|_2 < 2\delta' .
\end{align*}

Let $(\omega ,\psi )\in\sY_2$.  By the same type of perturbation argument 
alluded to in the previous paragraph, if we assume $\kappa$ to be sufficiently small as a function of $\delta'$ 
then we can find a $W\in S_{\ell d}$ such that the pair
$(\omega' ,\psi' ) = W\cdot (\omega , \psi )$ satisfies $\| \psi' (p) - \varphi (p) \|_2 < \delta'$ for all $p\in\cP_K$.
Write $\Upsilon_{d,\sigma ,\varphi ,\omega ,\psi ,W}$ for the set of all $U\in\sX_d$ such that 
for every $k=1,\dots ,n$ the map $\Omega = \Omega_{U\cdot\omega'}$ satisfies 
$\tr_{\ell d} (\Omega_{t_1 \cdots t_k} ) < \delta'$ for all reduced words 
$t_1 \cdots t_k \neq e$ where the $t_j$ alternate membership in $F_1^n$ and $F_2^n$. 

Now given any $U\in \sX_d$ and $A\in S_{\ell d}$, if we view these as elements of $M_\ell \otimes M_d$ 
and write $U = \sum_{i=1}^\ell E_{i,i} \otimes U_i$ and $A = \sum_{i,j=1}^\ell E_{i,j} \otimes A_{i,j}$
then $UAU^* = \sum_{i,j=1}^\ell E_{i,j} \otimes U_i A_{i,j} U_j^*$.
Thus, by multiple applications of Lemma~\ref{L-single}, whenever $d$ is
large enough we can find a $V \in \sX_d$ such that if
$A$ is equal to $\sigma_s$ for some $s\in F_1^{2n} \setminus \{ e \}$ or to $\omega'_s$ for some $s\in F_2^{2n} \setminus \{ e \}$ then writing
$V A V^* =  \sum_{i,j=1}^\ell E_{i,j} \otimes A_{i,j}$ the quantity $\tr ( A_{i,j})$ is smaller than a prescribed positive value
for all distinct $i,j\in \{1,\dots,\ell \}$, and we can also ensure 
that $\tr ( A_{i,i})$ is smaller than the same prescribed positive value for all $i=1,\dots,n$
by assuming $\delta''$ to be small enough as a function of $\ell$. Consider a product of the form
\begin{align*}\tag{$\ast\ast$}
VA_1 V^* (U(V A_2 V^* )U^* )\cdots VA_{2r-1} V^*(U(V A_{2r} V^* )U^* )
\end{align*}
for $1\leq r\leq n/2$, $U\in\sX_d$, and each $A_1 , \dots , A_{2r}$ equal to
$\sigma_s$ for some $s\in F_1 \setminus \{ e \}$ or to $\omega'_s$ for some $s\in F_2 \setminus \{ e \}$. Expressing
$U$ as $\sum_{i=1}^\ell E_{i,i} \otimes U_i$ and each $VA_k V^*$ as a sum of the form
$\sum_{i,j=1}^\ell E_{i,j} \otimes A_{i,j}$, we expand the product ($\ast\ast$) to obtain
a sum of terms of the form $E_{i,j} \otimes B$ and apply
Lemma~\ref{L-conc} to the second tensor product factor of each of these terms
to deduce, assuming $\delta''$ is small enough, that $\Upsilon_{d,\sigma ,\varphi ,\omega ,\psi ,W}$
contains enough elements of the form $V^* UV$ as $d\to\infty$ so that $\lim_{d\to\infty} |\Upsilon_{d,\sigma ,\varphi ,\omega ,\psi ,W} |/|\sX_d | = 1$.
Note that although Lemma~\ref{L-conc} addresses only the case of even $k$, we can handle the odd case with the following
reduction argument.
For a product of the form 
\[
VA_1 V^* (U(V A_2 V^* )U^* )\cdots VA_{2r-1} V^* (U(V A_{2r} V^* )U^* )VA_{2r+1} V^*
\]
we write its trace as $\tr (VA_{2r+1} A_1 V^*(U(VA_2 V^* )U^* )\cdots VA_{2r-1} V^* (U(VA_{2r} V^* )U^* ))$. If
$A_{2r+1} = \sigma_{s_1}$ and $A_1 = \sigma_{s_2}$ for some $s_1$ and $s_2$ contained 
in $F_i \setminus \{ e \}$ for some $i\in \{ 1,2 \}$ with $s_1 s_2 \neq e$, then up to a perturbation
we have reduced to the even case since $s_1 s_2$ lies in the set $F_i^{2n}$ appearing in the definition of $\sY_i$.
Otherwise up to a perturbation we are back in the odd case with fewer factors, and we can repeat the procedure as necessary.

Take a $U\in\Upsilon_{d,\sigma ,\varphi ,\omega ,\psi ,W}$ and set $\omega'' = U\cdot\omega'$.
Let us show that $(\Omega ,\varphi )\in\HA (F_1 \cup F_2 ,\cP ,n,\delta ,\ell d)$ where $\Omega = \Omega_{\omega''}$.
Let $t_1 ,\dots ,t_n \in F_1 \cup F_2$. Let $j_1 = 1 < j_2 < \dots j_k \leq n$ 
be such that for each $i=1,\dots ,k$ the elements
$t_{j_i} ,\dots , t_{j_{i+1} - 1}$ either all lie in $F_1$ or all lie in $F_2$ and this common membership alternates between
$F_1$ and $F_2$ from one $i$ to the next. Writing $\rho^{(i)} = \sigma$ 
if $t_{j_i} \in F_1$ and $\rho^{(i)} = \omega''$ otherwise, we have
\begin{align*}
\| \Omega_{t_1 \cdots t_n} - \Omega_{t_1} \cdots\Omega_{t_n} \|_2 
\leq \sum_{i=1}^k \| \rho^{(i)}_{t_{j_i} \cdots t_{j_{i+1} - 1}} 
- \rho^{(i)}_{t_{j_i}} \cdots \rho^{(i)}_{t_{j_{i+1} - 1}} \|_2 
< k\delta'' \leq \delta .
\end{align*}
Also, if we are given a $t\in (F_1 \cup F_2 )^n \setminus \{ e \}$
then we can write
$t = t_1 \cdots t_k$ where $1\leq k\leq n$ and $t_1 ,\dots , t_k$ alternate membership
in $F_1^n$ and $F_2^n$, so that
$\tr_{\ell d} (\Omega_t ) = \tr_{\ell d} (\Omega_{t_1 \cdots t_k} ) < \delta' \leq\delta$.
Thus $\Omega$ is an element of $\GA (F_1  \cup F_2 ,n,\delta ,\ell d)$.

Now let us check that, to within the required tolerance, $\varphi$ is approximately equivariant 
on a reduced word $t_1 \cdots t_k$ where $1\leq k\leq n$ and the $t_j$ alternate membership in 
$F_1^n$ and $F_2^n$. Let $p\in\cP$. Given a $j\in \{1,\dots ,n\}$, the projection 
$\alpha_{t_{j+1} \cdots t_k} (p)$ lies in $\cP_K$ and hence when $t_j \in F_1^n$ we have
\begin{align*}
\| \Ad \Omega_{t_j} \circ\varphi (\alpha_{t_{j+1} \cdots t_k} (p))
- \varphi \circ\alpha_{t_j} (\alpha_{t_{j+1} \cdots t_k} (p)) \|_2
< \delta'' \leq \delta' 
\end{align*}
while in the case $t_j \in F_2^n$ we use from ($\ast$) the fact that $U$ approximately commutes 
with $\varphi (\alpha_{t_{j+1} \cdots t_k} (p))$ to within $2\delta'$ in trace norm to obtain
\begin{align*}
\lefteqn{\| \Ad \Omega_{t_j} \circ\varphi (\alpha_{t_{j+1} \cdots t_k} (p))
- \varphi \circ\alpha_{t_j} (\alpha_{t_{j+1} \cdots t_k} (p)) \|_2} \hspace*{30mm} \\
\hspace*{20mm} &= \| \Ad \omega_{t_j}' (U^* \varphi (\alpha_{t_{j+1} \cdots t_k} (p))U)
- U^* \varphi (\alpha_{t_j \cdots t_k} (p))U \|_2 \\
&\leq  \| \Ad \omega_{t_j}' (U^* \varphi (\alpha_{t_{j+1} \cdots t_k} (p))U - \varphi (\alpha_{t_{j+1} \cdots t_k} (p))) \|_2 \\
&\hspace*{10mm} \ +\| \Ad \omega_{t_j}' (\varphi (\alpha_{t_{j+1} \cdots t_k} (p))
- \psi' (\alpha_{t_{j+1} \cdots t_k} (p))) \|_2 \\
&\hspace*{10mm} \ + \| \Ad \omega_{t_j}' \circ\psi' (\alpha_{t_{j+1} \cdots t_k} (p))
- \psi' \circ\alpha_{t_j} (\alpha_{t_{j+1} \cdots t_k} (p)) \|_2 \\
&\hspace*{10mm} \ + \| \psi' (\alpha_{t_j \cdots t_k} (p))
- \varphi (\alpha_{t_j \cdots t_k} (p))  \|_2 \\
&\hspace*{10mm} \ + \| \varphi (\alpha_{t_j \cdots t_k} (p)) 
-U^* \varphi (\alpha_{t_j \cdots t_k} (p))U \|_2 \\
&< 2\delta' + \delta' + \delta'' + \delta' + 2\delta' \leq 7\delta' .
\end{align*}
It follows that
\begin{align*}
\lefteqn{\| \Ad \Omega_{t_1 \cdots t_k} \circ\varphi (p) - \varphi\circ\alpha_{t_1 \cdots t_k} (p) \|_2}\hspace*{10mm} \\ 
\hspace*{10mm} &= \| \Ad\Omega_{t_1} \circ\cdots\circ \Ad\Omega_{t_k} \circ\varphi (p) 
- \varphi\circ\alpha_{t_1} \circ\cdots\circ\alpha_{t_k} (p) \|_2 \\
&\leq \sum_{j=1}^k \| \Ad\Omega_{t_1} \circ\cdots\circ\Ad\Omega_{t_{j-1}} 
(\Ad\Omega_{t_j} \circ\varphi (\alpha_{t_{j+1} \cdots t_k} (p))
- \varphi\circ\alpha_{t_j} (\alpha_{t_{j+1} \cdots t_k} (p))) \|_2 \\
&< 7n\delta' < \delta  .
\end{align*}
Since $|\tr_{\ell d} \circ\varphi (p) - \mu (p)| = \delta'' < \delta$ for all $p\in\cP_{(F_1 \cup F_2 )^n}$
by virtue of the fact that $(\sigma ,\varphi )\in\sY_1$, 
we thus conclude that $(\Omega ,\varphi )\in\HA (F_1 \cup F_2 ,\cP ,n,\delta ,\ell d)$, as desired.

Note that $\Omega$ was obtained by combining in a free manner the maps $\sigma$ and $\omega''$, 
where the latter was obtained from $\omega$ by conjugating by $W$ and then by $U$. Let $\gamma > 0$. 
The set $\Lambda_{d,\sigma ,\varphi, \omega ,\psi}$ of all products $UW$ such that $U$ and $W$
together do the required job has cardinality at least $|\sX_d |/2$
for all $d$ larger than some $d_0$ not depending on $(\sigma,\varphi)$ or $(\omega,\psi)$.
Hence, by Stirling's approximation, for all sufficiently large $d$ we have
\begin{align*}
|\Lambda_{d,\sigma ,\varphi ,\omega ,\psi} | \
\geq \frac12 |\sX_d | = \frac12 d!^{\ell} \geq d^{\ell d(1-\gamma )} .
\end{align*}
Writing $\cS$ for the set of all $U\in S_{\ell d}$
such that $(U\cdot\omega )|_{E_2} = \omega |_{E_2}$ and $(U\cdot\psi )|_{\cQ_2} = \psi |_{\cQ_2}$, we then have
\begin{align*}
|S_{\ell d} \cdot (\omega ,\psi )|_{E_2 ,\cQ_2} = \frac{|S_{\ell d} |}{|\cS |} 
&\leq \frac{(\ell d)^{\ell d}}{|\cS |} \cdot 
\frac{|\Lambda_{d,\sigma ,\varphi ,\omega, \psi}|}{d^{\ell d(1-\gamma )}} \\
&\leq \ell^{\ell d} d^{\ell d\gamma} \frac{|\Lambda_{d,\sigma , \varphi ,\omega ,\psi} |}{|\cS |} \\
&\leq \ell^{\ell d} d^{\ell d\gamma}  |\Lambda_{d,\sigma , \varphi ,\omega ,\psi} \cdot (\omega ,\psi )|_{E_2 ,\cQ_2} .
\end{align*}
Take a set $\sR$ of representatives for the action of $S_{\ell d}$ on $\sY_2$
modulo the relation under which two pairs
are equivalent if the first coordinates agree on $E_2$ and the second coordinates agree on $\cQ_2$,
and take an $\sR' \subseteq\sR$ which is a set of representatives for the action of $S_{\ell d}$ on $\sY_2$
modulo the relation under which two pairs are equivalent if their first coordinates agree on $E_2$.
Note that $|\sY_2 |_{E_2 ,\cQ_2} = \sum_{(\omega ,\psi )\in\sR} |S_{\ell d} \cdot (\omega ,\psi )|_{E_2 ,\cQ_2}$. 
Setting $m = |(\cP_K )_{(F_2^{2n} )^n} |$, for
every $\omega\in\GA (F_2^{2n},n,\delta'' ,\ell d)$ there are at most $m^{\ell d}$ many homomorphisms $\psi$ such that 
$(\omega ,\psi )\in\sY_2$, so that for every $(\omega ,\psi )\in\sY_2$ we have
\[
|\Lambda_{d,\sigma , \varphi ,\omega ,\psi} \cdot\omega |_{E_2} \geq 
m^{-\ell d} |\Lambda_{d,\sigma , \varphi ,\omega ,\psi} \cdot (\omega ,\psi )|_{E_2 ,\cQ_2} 
\]
and every pair in $\sR'$ has the same first coordinate modulo agreement on $E_2$ as at most $m^{\ell d}$ many pairs in $\sR$.
Assuming that each pair $(\omega,\psi)$ in $\sR'$ was chosen so as to maximize the quantity
$|\Lambda_{d,\sigma , \varphi ,\omega ,\psi} \cdot (\omega ,\psi )|_{E_2 ,\cQ_2}$ over all pairs
in $\sR$ which share the same first coordinate modulo agreement on $E_2$, we then have
\begin{align*}
 \sum_{(\omega ,\psi )\in\sR'} 
|\Lambda_{d,\sigma , \varphi ,\omega ,\psi} \cdot\omega |_{E_2} 
&\geq \sum_{(\omega ,\psi )\in\sR'}
m^{-\ell d} |\Lambda_{d,\sigma , \varphi ,\omega ,\psi} \cdot (\omega ,\psi )|_{E_2 ,\cQ_2} \\
&\geq m^{-2\ell d} \sum_{(\omega ,\psi )\in\sR}
|\Lambda_{d,\sigma , \varphi ,\omega ,\psi} \cdot (\omega ,\psi )|_{E_2 ,\cQ_2} .
\end{align*}
Taking a set $\sY_1'$ of representatives for the relation on $\sY_1$ under which two pairs
are equivalent if the first coordinates agree on $E_1$ and the second agree on $\cQ_1$, we thus obtain
\begin{align*}
\lefteqn{|\HA (F_1 \cup F_2 ,\cP ,n,\delta ,\ell d)|_{E_1 \cup E_2 ,\cQ_1 \cup \cQ_2}}\hspace*{15mm} \\
\hspace*{10mm} &\geq \sum_{(\sigma ,\varphi )\in\sY_1'} \sum_{(\omega ,\psi )\in\sR'} 
|\Lambda_{d,\sigma , \varphi ,\omega ,\psi} \cdot\omega |_{E_2} \\
&\geq m^{-2\ell d} \sum_{(\sigma ,\varphi )\in\sY_1'} \sum_{(\omega ,\psi )\in\sR}
|\Lambda_{d,\sigma , \varphi ,\omega ,\psi} \cdot (\omega ,\psi )|_{E_2 ,\cQ_2} \\
&\geq (m^2\ell )^{-\ell d} d^{-\ell d\gamma} |\sY_1 |_{E_1 ,\cQ_1} 
\sum_{(\omega ,\psi )\in\sR} |S_{\ell d} \cdot (\omega ,\psi )|_{E_2 ,\cQ_2} \\
&= (m^2\ell )^{-\ell d} d^{-\ell d\gamma} |\sY_1 |_{E_1 ,\cQ_1} |\sY_2 |_{E_2 ,\cQ_2} \\
&= (m^2\ell )^{-\ell d} d^{-\ell d\gamma} 
|\HA (F_1^{2n} ,\cP_K ,n,\delta'' ,\ell d)|_{E_1 ,\cQ_1} \\
&\hspace*{50mm} \ \times |\HA (F_2^{2n} , \cP_K ,n ,\delta'' ,\ell d)|_{E_2 ,\cQ_2}
\end{align*}
and hence, in view of Lemma~\ref{L-multiple},
\begin{align*}
\lefteqn{\liminf_{d\to\infty} \frac{1}{\ell d\log (\ell d)} 
\log |\HA (F_1 \cup F_2 ,\cP ,n,\delta ,\ell d)|_{E_1 \cup E_2 ,\cQ_1 \cup \cQ_2}}\hspace*{20mm} \\
\hspace*{20mm} &\geq \liminf_{d\to\infty} \frac{1}{\ell d\log (\ell d)} 
\log |\HA (F_1^{2n} ,\cP_K ,n,\delta'' ,\ell d)|_{E_1 , \cQ_1} \\
&\hspace*{10mm} \ + \liminf_{d\to\infty} 
\frac{1}{\ell d\log (\ell d)} \log |\HA (F_2^{2n} ,\cP_K ,n,\delta'' ,\ell d)|_{E_2 ,\cQ_2} - \gamma \\
&\geq \slower_{E_1 ,\cQ_1} (F_1^{2n} ,\cP_K ) + \slower_{E_2 ,\cQ_2} (F_2^{2n} ,\cP_K ) - \gamma .
\end{align*}
Since $n$ was an arbitary positive integer and $\delta$ and $\gamma$ arbitrary positive numbers, 
it follows that 
\[
\slower_{E_1 \cup E_2 ,\cQ_1 \cup \cQ_2} (F_1 \cup F_2 ,\cP ) \geq \slower_{E_1 ,\cQ_1} (G_1 ,X) 
+ \slower_{E_2 ,\cQ_2} (G_2 ,X)
\]
and hence, using Theorem~\ref{T-generating} and Proposition~\ref{P-actions local},
\begin{align*}
\slower (G_1 * G_2 ,X) 
&= \slower (G_1 \cup G_2 ,\cR ) \\ 
&\geq \slower_{E_1 \cup E_2 ,\cQ_1 \cup \cQ_2} (G_1 \cup G_2 , \cR  ) \\
&\geq \slower_{E_1 \cup E_2 ,\cQ_1 \cup \cQ_2} (F_1 \cup F_2 ,\cP ) - \eta \\
&\geq \slower_{E_1 ,\cQ_1} (G_1 ,X) + \slower_{E_2 ,\cQ_2} (G_2 ,X) - \eta \\
&\geq \slower (G_1 ,X) + \slower (G_2 ,X) - 3\eta .
\end{align*}
Since $\eta$ was an arbitrary positive number this yields the result.
\end{proof}

Combining Lemmas~\ref{L-amalgamated upper actions} and \ref{L-free} yields the following.

\begin{theorem}\label{T-free prod action}
Let $G_1$ and $G_2$ be countable discrete groups and
$\alpha$ be a measure-preserving action of $G_1 * G_2$ on $(X,\mu )$ whose restrictions
to $G_1$ and $G_2$ are approximation regular. 
Then $\alpha$ is approximation regular and
\[
s(G_1 * G_2 ,X) = s(G_1 ,X) + s(G_2 ,X) .
\]
\end{theorem}

\begin{corollary}\label{C-free action}
Let $r\in\Nb$ and let $F_r \curvearrowright (X,\mu )$ be a measure-preserving action.
Then $s(F_r ,X) = \slower (F_r ,X) = r$. 
\end{corollary}

\begin{proof}
Repeatedly apply Theorem~\ref{T-free prod action} using the fact that the action 
$\Zb\curvearrowright (X,\mu )$ obtained by restricting to any one of the standard generators of $F_r$
satisfies $s(\Zb ,X) = \slower (\Zb ,X) = 1$ by Theorem~\ref{T-amen action}.
\end{proof}

The above corollary implies that, for distinct $r_1 , r_2 \in\Nb$, given for each $i=1,2$ 
a measure-preserving action $F_{r_i} \curvearrowright (X,\mu )$, the associated groupoids 
are nonisomorphic. From this we recover both the fact that $F_{r_1}$ and $F_{r_2}$ are
not isomorphic when $r_1 \neq r_2$ and Gaboriau's result that for $r_1 \neq r_2$
there are no free ergodic measure-preserving actions 
$F_{r_1} \curvearrowright (X,\mu )$ and $F_{r_2} \curvearrowright (X,\mu )$
which are orbit equivalent \cite{Gab00}.

By combining the techniques of this section with the quasitiling arguments of 
Section~\ref{S-amalg groups} one could likely generalize the formula of Theorem~\ref{T-free prod action} to
allow for amalgamation over a common amenable subgroup on which the action is free.
We have refrained from attempting this given that the technical details appear formidable
and the equivalence relation approach of \cite{DykKerPic11} already gives the desired formula
under the hypothesis that the action of the amalgamated free product is free. Ultimately one would like to have
a general groupoid version of the free product formula in this amalgamated setting that would 
specialize to actions without any freeness assumptions.

\end{document}